\documentclass[11pt]{amsart}

\usepackage[draft]{changes}

\definechangesauthor[color=red]{zl}
\definechangesauthor[color=blue]{hh}
\definechangesauthor[color=orange]{hx}
\newcommand{\stkout}[1]{\ifmmode\text{\sout{\ensuremath{#1}}}\else\sout{#1}\fi}
\setdeletedmarkup{\stkout{#1}}
\usepackage{amssymb} 

\usepackage{graphics}
\usepackage{graphicx}

\usepackage{amsmath} 
\usepackage{amsthm}
\usepackage{amsfonts}
\usepackage{xcolor}
\usepackage[english]{babel}
\usepackage[margin=1.0in]{geometry}
\usepackage[colorlinks=true,
        linkcolor=blue]{hyperref}
\usepackage{bbm}
\usepackage{verbatim}
\usepackage[T1]{fontenc}

\numberwithin{equation}{section}

\newtheorem{prop}{Proposition}
\newtheorem{lemma}[prop]{Lemma}

\newtheorem{thm}[prop]{Theorem}
\newtheorem{cor}[prop]{Corollary}

\numberwithin{prop}{section}

\newtheorem{defn}[prop]{Definition}
\theoremstyle{definition}

\newtheorem{rmk}[prop]{Remark}

\definecolor{c1}{rgb}{0.2,0.4,0.5}
\definecolor{c2}{rgb}{0.1,0.3,0.5}
\definecolor{c3}{rgb}{0.2,0.7,0.5}
\usepackage{tikz}

\def \k {K\"ahler }
\newcommand{\oo}[1]{\overline{#1}}

\newcommand{\del}{\partial}
\newcommand{\bdel}{\bar{\partial}}

\newcommand{\ga}{\alpha}

\newcommand{\gw}{\omega}

\newcommand{\ten}{\otimes}

\newcommand{\dbar}{\oo\partial}

\newcommand{\eps}{\varepsilon}

\DeclareMathOperator{\Real}{Re}

\begin{document}

\title[Asymptotic properties of Bergman kernels]{Asymptotic properties of Bergman kernels for potentials with Gevrey regularity }

\begin{abstract} We study the asymptotic properties of the Bergman kernels associated to tensor powers of a positive line bundle on a compact K\"ahler manifold. We show that if the K\"ahler potential is in Gevrey class $G^a$ for some $a>1$, then the Bergman kernel accepts a complete asymptotic expansion in a neighborhood of the diagonal of shrinking size $k^{-\frac12+\frac{1}{4a+4\eps}}$ for every $\eps>0$. These improve the earlier results in the subject for smooth potentials, where an expansion exists in a $(\frac{\log k}{k})^{\frac12}$ neighborhood of the diagonal. We obtain our results by finding upper bounds of the form $C^m m!^{2a+2\eps}$ for the Bergman coefficients $b_m(x, \bar y)$ in a fixed neighborhood by the method of \cite{BBS}. We also show that sharpening these upper bounds would improve the rate of shrinking neighborhoods of the diagonal $x=y$ in our results. 
\end{abstract}

%\date{\today}

%\author [Hezari]{Hamid Hezari}
%\address{Department of Mathematics, UC Irvine, Irvine, CA 92617, %USA} \email{\href{mailto:hezari@uci.edu}{hezari@uci.edu}}

%\author[Lu]{Zhiqin Lu}
%\address{Department of Mathematics, UC Irvine, Irvine, CA 92617, USA}
%\email{\href{mailto:zlu@uci.edu}{zlu@uci.edu}}

\author [Xu]{Hang Xu}
\address{Department of Mathematics, Johns Hopkins University, Baltimore, MA 21218, USA}
\email{\href{mailto:hxu@math.jhu.edu}{hxu@math.jhu.edu}}

\maketitle

\section{Introduction}

Let $(L,h) \to M$ be a positive Hermitian holomorphic line bundle over a compact complex manifold of dimension $n$. The metric $h$ induces the \k form $\gw= -\tfrac{\sqrt{-1}}{2} \del \bdel \log(h)$ on $M$.  For $k$ in $\mathbb N$, let $H^0(M,L^k)$ denote the space of holomorphic sections of $L^k$. The {Bergman projection} is the orthogonal projection $\Pi_k: {L}^{2}(M,L^k) \to H^0(M,L^k)$ with respect to the natural inner product induced by the metric $h^k$ and the volume form $\frac{ \gw^n }{n!}$. The \emph{Bergman kernel} $K_k$, a section of $L^k\ten \bar{L}^k$, is the distribution kernel of $\Pi_k$.
Given $p \in M$,  let $(V, e_L)$ be a local trivialization of $L$ near $p$.   We write $| e_L |^2_{h}=e^{-\phi}$ and call $\phi$ a local \k potential. In the frame $ e_L^{k} \ten {\bar{e}_L^{k}}$, the Bergman kernel $K_k(x,y)$ is understood as a function on $V \times V$. We note that on the diagonal $x=y$, the function $K_k(x,x)e^{-k\phi(x)}$ is independent of the choice of the local frame, hence it is a globally defined function on $M$  called the \emph{Bergman function}, which is also equal to $| K_k(x, x)  |_{h^k}$. 

Zelditch \cite{Ze1} and Catlin \cite{Ca} proved that on the diagonal $x=y$, the Bergman kernel accepts a complete asymptotic expansion of the form
\begin{equation}\label{ZC}
K_k(x,x)e^{-k\phi(x)}\sim\frac{k^n}{\pi^n}\left(b_0(x,\bar{x})+\frac{b_1(x,\bar{x})}{k}+\frac{b_2(x,\bar{x})}{k^2}+\cdots\right).
\end{equation}
\textit{Near the diagonal}, i.e. in a $\sqrt{\frac{\log k}{k}}$-neighborhood of the diagonal, one has a scaling asymptotic expansion for the Bergman kernel (see \cite{ShZe, MaMaBook, MaMaOff, LuSh, HKSX}).  For $d(x, y) \gg \sqrt{\frac{\log k}{{k}}}$, where $d$ is the Riemannian distance induced by $\omega$, no useful asymptotics are known for smooth metrics. However,  there are off-diagonal upper bounds of Agmon type 
\begin{equation}\label{Agmon} \left | K_k(x, y) \right |_{h^k}  \leq C k^n e^{- c \sqrt{k} d(x, y)}, \end{equation}
proved for smooth metrics in \cite{Ch1, Del, Lindholm, Bern,MaMaAgmon}. In fact as shown in \cite{Ch3, HX}, one has better decay estimates. More precisely, there exist positive constants $c, C$ and a function $f(k) \to \infty$ as $k \to \infty$ such that
\begin{equation*}
|K_k(x,y)|_{h^k} \leq 
\begin{cases}
C k^{n} e^{-c \, k d(x,y)^2}, \quad & d(x,y)\leq f(k)\sqrt{\frac{\log k}{k}},\\
C k^{n} e^{-c \, f(k)\sqrt{k\log k} \,\, d(x, y)}, \quad & d(x,y) \geq f(k)\sqrt{\frac{\log k}{k}}.
\end{cases}
\end{equation*} 
A quantitative version of the above
estimate that relates the growth rate of $f(k)$ to the growth rate of the derivatives of the metric $h$ is obtained in \cite{HX}. In particular, when $h$ is in the Gevrey $a$ $(a\geq 1)$ class, we get $f(k)= \frac{k^{\frac{1}{4a-2}}}{\sqrt{\log k}}$.

This article generalizes the results in \cite{HLXanalytic} to the setting of Gevrey classes. To be precise, we prove an asymptotic expansion in a $k^{-\frac{1}{2}+\frac{1}{4a+4\eps}}$ neighborhood of the diagonal  for any $\eps>0$ if the metric $h$ is in the Gevrey $a$ $(a>1)$ class. In particular, we show that in the Gevrey $a$ class, uniformly for all sequences $x_k$ and $y_k$ with $d(x_k, y_k) \leq k^{-\frac{1}{2}+\frac{1}{4a+4\eps}}$, we have
$$  \left | K_k(x_k, y_k) \right |_{h^k} \sim \frac{k^n}{\pi^n} e^{- \frac{k D(x_k, y_k)}{2}}, \quad \text{as} \quad k \to \infty,$$ where $D(x, y)$ is Calabi's diastasis function \eqref{Diastatis}, which is controlled from above and below by $d^2(x, y)$.  
Before we state the results we must also mention that in \cite{BBS}, there is an off-diagonal asymptotic expansion for the Bergman kernel of the form 
\begin{equation}\label{ill}
K_k(x,y)=e^{k\psi(x,\bar y)}\frac{k^n}{\pi^n}\left (1+\sum _{j=1}^{ {N-1}}\frac{b_j(x, \bar y)}{k^j}\right )+e^{k\left(\frac{\phi(x)}{2}+\frac{\phi(y)}{2}\right)}k^{-N+ n}O_N(1),
\end{equation} 
which holds for all $d(x, y) \leq \delta$ for some $\delta >0$. Here, $\psi(x, \bar y)$ and $b_j(x, \bar y)$ are almost holomorphic extensions of $\phi(x)$ and $b_j(x, \bar x)$ from \eqref{ZC}. However, note that this expansion is only useful when the term $e^{k\left(\frac{\phi(x)}{2}+\frac{\phi(y)}{2}\right)}k^{-N+ n}$ is a true remainder term, i.e. it is less than the principal term $k^n e^{k \psi(x,\bar y)}$ in size, which holds only in a neighborhodd $d(x, y) \leq {C} \sqrt{\frac{\log k}{{k}}}$ in general. In the case that $h$ is real analytic, this is valid in a larger neighborhood $d(x,y)\leq k^{-1/4}$ \cite{HLXanalytic}. In a recent preprint \cite{RSN}, this is further improved to a fixed neighborhood independent of $k$. 

We now state our main result and its corollaries. 
\begin{thm}\label{Main}
Assume that the local \k potential $\phi$ is in the Gevrey class $G^a(V)$ for some $a>1$, meaning that for some $C_0$ and $C_1>0$, we have
\begin{equation}
\|D_z^\alpha D_{\bar{z}}^\beta \phi(z)\|_{L^\infty (V)}\leq C_0C_1^{|\alpha|+|\beta|}(\alpha!\beta!)^a,
\hspace{12pt} \mbox{ for any multi-indices }\alpha,\beta\geq 0.
\end{equation}
Then for every $\eps>0$, there exist positive constants $\delta$ and $C$, and an open set $U \subset V$ containing $p$, such that for $N_0(k)=[(\frac{k}{C})^{\frac{1}{2a+2\eps}}]$ and uniformly for any $x,y\in U$, we have in the frame $e_L^k\otimes \bar{e}_L^k$
\begin{equation*}
K_k(x,y)=e^{k\psi(x,\bar y)}\frac{k^n}{\pi^n}\left (1+\sum _{j=1}^{ N_0(k)-1}\frac{b_j(x, \bar y)}{k^j}\right )+e^{k\left(\frac{\phi(x)}{2}+\frac{\phi(y)}{2}\right)}e^{-\delta k^{\frac{1}{2a+2\eps}}}O(1), 
\end{equation*}
where $\psi(x, z)$ is a certain almost holomorphic extension of $\phi(x)$ near the diagonal \footnote{In the sense of Borel and H\"ormander \cite{HoAH}; see our definition \eqref{AHE}.} and $b_m(x , z)$ are certain almost holomorphic extensions (defined by \eqref{Recursive}) of the Bergman kernel coefficients $b_m(x, \bar x)$. 
\end{thm}

As a first corollary of this theorem, we get a complete asymptotic expansion in a $k^{-\frac12+\frac{1}{4a+4\eps}}$ neighborhood of the diagonal.

\begin{cor}\label{complete asymptotics}
Given the same assumptions and notations as in the above theorem, there exist positive  constants $C$ and $\delta$, and an open set $U \subset V$ containing $p$, such that for all $k$ and $N\in \mathbb{N}$, we have for all $x,y\in U$ satisfying $d(x,y)\leq \delta k^{-\frac{1}{2}+\frac{1}{4a+4\eps}}$,
\begin{equation} \label{complete formula}
K_k(x,y)=e^{-k\psi(x,\bar{y})}\frac{k^n}{\pi^n}\left(1+\sum_{j=1}^{{N-1}}\frac{b_j(x,\bar{y})}{k^j}+\mathcal{R}_N(x, \bar y,k)\right),
\end{equation}
where 
\begin{equation}\label{complete error}
	\left|\mathcal{R}_N(x, \bar y ,k)\right|\leq {\frac{C^{N}N!^{2a+2\eps}}{k^N}}.
\end{equation}

And if we only assume $b_m(x,z)$ are arbitrary almost holomorphic extensions of Bergman kernel coefficients $b_m(x,\bar{x})$, then we still have \eqref{complete formula}. But the remainder term estimate will be weaker: $\left|\mathcal{R}_N(x, \bar y ,k)\right|\leq \frac{C_N}{k^N}$ for some constant $C_N$.
 
\end{cor}

As another corollary, we obtain the following off-diagonal asymptotic in terms of Calabi's diastasis \cite{Cal} function defined by
\begin{equation}\label{Diastatis} 
D(x, y)= \phi(x) + \phi(y) - \psi(x, \bar y) - \psi(y, \bar x).  
\end{equation} We point out that near a given point $p \in M$, we have  $D(x, y) = |x -y|_p^2 + O(|x-p|_p^3+|y-p|_p^3)$, where $|z|^2_p:= \sum_{ i, j=1}^n \phi_{i \bar j}(p)z_i\oo{z_j }$. If we use \textit{Bochner coordinates} at $p$ (introduced in \cite{Boc}), in which the \k potential admits the form $\phi(x)=|x|^2+O(|x|^4)$, we have $D(x, y) = |x -y|_p^2 + O(|x-p|_p^4+|y-p|_p^4)$. 
\begin{cor}\label{Log}
Under the same assumptions and notations (and the same $\delta$ and same $U$) as in Theorem \ref{Main},  we have uniformly for all $x , y \in U$ satisfying $ D(x, y) \leq  \frac12 \delta k^{-1+\frac{1}{2a+2\eps}}$,
\begin{equation}
\frac{1}{k}\log \left | K_k (  x, y ) \right | _{h^k}=-\frac{D(x, y)}{2} + \frac{n \log k}{k} -\frac{n\log\pi}{k}+ O \left (\frac{1}{k^2} \right ).
\end{equation}
\end{cor}
The following scaling asymptotic is then immediate:
\begin{cor}\label{Cor2} In Bochner coordinates at $p$, we have uniformly for all $u,v \in \mathbb{C}^n$ with $| u|_p$ and $|v|_{p} < \frac{\sqrt{\delta}}{3}$,
\begin{equation*}
\frac{1}{k^{\frac{1}{2a+2\eps}}}\log \left | K_k \left (\frac{u}{k^{\frac12-\frac{1}{4a+4\eps}}},\frac{v}{k^{\frac12-\frac{1}{4a+4\eps}} }\right ) \right |_{h^k}=-\frac{|u-v|_p^2}{2} + \frac{n \log k}{k^{\frac{1}{2a+2\eps}}}-\frac{n\log\pi}{k^{\frac{1}{2a+2\eps}}} + O \left ( \frac{1}{k^{1+\frac{1}{2a+2\eps}}}\right ).
\end{equation*}
\end{cor}

One of the key ingredients in our proofs is the following estimate on the Bergman kernel coefficients $b_m(x, z)$. We emphasize again that $b_m(x , z)$ are particular almost holomorphic extensions of the Bergman kernel coefficients $b_m(x, \bar x)$ appearing in the on-diagonal expansion \eqref{ZC} of Zelditch \cite{Ze1} and Catlin \cite{Ca}.
\begin{thm}\label{MainLemma}
Assume the \k potential $\phi$ is in Gevrey class $G^a(V)$ for some $a>1$. Let $b_m(x , z)$ be the almost holomorphic extensions (defined by \eqref{Recursive}) of the Bergman kernel coefficients $b_m(x, \bar x)$.  Then, there exists a neighborhood $U\subset V$ of $p$, such that for any $m \in \mathbb N$ we have
\begin{equation*}
\|b_m(x,z)\|_{L^\infty(U\times U)} \leq C^{m}m!^{2a+2\eps}, 
\end{equation*}
where $C$ is a constant independent of $m$ but dependent on $\eps$.  Moreover, we have the following estimates on the derivatives of $b_m(x,z)$. Denote $v=(x,z)$. For any multi-indices $\alpha$ and $\beta$ and any $(x, z) \in U \times U$
\begin{equation}\label{bmupperbound}
\left | D_v^\alpha D_{\oo{v}} ^\beta b_m(x,z) \right |  \leq C^{m+|\alpha|+|\beta|}m!^{2a+2\eps}\alpha!^{a+\eps} \beta!^{a+\eps} \exp{\left(-b (1- \delta_0(|\beta|)){|x-\bar{z}|^{-\frac{1}{a-1}}}\right)} ,
\end{equation}
where $C$ is a constant independent of $m, \alpha, \beta$ but dependent on $\eps$, and $\delta_0(|\beta|)=1$ only if $\beta=0$ and is zero otherwise.  The constant $b$ is positive and is independent of $\alpha, \beta, m, \eps$. In addition, when we are restricted to the diagonal $z= \bar x$, we can choose $\eps=0$.  
\end{thm}

\begin{rmk} We conjecture that in the Gevrey $a$ case, there exist certain almost holomorphic extension $b_m(x,z)$ of the Bergman coefficients $b_m(x,\bar{x})$ such that
\begin{equation}\label{conj}
	\left \| D_v^\alpha D_{\oo{v}} ^\beta b_m(x,z) \right \|_{L^\infty(U\times U)}  \leq C^{m+|\alpha|+|\beta|}m!^{2a-1}\alpha!^{a} \beta!^{a} \exp{\left(-b (1- \delta_0(|\beta|)){|x-\bar{z}|^{-\frac{1}{a-1}}}\right)}
\end{equation}
As we show in this paper, if this conjecture holds true, then all of the above results can be improved accordingly. In particular, the quantities $N_0(k)= [(k /C)^{\frac{1}{2a+2\eps}} ]$ and $e^{- \delta k^\frac{1}{2a+2\eps}}$ in the remainder estimate of Theorem \ref{Main} would be replaced by $ [(k /C)^{\frac{1}{2a-1}}]$ and $e^{- \delta k^\frac{1}{2a-1}}$, moreover Corollary \ref{Log} would hold for all $D(x, y) \leq \frac12 \delta k^{-1+\frac{1}{2a-1}}$. We expect \eqref{conj} is the best possible result one can seek, because by \cite{LuTian} the leading term in $b_m(x,\bar{x})$ is $\frac{m}{(m+1)!}\Delta^{m-1}\rho(x)$ where $\rho$ is the scalar curvature, so when the metric is in Gevrey class $G^a$, we have $\frac{m}{(m+1)!}\Delta^{m-1}\rho(x)\approx C^m m!^{2a-1}$.  However we are unable to prove this conjecture for general Gevrey $a$ \k metrics using our method, which is based  on a recursive formula of \cite{BBS}. In Section \ref{optimal}, we discuss the optimality and limitations of this method. 

\end{rmk}

There is a huge literature on Bergman kernels on complex manifolds. Before closing the introduction we only list some related work that were not cited above: \cite{BoSj, En, Ch, Loi, LuTian, Loi, MaMa, Liu, LiuLu1, Seto, LuZe, Ze3, LS}. Applications of the Bergman kernel, and the closely related Szeg\"o kernel, can be found in \cite{Do}, \cite{BSZ}, \cite{ShZe}, \cite{YZ}. The book of Ma and Marinescu \cite{MaMaBook} contains an introduction to the asymptotic expansion of the Bergman kernel and its applications. See also the book review \cite{ZeBookReview} for more on the applications of Bergman kernels.  

\subsection*{Organization of the paper} In Sections \ref{Sec Local} and \ref{Sec MethodBBS}, we follow the construction of local Bergman kernel in \cite{BBS}, but we obtain precise estimates for the error term by using the growth rate of Bergman coefficients $b_m(x,z)$ provided by Theorem \ref{MainLemma}. In Section \ref{Sec Global}, we give the proofs of Theorem \ref{Main} and Corollaries  \ref{complete asymptotics} and \ref{Log}. The proof of Theorem \ref{MainLemma} will be given in Section \ref{Sec ProofofMainLemma}. Section \ref{optimal} discusses the optimality of our bounds on Bergman coefficients. Section \ref{ProofOfGevreyStuff} contains the proofs of the properties of almost homomorphic extensions of Gevrey functions.

\section{Local Bergman kernels}\label{Sec Local}

In \cite{BBS}, by using \emph{good} complex contour integrals, Berman-Berndtsson-Sj\"ostrand constructed  \textit{local reproducing kernels} (mod $e^{-k\delta}$) for $U=B^n(0,1)\subset \mathbb{C}^n$, which reproduce holomorphic sections in $U$ up to  $e^{-k\delta}$ error terms.  These kernels  are in general not holomorphic. By allowing more flexibility in choosing the amplitudes in the integral, the authors modified  these local reproducing kernels  to  local Bergman kernels, which  means that they are almost \emph{holomorphic}  local reproducing kernels mod $O(k^{-N})$. The global Bergman kernels are then approximated  using the standard  H\"ormander's $L^2$ estimates. 

Throughout this paper, we assume that $\phi$ is in the Gevrey class $G^a(V)$ for some open neighborhood $V\subset M$ of a given point $p$. Let $B^n(0,r)$ be the ball of radius $r$ in $\mathbb C^n$. We identify $p$ with $0\in \mathbb{C}^n$ and $V$ with the ball $B^n(0,3)\subset\mathbb{C}^n$ and denote $U=B^n(0, 1)$. Let $e_L$ be a local holomorphic frame of $L$ over $V$ as introduced in the introduction. 
For each positive integer $k$, we denote $H_{k\phi}(U)$ to be the inner product space of $L^2$-holomorphic functions on $U$ with respect to
\begin{equation*}
\left ( u, v \right)_{k\phi}=\int_U u \bar v \, e^{-k\phi}d \text{Vol},
\end{equation*} 
where $d\text{Vol}=\frac{\omega^n}{n!}$ is the natural volume form induced by the \k form $\omega= \frac{\sqrt{-1}}{2} \partial \bar \partial \phi$. 
So the norm of $u \in H_{k\phi}(U)$ is given by
\begin{equation*}
\|u\|^2_{k\phi}=\int_U |u|^2e^{-k\phi}d\text{Vol}. 
\end{equation*}
Let $\chi\in C_0^\infty(B^n(0,1))$ be a smooth cut-off function such that $\chi=1$ in $B^n(0,\frac{1}{2})$ and vanishes outside $B^n(0,\frac{3}{4})$.  The following result  gives a refinement of the the result of \cite{BBS} by providing a more precise estimate for the error term when the K\"ahler potential is in Gevrey class $G^a$. The main ingredient of the proof is Theorem \ref{MainLemma}, whose proof is delayed to Section \ref{Sec ProofofMainLemma}.
\begin{prop}\label{LocalBergmanKernel}
For each $N \in \mathbb N$, there exist $K_{k,x}^{(N)}(y) \in H_{k\phi}(U)$ and a positive constant $C$ independent of $N$ and $k$, such that for all $u \in H_{k\phi}(U)$ we have 
\begin{align}\label{eq 5.2}
\forall x \in B^n(0, 1/4): \quad u(x)=\left(\chi u, K^{(N)}_{k,x}\right)_{k\phi}+k^ne^{\frac{k\phi(x)}{2}}\mathcal R_{N+1}(\phi, k) \|u\|_{k\phi},
\end{align}
where
\begin{equation}\label{eq 5.3}
\quad |\mathcal R_{N+1}(\phi, k)| \leq \frac{C^{N+1}(N+1)!^{2a+2\eps}}{k^{N+1}}.
\end{equation}
The function $K_{k,x}^{(N)}$ is called a local Bergman kernel of order $N$.
\end{prop}
\begin{rmk}
In \cite{BBS}, only the qualitative estimate $\mathcal R_{N+1}(\phi, k)= O_N(\frac{1}{k^{N+1}})$ is given.
\end{rmk}

To prove Proposition \ref{LocalBergmanKernel}, we first need to recall the techniques of \cite{BBS}. 

\subsection{Review of the method of Berman-Berndtsson-Sj\"ostrand} \label{ReviewBBS} The main idea is to construct the local almost holomorphic reproducing kernel (also called local Bergman kernel) by means of the calculus of contour pseudo-differential operators (contour $\Psi$DO for short) introduced by Sj\"ostrand \cite{Sj}.  Before we introduce the notion of contour integrals we present some notations and definitions. 

Suppose $\phi(x)$ is in Gevrey class $G^a(V)$ and $V=B^n(0, 3)$. By replacing $\phi(x)$ by $\phi(x) - \phi(0)$, we can assume that $\phi(0)=0$.  We  then denote $\psi(x,z)=F(\phi)(x,z)$ defined later in Definition \ref{AHE} to be one holomorphic extension of $\phi(x)$. 
Moreover, since $\phi(x)$ is real-valued, we have $ \oo{\psi(x, z)} = \psi( \bar z, \bar x)$. We also define
\begin{equation}\label{theta}
	\theta(x,y,z)=\int_0^1(D_x\psi)(tx+(1-t)y,z)dt,
\end{equation}
where the differential operator $D_x$ is the gradient operator defined by
\begin{align*}
D_x&=(D_{x_1},D_{x_2},\cdots,D_{x_n}).
\end{align*}
Note that $\theta(x, x, z)= \psi_x(x, z)$. 
It is easy to prove that the Jacobian of the map $(x,y,z)\rightarrow (x,y,\theta)$ at $(x,y,z)=(0,0,0)$ is non-singular.  Thus the map is actually  an almost biholomorphic map between two neighborhoods of the origin of $\mathbb C^{3n}$. As a result, we can use $(x,y,z)$ or $(x,y,\theta)$ as local coordinates interchangeably. Without  loss of generality we can assume that $(x, y, z) \in B^n(0, 3) \times B^n(0, 3) \times B^n(0, 3)$ and  $\theta \in W$, where
$$W =\theta \left ( B^n(0, 3) \times B^n(0, 3) \times B^n(0, 3)\right ). $$ Note that $W$ contains the origin because by our assumption $\phi(0)=0$. 

A fundamental idea of \cite{BBS} is to use the estimate
\begin{equation}\label{BBS 2.1}
u(x)= c_n \left(\frac{k}{2\pi}\right)^n\int_\Lambda e^{k\,\theta\cdot (x-y)}u(y)\chi(y)d\theta\wedge dy
+O(e^{-k\delta})e^{\frac{k\phi(x)}{2}}\|u\|_{k\phi}
\end{equation}
which holds uniformly for $x \in B^n(0, \frac14)$, for any holomorphic function $u$ defined on $ B^n(0, 1)$. Here, 
$c_n=i^{-n^2}$, $\delta$ is a positive constant, and $\Lambda=\{(y,\theta):\theta=\theta(x,y)\}$ is a \textit{good contour}, which means that there exists $\delta>0$ such that for any $x,y$ in a neighborhood of the origin, 
\begin{align}\label{GoodContour}
	2\Real\theta\cdot (x-y)\leq -\delta|x-y|^2-\phi(y)+\phi(x).
\end{align} 
One can easily verify that 
\begin{equation} \label{contour} \Lambda=\{(y,\theta):\theta=\theta(x,y,\bar{y})\}, \end{equation} with $\theta(x,y,z)$ defined by \eqref{theta}, is a good contour by observing that
\begin{equation*}
\theta\cdot (x-y)=\psi(x,\bar{y})-\psi(y,\bar{y}).
\end{equation*} 
To put \eqref{BBS 2.1} into a useful perspective, one should think of the integral in \eqref{BBS 2.1} as a contour $\Psi$DO defined as follows. Let ${a=a(x,y,\theta,k)}$ be an almost holomorphic symbol in $B^n(0, 3) \times B^n(0, 3) \times W$, with an asymptotic expansion of the form
\begin{equation*}
a(x,y,\theta,k)\sim a_0(x,y,\theta)+\frac{a_1(x,y,\theta)}{k}+\frac{a_2(x,y,\theta)}{k^2}+\cdots \; .
\end{equation*}
For simplicity, we will suppress the dependency on $k$ and write $a=a(x,y,\theta)$.

A $\Psi$DO associated to a good contour $\Lambda$ and an amplitude $a(x, y, \theta)$, is an operator on $C^\infty_0(U)$ defined by 
$$\text{Op}_\Lambda(a)\, u = c_n \left(\frac{k}{2\pi}\right)^n \int_\Lambda e^{k\,\theta\cdot (x-y)}a(x, y, \theta)\,  u(y)\, d\theta\wedge dy.$$ Thus in this language \eqref{BBS 2.1} means that for $x \in B^n(0, 1/4)$
$$(\chi u)(x)= \text{Op}_\Lambda(1) (\chi u) 
+O(e^{-k\delta})e^{\frac{k\phi(x)}{2}}\|u\|_{k\phi}.$$
Roughly speaking this says that $\text{Op}_\Lambda (1)$ is the identity operator mod $O(e^{-k\delta})$.  
We define the integral kernel $K_{k, x}(y)$ of $\text{Op}_\Lambda(a)$ with respect to the inner product $( \cdot, \cdot )_{k \phi}$, by 
$$\text{Op}_\Lambda(a) u =\left(u, K_{k,x}\right)_{k\phi}. $$ The first observation is that the kernel $K_{k, x}(y)$ of $\text{Op}_\Lambda(1)$, associated to the contour \eqref{contour}, is not almost holomorphic. The idea of \cite{BBS} is to replace $\text{Op}_\Lambda(1)$ by $\text{Op}_\Lambda(1+a)$ where $a(x, y, \theta)$ is a \textit{negligible amplitude} and the kernel of $\text{Op}_\Lambda(1+a)$ is almost holomorphic. An amplitude $a(x, y, \theta)$ is negligible if 
$$\text{Op}_\Lambda(a) (\chi u)= O(k^{-\infty})e^{\frac{k\phi(x)}{2}}\|u\|_{k\phi}. $$ 
To find a suitable condition for negligible amplitudes one formally writes 
$$ \text{Op}_\Lambda(a) = \text{Op}_\Lambda(S a |_{x=y}),$$
where $S$ is a standard operator that is used in microlocal analysis to turn a symbol $a(x, y, \theta)$ of a $\Psi$DO to a symbol of the form $\widetilde a (x, \theta)$. The operator $S$ is formally defined by
\begin{equation*}
S=e^{\frac{D_\theta \cdot D_y}{k}}=\sum_{m=0}^\infty\frac{(D_\theta\cdot D_y)^m}{m!k^m}.
\end{equation*}
Then an amplitude $a$ is negligible if $S a |_{x=y} \sim 0$ as a formal power series. This implies that there exists an almost holomorphic vector field $A(x, y, \theta)$ with formal power series \begin{equation*}
A(x,y,\theta)\sim
A_0(x,y,\theta)+\frac{A_1(x,y,\theta)}{k}+\frac{A_2(x,y,\theta)}{k^2}+\cdots.
\end{equation*} such that 
\begin{equation}\label{SA}
Sa\sim k(x-y)\cdot SA \mod \mathcal{I}^\infty,
\end{equation}
where $\mathcal{I}^\infty$ is the set of functions $f$ such that for any multi-index $\alpha$, $D^\alpha f=0$ when $x=y=\bar{z}$. Here $a_m(x,y,\theta)$ are almost holomorphic functions and $A_m(x,y,\theta)$ are almost holomorphic vector fields in $\mathbb C^n$, defined on $B^{n}(0,3) \times B^{n}(0,3) \times W$. 

One particular $SA$ can be solved as follows. First note that by \eqref{SA} we must have $(SA)_0=0$ and $A_0=0$. Then we put
\begin{equation}\label{def SA}
SA(x,y,z)= - \frac{1}{k}\int_0^1 (D_{y}Sa)(x,tx+(1-t)y,z)dt.
\end{equation}
By taking $S^{-1}$, $A$ can be solved uniquely as 
\begin{equation}\label{A}
	A(x,y,z)= - \frac{1}{k}S^{-1}\int_0^1 (D_{y}Sa)(x,tx+(1-t)y,z)dt.
\end{equation} 
Then by the fundamental theorem of calculus we have
\begin{equation*}
Sa(x,y,z)=k(x-y)\cdot SA(x,y,z)- (\oo{x-y})\cdot \int_0^1 (D_{\bar{y}}Sa)(x,tx+(1-t)y,z)dt.
\end{equation*}

By using the inverse operator $S^{-1}$,we have
\begin{equation*}
a(x,y,z)= D_\theta\cdot A  + k(x-y)\cdot A - (\oo{x-y})\cdot S^{-1}\int_0^1 (D_{\bar{y}}Sa)(x,tx+(1-t)y,z)dt.
\end{equation*}
We use $a^{(N)}$ and $A^{(N)}$ to denote the partial sums of $a$ and $A$ up to order $\frac{1}{k^N}$ respectively. And we denote 
\begin{equation*}
	\nabla A:=D_\theta\cdot A+k(x-y)\cdot A.
\end{equation*}
Since $A_0=0$, we obtain
\begin{align}\label{a and nabla A}
\begin{split}
a^{(N)}-\nabla \left(A^{(N+1)}\right)
=&\frac{D_\theta\cdot A_{N+1}}{k^{N+1}}
-(\oo{x-y})\cdot \left(S^{-1}\int_0^1 (D_{\bar{y}}Sa)(x,tx+(1-t)y,z)dt\right)^{(N)}.
\end{split}
\end{align}

Next, we observe that the integral kernel of $\text{Op}_\Lambda(1+a)$ is almost holomorphic if
\begin{equation}\label{aB}
1+a(x,y,\theta)\sim B(x,z(x,y,\theta))\Delta_0(x,y,\theta),
\end{equation}
where
\begin{equation*}
	\Delta_0(x,y,\theta)=\frac{\det \psi_{yz}(y,z)}{\det\theta_z(x,y,z)},
\end{equation*}
and $B(x, z)$ is almost holomorphic and has an asymptotic expansion of the form
\begin{equation}\label{B} B(x, z) \sim b_0(x,z)+\frac{b_1(x,z)}{k}+\frac{b_2(x,z)}{k^2}+\cdots, \end{equation} where  $b_m(x, z)$ are almost holomorphic.  In fact, as it turns out, $b_m(x, z)$ are an almost holomorphic extensions of $b_m(x, \bar x)$, the Bergman kernel coefficients of the on-diagonal asymptotic expansion of Zelditch-Catlin \eqref{ZC}. 

If the amplitude $a$ is negligible, then by applying $S( \cdot) |_{x=y}$ to both sides of \eqref{aB}, we get 
$$S \left( B(x,z(x,y,\theta))\Delta_0(x,y,\theta) \right )|_{x=y} \sim 1.$$ 
From this, one gets the following recursive equations for Bergman kernel coefficients $b_m(x, z)$, which will play a key role in the proof of Theorem \ref{MainLemma}:
\begin{equation}\label{Recursive1}
	b_m(x,z(x,x,\theta))=-\sum_{l=1}^m\frac{(D_y\cdot D_\theta)^l}{l!}\big(b_{m-l}\left(x,z(x,y,\theta)\right)\Delta_0(x,y,\theta)\big) \Big|_{y=x}.
\end{equation}
Additionally, by comparing the coefficients on both sides of \eqref{aB}, we have the following relations between $a_m$ and $b_{m}$:
\begin{equation}\label{ambm}
a_m(x,y,\theta)=
\begin{cases}
\Delta_0(x,y,\theta)-1 & \mbox{ when }m=0,\\
b_m(x,z(x,y,\theta))\Delta_0(x,y,\theta) & \mbox{ when }m\geq 1.
\end{cases}
\end{equation}
These equations will be useful in estimating $a_m$ in terms of the bounds on $b_m$ from Theorem \ref{MainLemma}.

\subsection{Almost Holomorphic Extensions of Gevrey functions}\label{AHEoGF}
In this section, we will review the Gevrey class and consider almost holomorphic extensions of functions in such a class. Indeed, there are many different ways to construct almost holomorphic functions. We will adapt the way in \cite{Ju1} to construct a particular one which is suitable for our analysis. Afterwards, various properties of such an extension are introduced, which will be used for the proof of Proposition \ref{LocalBergmanKernel} in Section \ref{Sec MethodBBS}. Although all the properties are natural and elementary, the proofs are however very lengthy. For the convenience of the readers, we shall only state the results we need and postpone the proofs to Section \ref{ProofOfGevreyStuff}.   

We recall the definition of Gevrey class $G^a(U)$. For more details, we refer the readers to \cite{Ge}. Take $\alpha,\beta\in (\mathbb{Z}^{\geq 0})^n$. Here are some standard notations of multi-indices we shall use in the following. 
\begin{itemize}
	\item $|\alpha|=\alpha_1+\alpha_2+\cdots+\alpha_n$.
	\item $ \alpha \leq \beta $ if $\alpha_1\leq \beta_1, \alpha_2\leq \beta_2,\cdots,\alpha_n\leq \beta_n$.
	\item $\alpha< \beta$ if $ \alpha \leq \beta$ and $\alpha \ne \beta$.
	\item $ \alpha !=\alpha_1!\alpha_2!\cdots \alpha_n!$.
\end{itemize}
\begin{defn}\label{Gevrey}
	Let $a\in (1,\infty)$ and $U$ be an open subset of $\mathbb C^n$. We denote by $G^{a}(U)$ the set of functions $f(x)\in C^\infty(U,\mathbb{C})$ such that there exists some constant $C_0=C_0(f)>0$ and $C_1=C_1(f)>0$, satisfying
	\begin{equation}
	\|D_x^\alpha D_{\bar{x}}^\beta f\|_{L^\infty(U)} \leq C_0 C_1^{|\alpha|+|\beta|}(\alpha!\beta!)^a,
	\end{equation}
	for any multi-indices $\alpha,\beta\geq 0$. The space $G^a(U)$ is called the Gevrey class of index $a$. Note that each class $G^a(U)$ forms an algebra which is closed under differentiation and integration. 
\end{defn}

For any $f\in G^a(U)$, an almost holomorphic extension $F(f)(x,z)$ is a smooth function on $U\times U$ such that $F(f)(x,\bar{x})=f(x)$ and the anti-holomorphic derivatives have infinite vanishing order along $x=\bar{z}$. We will use the way in \cite{Ju1} to construct a particular almost holomorphic extension. In fact the construction of \cite{Ju1} is adapted from Borel's method (see also H\"ormander\cite{HoAH} ). Here, we follow \cite{Ju1} but we use  a cut-off function $\chi$ in the Gevrey class $\in G^{1+\eps}(\mathbb{R})$ where $\eps$ is an arbitrary positive constant, and
\begin{equation}\label{Gevreycutoff}
\chi(x)=
\begin{cases}
1 & |x|\leq \frac{1}{2},\\
0 & |x|\geq 1.
\end{cases}
\end{equation}
To show the existence of such a cut-off function, one can use the fact that for any $\eps>0$, the function defined as 
\begin{equation*}
f_\eps(x)=
\begin{cases}
\exp(-x^{-\frac{1}{\eps}}) & x>0\\
0  & x\leq 0,
\end{cases}
\end{equation*}
belongs to $G^{1+\eps}(\mathbb{R})$ (See \cite{CC} for more details). Then by the standard construction, we define 
\begin{equation*}
g(x)=\begin{cases}
0 & x\leq 0,\\
\frac{\int_0^x f_\eps(t)f_\eps(\frac{1}{2}-t)dt}{\int_0^1 f_\eps(t)f_\eps(\frac{1}{2}-t)dt} & x\in (0,1),\\
1 & x\geq 1.
\end{cases}
\end{equation*} 
We can take our cut-off function to be $\chi(x)=g(x+1)g(-x+1)$. 

We now define our almost holomorphic extension of Gevrey functions. 

\begin{defn}\label{AHE}
	Let $a\in (1,\infty)$, $U$ be the unit ball $B(0,1)$ in $\mathbb C^n$,  and $f(x)\in G^{a}(U)$. Let $C_1=C_1(f)$ be the constant in Definition \ref{Gevrey}. Then for $(y,z)\in U\times U$, we define an almost holomorphic extension
	\begin{equation*}
	F(f)(y,z)=\sum_{\alpha,\beta\geq 0} \frac{D_x^\alpha D_{\bar{x}}^\beta f}{\alpha!\beta!}\left(\frac{y+\bar{z}}{2}\right)\left(\frac{y-\bar{z}}{2}\right)^\alpha \left(\frac{z-\bar{y}}{2}\right)^\beta\chi\left(|\alpha+\beta|^{2(a-1)}4^{a-1}C_1^2\left|y-\bar{z}\right|^2\right).
	\end{equation*}
\end{defn}

We will justify that $F(f)$ defined as above is genuinely an almost holomorphic extension of $f$ along $y=\bar{z}$. It is easy to see $F(f)(x,\bar{x})=f(x)$. And in the next lemma, we will verify that  $D_{\bar{y}}F(f)(y,z)=O(|y-\bar{z}|^\infty)$, and $D_{\bar{z}}F(f)(y,z)=O(|y-\bar{z}|^\infty)$. To be more precise, we show that these quantities vanish at a certain exponential rate along $y=\bar{z}$. 
\begin{lemma}\label{lem almostholomorphic}
	There exist positive constants $C$ and $b$ such that for any $y,z \in U$, the almost holomorphic extension $F(f)$ satisfies 
	\begin{align}
	\begin{split}
	|D_{\bar{y}}F(f)(y,z)|\leq C \exp{\left(-b|y-\bar{z}|^{-\frac{1}{a-1}}\right)},\\
	|D_{\bar{z}}F(f)(y,z)|\leq C \exp{\left(-b|y-\bar{z}|^{-\frac{1}{a-1}}\right)}.
	\end{split}
	\end{align}	
	In particular, $F(f)$ is almost holomorphic along $y=\bar{z}$.
\end{lemma}

Indeed, there are various ways to define an almost holomorphic extension besides Definition \ref{AHE}. But they are all the same up to an $O(|y-\bar{z}|^\infty)$ error term.
\begin{lemma}\label{AHE Uniqueness}
	Let $U$ be the unit ball $B(0,1)$ in $\mathbb C^n$ and $f(x)\in C^\infty(\oo{U})$. If $F(y,z),  \widetilde{F}(y,z)\in C^\infty(U\times U)$ are both almost holomorphic extensions of $f$, then 
	\begin{equation*}
		F(y,z)-\widetilde{F}(y,z)=O\left(|y-\bar{z}|^\infty\right).
	\end{equation*}
\end{lemma}

Next, we show a more general version of Lemma \ref{lem almostholomorphic}, which gives estimates on all the derivatives of $F(f)$. It turns out that if $f\in G^a(U)$, then $F(f)\in G^{a+\eps}(U\times U)$ and when the anti-holomorphic derivative appears, it always vanishes to infinite order along $y=\bar{z}$ at a certain exponential rate. 
\begin{lemma}\label{AHE2}
	Take $f\in G^a(U)$. Let $C_0(f)$ and $C_1(f)$ be the constants satisfying \eqref{Gevrey} for $f$. Then for any $\eps>0$, there exist positive constants $C_1=C_1(\eps, a, C_1(f))$, $b=b(a, C_1(f))$ and $A=A(a,n)$ such that for any multi-indices $\gamma,\delta,\xi,\eta\geq 0$, we have
	\begin{equation}
	|D_y^\gamma D_z^\delta D_{\bar{y}}^\xi D_{\bar{z}}^{\eta} F(f)(y,z)|
	\leq AC_0(f)C_1^{|\gamma+\delta+\xi+\eta|}(\gamma!\delta!\xi!\eta!)^{a+\eps},
	\end{equation} hence $F(f)(y,z)\in G^{a,\eps}(U\times U)$.
	Moreover, if $\xi+\eta>0$, then
	\begin{equation}
	| D_y^\gamma D_z^\delta D_{\bar{y}}^\xi D_{\bar{z}}^{\eta} F(f)(y,z)|
	\leq AC_0(f)C_1^{|\gamma+\delta+\xi+\eta|}(\gamma!\delta!\xi!\eta!)^{a+\eps}
	\,  \exp\left ( {- b {|y-\bar{z}|^{-\frac{1}{a-1}}}} \right ).
	\end{equation}
	In addition, when we are restricted to the diagonal $z=\bar{y}$, we can let $\eps=0$ in the above estimates.
\end{lemma}

This motivates us to give the following definition.

\begin{defn} Let $U$ be an open neighborhood of the origin in $\mathbb{C}^{2n}$ and let $\eps>0$ be a constant. A function $F(y, z) \in C^\infty(U)$ is called $G^{a,\eps}$-\textit{almost holomorphic} along the diagonal $z=\bar{y}$ if  there exist positive constants $C_0=C_0(F), C_1=C_1(F)$ and $b=b(F)$ such that for any multi-indices  $\gamma,\delta,\xi,\eta\geq 0$, we have
	\begin{equation}
	| D_y^\gamma D_z^\delta D_{\bar{y}}^\xi D_{\bar{z}}^{\eta} F(y,z)|
	\leq C_0 C_1^{|\gamma+\delta+\xi+\eta|}(\gamma!\delta!\xi!\eta!)^{a+\eps}.  
	\end{equation}	
	And when $\xi+\eta>0$, we have
	\begin{equation}\label{AHG}
	| D_y^\gamma D_z^\delta D_{\bar{y}}^\xi D_{\bar{z}}^{\eta} F(y,z)|
	\leq C_0 C_1^{|\gamma+\delta+\xi+\eta|}(\gamma!\delta!\xi!\eta!)^{a+\eps}
	\,  \exp\left ( {- b {|y-\bar{z}|^{-\frac{1}{a-1}}}} \right ).
	\end{equation}
	In addition, when we are restricted to the diagonal $z=\bar{y}$, we can let $\eps=0$ in the above estimates.
\end{defn}
We use $\mathcal A_{diag}^{a,\eps}(U)$ for the class of such functions. And we also use $\mathcal I_{diag}^{a,\eps}(U)$ for functions $F(y, z) \in C^\infty(U \times U)$ satisfying \eqref{AHG} with no restrictions on $\xi$ and $\eta$ (i.e. \eqref{AHG} holds even if $\xi=\eta=0$). And we say a vector belongs to $\mathcal A_{diag}^{a,\eps}(U)$ or $\mathcal I_{diag}^{a,\eps}(U)$ if each component function belongs to that class. Obviously, we have $\dbar \mathcal A_{diag}^{a,\eps}(U)\subset \mathcal I_{diag}^{a,\eps}(U)$.

Since the recursive formula for Bergman coefficients requires studying functions of three variables in $\mathbb C^n$, we also present the following definition for functions in $\mathbb C^{3n}$. 

\begin{defn}\label{A_theta}
	Let $\eps>0$ be a constant. Let $\theta(x, y, z)$ be a function on $U$ such that $ \Phi: (x, y, z) \to (x, y, \theta(x, y, z))$ is a diffeomorphism between $U$ and its image denoted by an open set $U'\subset \mathbb{C}^{3n}$. Take $f(x,y,\theta)\in C^\infty(U')$.  Denote $v'=(x,y, \theta)$. We say $f(x,y,\theta)$ is \textit{ $G^{a,\eps}$-almost holomorphic} along $x=y=\bar{z}$ under $\theta$, if there exist some positive constants $C_0=C_0(f), C_1=C_1(f)$, and $b=b(f)$, such that for any multi-indices $\alpha,\beta\geq 0$, we have
	\begin{equation}\label{eq 4.13}
	\left | \left ( D^\alpha_{v'}D^{\beta}_{\bar{v}'}f \right )(x, y, \theta(x, y, z)) \right |\leq C_0C_1^{|\alpha+\beta|}(\alpha!\beta!)^{a+\eps} \, \exp\left(- b \,  (1-\delta_0(|\beta|)) {\max \{ |x- \bar z|, |y-\bar{z} | \} ^{-\frac{1}{a-1}}}  \right)   ,
	\end{equation}
	where $\delta_0(\cdot)$ is the delta function whose value is $1$ at $0$ and it is zero elsewhere. In addition, when we are restricted to $x=y=\bar{z}$, we can let $\eps=0$ in the above estimate. 
	
	We use $\mathcal{A}^{a,\eps}_{\theta}(U')$ to denote the set of all  $G^{a,\eps}$-almost holomorphic functions along $x=y=\bar{z}$ under $\theta$ in the above sense. We will also use $\mathcal I_{\theta}^{a,\eps}(U)$ for smooth functions $f(x,y,\theta)$ such that for any multi-indices $\alpha,\beta\geq 0$, 
	\begin{equation*}
	\left | \left ( D^\alpha_{v'}D^{\beta}_{\bar{v}'}f \right )(x, y, \theta(x, y, z) \right |\leq C_0C_1^{|\alpha+\beta|}(\alpha!\beta!)^{a+\eps} \, \exp\left(- b \,  {\max \{ |x- \bar z|, |y-\bar{z} | \} ^{-\frac{1}{a-1}}}  \right)   ,
	\end{equation*}
	And we say a vector belongs to $\mathcal A_{\theta}^{a,\eps}(U)$ or $\mathcal I_{\theta}^{a,\eps}(U)$ if each component function belongs to that class.
\end{defn}

In the following, for simplicity we will use the notation 
\begin{align}\label{Lambda}
\lambda_{b, |\beta|}(x, y, z)= \exp\left(- b \,  (1-\delta_0(|\beta|)) {\max \{ |x- \bar z |, |y-\bar{z} | \} ^{-\frac{1}{a-1}}}  \right)
\end{align}

\begin{rmk}
	Note that by the above notation, $\mathcal{A}^{a,\eps}_{z}$ (or $\mathcal{I}^{a,\eps}_{z}$) means $\mathcal{A}^{a,\eps}_{\theta}$ (or $\mathcal{I}^{a,\eps}_{\theta}$) when $\theta(x, y, z)=z$, which corresponds to the case $\Phi =I$.  
\end{rmk}

The space $\mathcal{A}^{a,\eps}_\theta(U)$ is closed under algebraic operations and differentiations.
\begin{lemma}\label{AHE CLOSED}
	For each $\theta$ as described in the previous definition, $\mathcal{A}^{a,\eps}_\theta(U)$ is closed under summation, subtraction, multiplication and differentiation. It is also closed under division if the denominator is uniformly away from zero in $U$.
	
	In particular, suppose $f,g\in \mathcal{A}_\theta^{a,\eps}(U)$. Then we can choose the constants appearing in \eqref{eq 4.13} for the product $fg\in \mathcal{A}_\theta^{a,\eps}(U)$ as 
	\begin{align*}
	C_0(fg)=C_0(f)C_0(g),\quad C_1(fg)=2\max\{C_1(f),C_1(g)\}, \quad\mbox{and}\quad b(fg)=\min\{b(f),b(g)\}.
	\end{align*}
	And for the differentiation, we can choose the constants as 
	\begin{align*}
	C_0(D_{v'}^\alpha D_{\bar{v}'}^\beta f)=C_0(f)(2^aC_1(f))^{|\alpha+\beta|},\quad 
	C_1(D_{v'}^\alpha D_{\bar{v}'}^\beta f)=2^aC_1(f),\quad
	\mbox{and}\quad b(D_{v'}^\alpha D_{\bar{v}'}^\beta f)=b(f),
	\end{align*}
	where $v'=(x,y,\theta)$.	
\end{lemma}

We shall use the following lemma that $\mathcal{A}^{a,\eps}_{z}$ is closed under certain integrals.
\begin{lemma}\label{AHE Integral}
	If $f(x,y,z)\in \mathcal{A}^{a,\eps}_z$, then $g(x,y,z)=\int_0^1 f(x,tx+(1-t)y,z)dt \in \mathcal{A}^{a,\eps}_z$. And we can choose the constants appearing in Definition \ref{A_theta} as
	\begin{align*}
	C_0(g)=C_0(f), \quad C_1(g)=2^{a+\eps+1}C_1(f), \quad b(g)=b(f).
	\end{align*} 
	Similarly, 	If $f(x,y,z)\in \mathcal{I}^{a,\eps}_z$, then $g(x,y,z)=\int_0^1 f(x,tx+(1-t)y,z)dt \in \mathcal{I}^{a,\eps}_z$. And we can choose 
	\begin{align*}
	C_0(g)=C_0(f),\quad C_1(g)=2^{a+\eps+1}C_1(f), \quad b(g)=b(f).
	\end{align*} 
\end{lemma}

The space $\mathcal{A}^{a,\eps}_{\theta}$ is also closed under composition in the following sense.
\begin{lemma}\label{AHE Composition}
	Let $f(x,y,z)\in \mathcal{A}^{a,\eps}_z(U)$ be a function defined on $U\subset \mathbb{C}^{3n}$. Let $\theta(x,y,z)$ be a map on $U$ such that $\Phi: (x,y,z)\rightarrow (x,y,\theta(x,y,z))\in $ is a diffeomorphism between $U$ and its image denoted by an open set $U'\subset \mathbb{C}^{3n}$. Let $\Phi^{-1}: (x,y, \theta)\rightarrow (x,y, z(x,y,\theta))$ be the inverse map of $\Phi$.  If $z=z(x,y,\theta)\in \mathcal{A}^{a,\eps}_\theta (U')$, then the composition function $\widetilde{f}(x,y,\theta)=f(x,y,z(x,y,\theta))\in \mathcal{A}^{a,\eps}_\theta(U')$.
	
	In particular, if we use $C_0(f),C_1(f)$ and $b(f)$ to denote the constants in \eqref{eq 4.13} for an the function $f$, then we can choose the constants for $\widetilde{f}$ as 
	\begin{align*}
	C_0(\widetilde{f})=C_0(f),\quad C_1(\widetilde{f})=2^{a+\eps+3m} m^{a+\eps-1}C_0(z(x,y,\theta))C_1(f)C(z(x,y,\theta)),
	\end{align*}
	and
	\begin{align*}
	b(\widetilde{f})=\min\{b(f),b(z(x,y,\theta))\},
	\end{align*}
	where $m=3n$, $C_0(z(x,y,\theta))=\max_{1\leq i\leq n}C_0(z_i(x,y,\theta))$, $C_1(z(x,y,\theta))=\max_{1\leq i\leq n}C_1(z_i(x,y,\theta))$,
	and $b(z(x,y,\theta))=\min_{1\leq i\leq n}b(z_i(x,y,\theta))$.
\end{lemma}

\begin{rmk}\label{I Composition}
	In Lemma \ref{AHE Composition}, if we further assume that $f(x,y,z)\in \mathcal{I}_z^{a,\eps}$, then the composition $\widetilde{f}$ belongs to $\mathcal{I}_\theta^{a,\eps}$ with the same choice of constants.
\end{rmk}

Now suppose $U=B(0,1)\subset \mathbb{C}^n$ and the \k potential $\phi$ belongs to $G^a(U)$ and let  $\psi=F(\phi)$ be the almost holomorphic extension of $\phi$ defined by \eqref{AHE}. Then it is easy to see that $\psi(y,z)\in \mathcal{A}^{a,\eps}_{z}(U)$. Further by using Lemma \ref{AHE Integral}, if we take $\theta(x,y,z)=\int_0^1(D_y\psi)(tx+(1-t)y,z)dt$, then $\theta\in \mathcal{A}^{a,\eps}_{z}$. The following lemma says that the implicit functions $z=z(x,y,\theta)$ belong to $\mathcal{A}^{a,\eps}_\theta$.

\begin{lemma}\label{AHE Inverse}
	Consider the following system of equations:
	\begin{align}\label{eq 4.32}
	\theta=\int_0^1(D_y\psi)(tx+(1-t)y,z)dt\, .
	\end{align}
	Then the implicit functions $z=z(x,y,\theta)$ determined by the above equations belong to $\mathcal{A}^{a,\eps}_\theta$.
\end{lemma}

As we said at the beginning of this section, the proofs of all the above lemmas will be given in Section \ref{ProofOfGevreyStuff}.

We are now prepared to prove Proposition \ref{LocalBergmanKernel}.

\section{The remainder estimates and the proof of Proposition~\ref{LocalBergmanKernel}}\label{Sec MethodBBS}

Let $a_m$, $A_m$, and $b_m$ be given by \eqref{ambm}, \eqref{A}, and \eqref{B}. Remember that $a^{(N)}$, $A^{(N)}$, and $B^{(N)}$ are the partial sums of $a$, $A$, and $B$ up to order ${k^{-N}}$. When we apply the method of Berman-Berndtsson-Sj\"ostrand, the remainder term is closely related to the growth rate of $a_m$, $A_m$ and their derivatives as we will see soon. So we will first make a series of lemmas on estimating $a_m$ and $A_m$ preparing for the proof of Proposition \ref{LocalBergmanKernel}. 

Let's begin with estimating $a_m$. 

\begin{lemma}\label{amxyz}
	For each integer $m\geq 0$, we have $a_m(x,y,z)\in \mathcal{A}^{a,\eps}_z$. And we can choose the constants appearing in Definition \ref{A_theta} as
	\begin{align*}
	C_0(a_m)= C^{m+1}m!^{2a+2\eps}, && C_1(a_m)=C, &&b(a_m)=b,
	\end{align*}
	where $C$ and $b$ are some positive constants independent of $m$. 
\end{lemma}
\begin{proof}
	Recall the relations between $a_m$ and $b_m$ from \eqref{ambm}:
	\begin{equation*}
	a_m(x,y,z)=
	\begin{cases}
	\Delta_0(x,y,z)-1 & \mbox{ when }m=0,\\
	b_m(x,z)\Delta_0(x,y,z) & \mbox{ when }m\geq 1.
	\end{cases}
	\end{equation*}
	Since $\phi \in G^a$ , the almost holomorphic extension $\psi(y,z)$ introduced as in definition \ref{AHE} belongs to $\mathcal A^{a,\eps}_{z}$. Recall that by Lemma \ref{AHE Integral}, $\theta(x,y,z)=\int_0^1 (D_y\psi)(tx+(1-t)y,z)dt\in \mathcal A^{a,\eps}_{z}$. By Lemma \ref{AHE CLOSED}, we know $\mathcal{A}^{a,\eps}_{z}$ is closed under certain algebraic operations and differentiation, $\Delta_0(x,y,z)=\frac{\det \psi_{yz}(y,z)}{\det \theta_z(x,y,z)}$ is therefore also contained in $\mathcal{A}^{a,\eps}_{z}$. Since $a_m(x,y,z)=b_m(x,z)\Delta_0(x,y,z)$ for $m\geq 1$, by our Lemma \ref{AHE CLOSED} on the multiplication, we can choose
	$$C_0(a_m)=C_0(b_m)C_0(\Delta_0)=C^{m+1}m!^{2a+2\eps},$$  $$C(a_m)=2\max\{C_1(b_m),C_1(\Delta_0)\},$$  and  $$b(a_m)=\min\{b(b_m),b(\Delta_0)\}.$$ 
	Thus, the result follows as $C_0(b_m)=C^mm!^{2a+2\eps}$ for some positive constant $C$ and $C_1(b_m), b(b_m)$ are both independent of $m$ by Theorem \ref{MainLemma}. In addition, it is easy to see that when we are restricted to $x=y=\bar{z}$, $\eps$ can be replaced by $0$.
\end{proof}
\begin{lemma}\label{amxytheta}
	Denote $\widetilde{a}_m=a_m(x,y,\theta)=a_m(x,y,z(x,y,\theta))$. Then $a_m(x,y,\theta)\in \mathcal{A}^{a,\eps}_\theta$ and we can choose
	\begin{align*}
	C_0(\widetilde{a}_m)= C^{m+1}m!^{2a+2\eps}, && C_1(\widetilde{a}_m)=C, &&b(\widetilde{a}_m)=b,
	\end{align*}
	where $C$ and $b$ are some positive constants independent of $m$.
\end{lemma}
\begin{proof}
	By Lemma \ref{AHE Inverse}, we have $z=z(x,y,\theta)\in \mathcal{A}^{a,\eps}_\theta$. Since	$\widetilde{a}_m$ is obtained from the composition of $a_m(x,y,z)$ and the map $z=z(x,y,\theta)$, by Lemma \ref{AHE Composition} 
	$$C_0(\widetilde{a}_m)=C_0(a_m),$$  $$C_1(\widetilde{a}_m)=2^{a+\eps+9n} (3n)^{a-1+\eps}C_0(z(x,y,\theta))C_1(a_m)C_1(z(x,y,\theta)), $$ $$b(\widetilde{a}_m)=\min\{b(a_m),b(z(x,y,\theta))\}.$$
	So the result follows directly from Lemma \ref{amxyz}.
\end{proof}

After we obtain the estimates on $a_m(x,y,z)$ and $a_m(x,y,\theta)$, now we proceed to $(Sa)_m(x,y,z)$ and $(Sa)_m(x,y,\theta)$.
\begin{lemma}\label{Sam}
	For each integer $m\geq 0$, $(Sa)_m(x,y,\theta)\in \mathcal{A}^{a,\eps}_\theta$ and $(Sa)_m(x,y,z)\in \mathcal{A}^{a,\eps}_z$. And we can choose 
	\begin{align*}
	C_0((Sa)_m(x,y,\theta))= C^{m+1}m!^{2a+2\eps}, && C_1((Sa)_m(x,y,\theta))=C, &&b((Sa)_m(x,y,\theta))=b,
	\end{align*}	
	\begin{align*}
	C_0((Sa)_m(x,y,z))= C^{m+1}m!^{2a+2\eps}, && C_1((Sa)_m(x,y,z))=C, &&b((Sa)_m(x,y,z))=b,
	\end{align*}	
	where $C$ and $b$ are some positive constants independent of $m$.
\end{lemma}

\begin{proof}
	Since
	\begin{equation*}
	Sa=\sum_{i=0}^\infty\frac{(D_\theta\cdot D_y)^i}{i!k^i}\sum_{j=0}^\infty \frac{a_j}{k^j}
	=\sum_{m=0}^\infty\sum_{i+j=m}\frac{(D_\theta\cdot D_y)^ia_j}{i!k^m},
	\end{equation*}
	we have
	\begin{equation*}
	(Sa)_m(x,y,\theta)=\sum_{i+j=m}\frac{(D_\theta\cdot D_y)^ia_j}{i!}(x,y,\theta)=\sum_{i+j=m}\sum_{|\delta|=l}\frac{1}{i!}\binom{l}{\delta}D^{\delta}_y D_\theta^\delta a_j(x,y,\theta).
	\end{equation*}
	Denote $v'=(x,y,\theta)$. Therefore, 
	\begin{align*}
	D_{v'}^{\alpha}D_{\bar{v}'}^{\beta}(Sa)_m(x,y,\theta)
	=\sum_{i+j=m}\sum_{|\delta|=i}\frac{1}{i!}\binom{i}{\delta}D_{v'}^{\alpha+(0,\delta,\delta)}D_{\bar{v}'}^{\beta} a_j(x,y,\theta)
	\end{align*}
	By Lemma \ref{amxytheta}, we have
	\begin{align*}
	&\left|D_{v'}^{\alpha}D_{\bar{v}'}^{\beta}(Sa)_m(x,y,\theta)\right|
	\\
	&\qquad
	\leq \sum_{i+j=m}\sum_{|\delta|=i}\frac{1}{i!}\binom{i}{\delta}C^{|\alpha+\beta|+2|\delta|+j+1} j!^{2a+2\eps}(\alpha+(0,\delta,\delta))!^{a+\eps}\beta!^{a+\eps}\lambda_{b,|\beta|}(x,y,z)
	\\
	&\qquad
	\leq \sum_{i+j=m}\frac{n^i}{i!}2^{(a+\eps)(|\alpha|+2i)}C^{|\alpha+\beta|+2i+j+1} j!^{2a+2\eps}i!^{2a+2\eps}\alpha!^{a+\eps}\beta!^{a+\eps}\lambda_{b,|\beta|}(x,y,z)
	\\		
	&\qquad\leq C^{m+1+|\alpha+\beta|} \alpha!^{a+\eps}\beta!^{a+\eps}m!^{2a+2\eps}\lambda_{b,|\beta|}(x,y,z),
	\end{align*}
	where $C$ is a constant independent to $m$, which may vary from line to line, and $b=b(a_m)$ is also independent to $m$. So we obtain the result on $(Sa)_m(x,y,\theta)$. Note $\theta(x,y,z)\in \mathcal{A}^{a,\eps}_z$. The result on $(Sa)_m(x,y,z)$ follows by Lemma \ref{AHE Composition} and keeping track of these constants.  
\end{proof}

Next, we will estimate the growth of $(SA)_m(x,y,z)$ and $(SA)_m(x,y,\theta)$.
\begin{lemma}\label{SAm}
	For each integer $m\geq 0$, $(SA)_m(x,y,\theta)\in \mathcal{A}^{a,\eps}_\theta$ and $(SA)_m(x,y,z)\in \mathcal{A}^{a,\eps}_z$. And we can choose
	\begin{align*}
	C_0((SA)_m(x,y,\theta))= C^{m+1}m!^{2a+2\eps}, && C_1((SA)_m(x,y,\theta))=C, &&b((SA)_m(x,y,\theta))=b,
	\end{align*}	
	\begin{align*}
	C_0((SA)_m(x,y,z))= C^{m+1}m!^{2a+2\eps}, && C_1((SA)_m(x,y,z))=C, &&b((SA)_m(x,y,z))=b,
	\end{align*}	
	where $C$ and $b$ are some positive constants independent of $m$.
\end{lemma}
\begin{proof}
	Recall $(SA)_0(x,y,z)=0$.
	By \eqref{def SA}, for $m\geq 1$, we have
	\begin{align*}
	(SA)_m(x,y,z)=-\int_0^1 D_y(Sa)_{m-1}(x,tx+(1-t)y,z)dt.
	\end{align*}
	By Lemma \ref{Sam}, we have $D_y(Sa)_{m-1}(x,y,z)\in \mathcal{A}^{a,\eps}_z$. Then by Lemma \ref{AHE Integral} and Lemma \ref{AHE Composition}, we have $(SA)_m(x,y,z)\in \mathcal{A}^{a,\eps}_z$ and $(SA)_m(x,y,\theta)\in \mathcal{A}^{a,\eps}_\theta$. The remaining part follows straightforwardly by keeping track of the constants.
\end{proof}

We define
\begin{equation}
d_m(x,y,z)=\int_0^1 D_{\bar{y}}(Sa)_{m}(x,tx+(1-t)y,z)dt,
\end{equation}
and 
\begin{equation}
d_m(x,y,\theta)=d_m(x,y,z(x,y,\theta)).
\end{equation}

Since $(Sa)_m\in \mathcal{A}_z^{a,\eps}$, we have $D_{\bar{y}}(Sa)_m\in \mathcal{I}_z^{a,\eps}$. By using Lemma \ref{AHE Integral} and Remark \ref{I Composition}, we can also similarly prove the following estimates on $d_m(x,y,z)$ and $d_m(x,y,\theta)$.
\begin{lemma}\label{dm}
	For each integer $m\geq 0$, $d_m(x,y,\theta)\in \mathcal{I}^{a,\eps}_\theta$ and $d_m(x,y,z)\in \mathcal{I}^{a,\eps}_z$. And we can choose
	\begin{align*}
	C_0(d_m(x,y,\theta))= C^{m+1}m!^{2a+2\eps}, && C_1(d_m(x,y,\theta))=C, &&b(d_m(x,y,\theta))=b,
	\end{align*}	
	\begin{align*}
	C_0(d_m(x,y,z))= C^{m+1}m!^{2a+2\eps}, && C_1(d_m(x,y,z))=C, &&b(d_m(x,y,z))=b,
	\end{align*}	
	where $C$ and $b$ are some positive constants independent of $m$.	
\end{lemma}

Now that we have the estimates on $(SA)_m$ in Lemma \ref{SAm}, by using the operator $S^{-1}$, we obtain the following estimates on $A_m$.
\begin{lemma}\label{Am}
	For each integer $m\geq 0$, $A_m(x,y,\theta)\in \mathcal{A}^{a,\eps}_\theta$ and $A_m(x,y,z)\in \mathcal{A}^{a,\eps}_z$. And we can choose
	\begin{align*}
	C_0(A_m(x,y,\theta))= C^{m+1}m!^{2a+2\eps}, && C_1(A_m(x,y,\theta))=C, &&b(A_m(x,y,\theta))=b,
	\end{align*}	
	\begin{align*}
	C_0(A_m(x,y,z))= C^{m+1}m!^{2a+2\eps}, && C_1(A_m(x,y,z))=C, && b(A_m(x,y,z))=b,
	\end{align*}	
	where $C$ and $b$ are some positive constants independent of $m$.
\end{lemma}

Formally, we write $d=\sum_{m=0}^\infty \frac{d_m}{k^m}$. Similarly, by using Lemma \ref{dm} and the operator $S^{-1}$, we obtain the estimates on $(S^{-1}d)_m$.
\begin{lemma}\label{S-1dm}
	For each integer $m\geq 0$, $(S^{-1}d)_m(x,y,\theta)\in \mathcal{I}^{a,\eps}_\theta$ and $(S^{-1}d)_m(x,y,z)\in \mathcal{I}^{a,\eps}_z$. And we can choose
	\begin{align*}
	C_0((S^{-1}d)_m(x,y,\theta))= C^{m+1}m!^{2a+2\eps}, && C_1((S^{-1}d)_m(x,y,\theta))=C, &&b((S^{-1}d)_m(x,y,\theta))=b,
	\end{align*}	
	\begin{align*}
	C_0((S^{-1}d)_m(x,y,z))= C^{m+1}m!^{2a+2\eps}, && C_1((S^{-1}d)_m(x,y,z))=C, && b((S^{-1}d)_m(x,y,z))=b,
	\end{align*}	
	where $C$ and $b$ are some positive constants independent of $m$.
\end{lemma}

Since the proof of Lemma \ref{Am} and \ref{S-1dm} follow in the same way as that of Lemma \ref{Sam}, we omit them here. 

We are now ready to estimate $A_N$, $A^{(N)}$ and $D_\theta \cdot A_N$ on the good contour $\Lambda=\{(y,\theta(x,y,\bar{y})): x,y\in B^n(0,1)\}$. For any smooth function $f(x,y,\theta)$, we denote $$\|f(x,y,\theta)\|_{L^\infty(\Lambda)}:=\|f(x,y,\theta(x,y,\bar{y}))\|_{L^\infty (B^{n}(0,1) \times B^{n}(0,1))}.$$

\begin{lemma}\label{error term 1 for Gevrey}
	We have
	\begin{equation}
	\|A^{(N)}(x,y,\theta(x, y, \bar y))\|_{L^\infty(\Lambda)}\leq C k + \frac{C^{N} N!^{2a+2\eps}}{k^N},
	\end{equation}
	where $C$ is some constant independent of $N$ and $k$. 
\end{lemma}

\begin{proof} 
	Note that by the definition of $A^{(N)}$  and estimates on each $A_m$ in Lemma \ref{Am}, we have
	\begin{align*}
	\|A^{(N)}\|_{L^\infty (\Lambda  )} & \leq \frac{C1!^{2a+2\eps}}{k}+\frac{C^2 2! ^{2a+2\eps}}{k^2}+\cdots+\frac{C^N N! ^{2a+2\eps}}{k^N}.\\
	\end{align*}
	We need to study the function $ \frac{C^xx^{(2a+2\eps)x}}{e^{(2a+2\eps)x}k^x}$. To find the minimum of this function we consider
	\begin{equation*}
	f(x)=\log \frac{C^xx^{(2a+2\eps)x}}{e^{(2a+2\eps)x}k^x}=x\log C+(2a+2\eps)x\log x-(2a+2\eps)x-x\log k \hspace{12 pt}, \mbox{ for } x\in (0,\infty).
	\end{equation*}
	Since
	\begin{equation*}
	f'(x)=\log C+(2a+2\eps)\log x-\log k,
	\end{equation*}
	the only critical point of $f$ is $x_0=(\frac{k}{C})^{\frac{1}{2a+2\eps}}$, and the function $f$ is decreasing on the interval $(0, x_0]$ and increasing on the interval $[x_0, \infty)$. Hence if we take $N_0=[(\frac{k}{C})^{\frac{1}{2a+2\eps}}]$, then by using Stirling's formula twice
	\begin{align*}
	\sum_{m=1}^N \frac{C^m m! ^{2a+2\eps}}{k^m} 
	& \leq C'  \sum_{m=1}^N m^{a+\eps}\frac{C^m m^{m(2a+2\eps)}}{e^{m(2a+2\eps)}k^m} 
	\\
	& \leq C' \left ( N_0^{a+\eps+1} + N^{a+\eps}(N-N_0)\frac{C^N N^{(2a+2\eps)N}}{e^{(2a+2\eps)N}k^N} \right ) 
	\\ 
	& \leq C' \left (\left(\frac{k}{C}\right)^{\frac{a+\eps+1}{2a+2\eps}} + N^{a+\eps+1}\frac{C^N N^{(2a+2\eps)N}}{e^{(2a+2\eps)N}k^N} \right )
	\\
	&  \leq C' \left (k +  \frac{C''^N N!^{2a+2\eps}}{k^N} \right ).
	\end{align*}
	The result follows by replacing $C'$ and $C''$ by a larger constant $C$.
\end{proof}
We also need the estimates on  the anti-holomorphic derivatives of $A_N$ and $A^{(N)}$. 
\begin{lemma}\label{error term 4 for Gevrey}
	There exists positive constants $C$ and $b$ independent of $N$ and $k$ such that for any $(x,y) \in B^{n}(0,1)\times B^{n}(0,1)$, we have
	\begin{equation*}
	\left | \left ( D_{\bar{\theta}}\cdot A^{(N)} \right ) (x,y,\theta(x, y, \bar y) \right | \leq \left(C k + \frac{C^{N}N!^{2a+2\eps}}{k^N}\right)\exp\left(-b{|x-y|^{-\frac{1}{a-1}}}\right),
	\end{equation*}
	\begin{equation*}
	\left | \left ( D_{\bar{y}}\cdot A^{(N)} \right ) (x,y,\theta(x, y, \bar y) \right | \leq \left(C k + \frac{C^{N}N!^{2a+2\eps}}{k^N}\right)\exp\left(-b{|x-y|^{-\frac{1}{a-1}}}\right).
	\end{equation*}
\end{lemma}

We omit the proof of this lemma here since it follows in a similar way as the previous lemma by using Lemma \ref{Am} and the only difference is the extra exponential factor that comes from the anti-holomorphic derivatives of $A_m$ since $A_m(x,y,\theta)\in \mathcal{A}_\theta^{a,\eps}$.

Another key lemma is:
\begin{lemma}\label{error term 3 for Gevrey} 
	There exists positive constants $C$ and $b$ independent of $N$ and $k$, such that for any $(x,y) \in B^{n}(0,1)\times B^{n}(0,1)$, we have
	\begin{equation*} 
	\left|\left(S^{-1}\int_0^1 (D_{\bar{y}}Sa)(x,tx+(1-t)y,z)dt\right)_m  (x,y,\bar y) \right|
	\leq C^{m+1}m!^{2a+2\eps}\exp\left(-b{|x-y|^{-\frac{1}{a-1}}}\right),
	\end{equation*}
	and
	\begin{align*}
	\left|\left(S^{-1}\int_0^1 (D_{\bar{y}}Sa)(x,tx+(1-t)y,z)dt\right)^{(N)} (x, y,\bar{y})\right| 
	\leq 
	\left(C k + \frac{C^{N}N!^{2a+2\eps}}{k^N}\right)\exp\left(-b{|x-y|^{-\frac{1}{a-1}}}\right).
	\end{align*}
\end{lemma}
\begin{proof}
	The first inequality directly follows from Lemma \ref{S-1dm} and
	the second inequality follows by the same argument as in the proof of Lemma \ref{error term 1 for Gevrey}.
\end{proof}

Recalling \eqref{a and nabla A} and using Lemma \ref{Am} and \ref{error term 3 for Gevrey} together, we obtain the following corollary.
\begin{cor}\label{prop 7.6}
	There exists positive constants $C$ and $b$ independent of $N$, such that for any $(x,y) \in B^{n}(0,1)\times B^{n}(0,1)$, we have	
	\begin{align*}
	\left|a^{(N)}-\nabla \left(A^{(N+1)}\right)\right|(x,y,\bar{y})
	\leq \frac{C^{N+1}(N+1)!^{2a+2\eps}}{k^{N+1}}
	+|x-y|\ \left(C k + \frac{C^{N}N!^{2a+2\eps}}{k^N}\right)\exp\left(-b|x-y|^{-\frac{1}{a-1}}\right).
	\end{align*}
\end{cor}

Now we are ready to prove Proposition \ref{LocalBergmanKernel}.

We claim that
\begin{equation}\label{LocalBergmanKernel2}
u(x)=\text{Op}_\Lambda \left (1+a^{(N)} \right ) ( \chi u)
+k^n\mathcal R_{N+1}( \phi, k) e^{\frac{k\phi(x)}{2}}\|u\|_{k\phi},
\end{equation} where uniformly for $x \in B(0, \frac14 )$ we have
\begin{equation}\label{R}
|\mathcal R_{N+1}( \phi, k)| \leq\frac{C^{N+1}(N+1)!^{2a+2\eps}}{k^{N+1}},
\end{equation}
and the integral kernel of $\text{Op}_\Lambda \left (1+a^{(N)} \right )$  is almost holomorphic. The complex conjugate of this kernel is given by 
\begin{equation}\label{KandB}
\oo{K^{(N)}_{k,x}(y)}=\left(\frac{k}{\pi}\right)^n  e^{k\psi(x,\bar{y})}\left(1+ a^{(N)}(x, y, \theta(x, y, \bar y)\right) \Delta_0(x, y, \theta(x, y, \bar y))^{-1},
\end{equation}
which by the relation \eqref{aB} is reduced to 
$$\oo{K^{(N)}_{k,x}(y)}=\left(\frac{k}{\pi}\right)^n  e^{k\psi(x,\bar{y})} B^{(N)}(x, \bar y). $$
Hence  ${K^{(N)}_{k,x}(y)}$ is almost holomorphic in $y$ because $B(x, z)$ is almost holomorphic. 

In the light of \eqref{BBS 2.1}, to prove \eqref{LocalBergmanKernel2} it suffices to show that
$$\forall x \in B(0, \frac14): \quad \left | \text{Op}_\Lambda \left (a^{(N)} \right )( \chi u)(x)  \right | \leq \frac{C^{N+1}(N+1)!^{2a+2\eps}}{k^{N+1-n}} e^{\frac{k\phi(x)}{2}}\|u\|_{k\phi}. $$
By definition,
$$\text{Op}_\Lambda \left (a^{(N)} \right )( \chi u)(x)= c_n \left(\frac{k}{2\pi}\right)^n\int_\Lambda e^{k\theta\cdot (x-y)}u(y)\chi(y)\, a^{(N)} \, d\theta\wedge dy . $$
It is easy to see that using integration by parts (see for example the proof of Proposition 2.2 in \cite{BBS}), we get
\begin{align*}\label{a^(N)}
\int_\Lambda &e^{k\theta\cdot (x-y)}u(y)\chi(y)a^{(N)}d\theta\wedge dy=
\\
&-\int_\Lambda d\chi\wedge u(y)e^{k\theta\cdot(x-y)}A^{(N+1)}\wedge dy
+\int_\Lambda e^{k\theta\cdot (x-y)}u(y)\chi(y)\left(a^{(N)}-\nabla \left(A^{(N+1)}\right) \right)d\theta\wedge dy
\\
&-\sum_{i,j}\int_{\Lambda} e^{k\theta\cdot (x-y)}u(y)\chi(y)\frac{\partial A_i}{\partial \oo{\theta}_j}d\oo{\theta}_j\wedge\widehat{d\theta_i}\wedge dy
-\sum_{i,j}\int_{\Lambda} e^{k\theta\cdot (x-y)}u(y)\chi(y)\frac{\partial A_i}{\partial \oo{y}_j}d\oo{y}_j\wedge\widehat{d\theta_i}\wedge dy
\end{align*}
In the first integral, we have  identified  the $n$-vector $A$ as an $(n-1, 0)$ form defined by $ A= \sum_{j=1}^n A_j \widehat{d\theta_j}$, where $\widehat{d\theta_j}$ is the wedge product of all $\{d \theta_k\}_{k \neq j}$ such that $d\theta_j \wedge \widehat{d\theta_j} = d\theta$. 

We now estimate the integrals on the right hand side of the above equality. For the first integral, as $d\chi(y)=0$ for $y\in B^{n}(0,\frac{1}{2})$, we have $|x-y|\geq \frac{1}{4}$ for $x\in B^n(0,\frac{1}{4})$ or otherwise the integrand vanishes. If we take $\theta^*(x,y,z)=\int_0^1( D_{\bar{y}}\psi)(tx+(1-t)y,z)dt$, then by Taylor expansion we have
\begin{equation*}
2\Real\left(\theta\cdot (x-y)\right)+2\Real\left(\theta^*\cdot (x-y)\right)=2\Real\left(\psi(x,\bar{y})-\psi(y,\bar{y})\right)\leq \phi(x)-\phi(y)-\delta|x-y|^2,
\end{equation*} 
where $	\theta^*=\theta^*(x,y,\bar{y})$ and $\delta$ is some positive constant. Note that $\theta^*(x,y,z)\in \mathcal{I}_z^{a,\eps}$ by Lemma \ref{AHE Integral}. We have $\theta^*(x,y,\bar{y})=O(|x-y|^{\infty})$. Thus by rescaling the unit ball, $\theta^*(x,y,\bar{y})$ can be absorbed by $\delta|x-y|^2$. Therefore, by changing $\delta$ to a smaller constant, the integrand of the first integral is bounded by some constant times
\begin{equation*}
|u(y)|e^{\frac{k\phi(x)}{2}-\frac{k\phi(y)}{2}-\delta k}\left|A^{(N+1)}\right|.
\end{equation*}
So by using Cauchy-Schwartz inequality, we obtain the first integral is bounded by some constant times
\begin{equation*}
\|A^{(N+1)}\|_{L^\infty(\Lambda)}e^{\frac{k\phi(x)}{2}} \|u\|_{k\phi}e^{-\delta k}.
\end{equation*}
By Lemma \ref{error term 1 for Gevrey} and $ke^{-\delta k}\leq \frac{(N+2)!}{k^{N+1}\delta^{N+2}}$, the first integral is bounded by
$$\frac{C^{N+1}(N+1)!^{2a+2\eps}}{k^{N+1}}e^{\frac{k\phi(x)}{2}} \|u\|_{k\phi}.$$

For the second term, the integrand is bounded by some constant times 
\begin{equation*}
|u(y)|e^{\frac{k\phi(x)}{2}-\frac{k\phi(y)}{2}-k\delta|x-y|^2}\left|a^{(N)}-\nabla \left(A^{(N+1)}\right)\right|.
\end{equation*}
By Corollary \ref{prop 7.6}, we have
\begin{align*}
e^{-k\delta|x-y|^2}&\left|a^{(N)}-\nabla \left(A^{(N+1)}\right) \right|
\\
\leq& \frac{C^{N+1}(N+1)!^{2a+2\eps}}{k^{N+1}}+e^{-k\delta|x-y|^2}  \left(C k + \frac{C^{N}N!^{2a+2\eps}}{k^N}\right)\exp\left(-b|x-y|^{-\frac{1}{a-1}}\right).
\end{align*}
Note  for any positive integer $M$, 
$$\exp\left(-b|x-y|^{-\frac{1}{a-1}}\right)\leq \left(\frac{a-1}{b}\right)^{M(a-1)}M!^{a-1}|x-y|^{M}.$$
Take $M=2N+4$. Since for any $x,y\in B^n(0,1)$,
$$e^{-k\delta|x-y|^2} |x-y|^{2N+4}\leq \frac{(N+2)!}{(\delta k)^{N+2}},$$
which implies
\begin{align*}
ke^{-k\delta|x-y|^2}\exp\left(-b|x-y|^{-\frac{1}{a-1}}\right)
\leq&  \left(\frac{(a-1)^{2a-2}}{b^{2a-2}\delta}\right)^{N+2}(2N+4)!^{a-1}\frac{(N+2)!}{k^{N+1}}\\
\leq&  \left(\frac{4^{a-1+\eps}(a-1)^{2a-2}}{b^{2a-2}\delta}\right)^{N+2}\frac{(N+2)!^{2a-1}}{k^{N+1}}\\
\leq& \frac{C^{N+1}(N+1)!^{2a-1}}{k^{N+1}},
\end{align*}
for some constant $C$ independent to $N$.  So it is easy to see that by choosing a larger constant $C$, we have
$$e^{-k\delta|x-y|^2}\left|a^{(N)}-\nabla \left(A^{(N+1)}\right) \right|
\leq \frac{C^{N+1}(N+1)!^{2a+2\eps}}{k^{N+1}}.$$
Therefore, the second integral is also bounded by 
$\frac{C^{N+1}(N+1)!^{2a+2\eps}}{k^{N+1}}e^{\frac{k\phi(x)}{2}} \|u\|_{k\phi}$.

For the third and fourth terms, the integrands are respectively bounded by some constant times 
\begin{align*}
|u(y)|e^{\frac{k\phi(x)}{2}-\frac{k\phi(y)}{2}-k\delta|x-y|^2}\left|D_{\bar{\theta}}\cdot A^{(N+1)}\right|,\quad
|u(y)|e^{\frac{k\phi(x)}{2}-\frac{k\phi(y)}{2}-k\delta|x-y|^2}\left|D_{\bar{y}}\cdot A^{(N+1)}\right|.
\end{align*}

By Lemma \ref{error term 4 for Gevrey}, we have
\begin{align*}
e^{-k\delta|x-y|^2}\left|D_{\bar{\theta}}\cdot A^{(N+1)}\right|
\leq e^{-k\delta|x-y|^2}  \left(C k + \frac{C^{N+1}(N+1)!^{2a+2\eps}}{k^{N+1}}\right)\exp\left(-b|x-y|^{-\frac{1}{a-1}}\right),
\end{align*}
and
\begin{align*}
e^{-k\delta|x-y|^2}\left|D_{\bar{y}}\cdot A^{(N+1)}\right|
\leq e^{-k\delta|x-y|^2}  \left(C k + \frac{C^{N+1}(N+1)!^{2a+2\eps}}{k^{N+1}}\right)\exp\left(-b|x-y|^{-\frac{1}{a-1}}\right),
\end{align*}
By the same argument as estimating the second term, we have
\begin{align*}
ke^{-k\delta|x-y|^2}\exp\left(-b|x-y|^{-\frac{1}{a-1}}\right)
\leq \frac{C^{N+1}(N+1)!^{2a-1}}{k^{N+1}}.
\end{align*}
By choosing a larger constant $C$,
$$e^{-k\delta|x-y|^2}\left|D_{\bar{\theta}}\cdot A^{(N+1)}\right|
\leq \frac{C^{N+1}(N+1)!^{2a+2\eps}}{k^{N+1}}.$$
Therefore, the third and fourth integral are also bounded by 
$\frac{C^{N+1}(N+1)!^{2a+2\eps}}{k^{N+1}}e^{\frac{k\phi(x)}{2}} \|u\|_{k\phi}$ and the result follows.

\section{From local to global and the proof of Theorem \ref{Main}}\label{Sec Global}
Let $K_k(x,y)$ be the Bergman kernel of $(M,L^k)$. As we noted before, we also write $K_k(x, y)$ for the representation of the Bergman kernel in the local frame $e^k_L \otimes \oo{e^k_L}$ and we denote $K_{k,y}(x):= K_k(x,y)$. In the last section, we constructed the local Bergman kernel of order $N$, which we denoted by $K_{k}^{(N)}(x,y)=K_{k,y}^{(N)}(x)$. In this section, we show that $K_k(x,y)$ is equal to $K_{k}^{(N)}(x,y)$ up to order $k^{-N}$ when $x,y$ are sufficiently close to each other. Moreover, we will give a precise upper bound for the error term.
\begin{prop}\label{LocalGlobal} There exists $ \delta >0$ such that whenever $d(x,y) < \delta$, we have
\begin{equation} \label{Cx}
K_k(x,y)=K_k^{(N)}(x,y)+ k^{\frac{3n}{2}} \widetilde{\mathcal R}_{N+1}(\phi, k) e^{\frac{k\phi(x)}{2}+\frac{k\phi(y)}{2}},
\end{equation}
where
\begin{equation}\label{C}
| \widetilde{\mathcal R}_{N+1}(\phi, k)| \leq \frac{C^{N+1}(N+1)!^{2a+2\eps}}{k^{N+1}}, 
\end{equation}
and the constant $C$ is independent of $N$, $x$, $y$, and $k$. 
\end{prop}

\begin{proof}
We fix $x \in M$ and assume that $\phi$ is in Gevrey $a$ class in $B^n(x, 3)$. Let $\chi$ be a  smooth cut-off function such that
\begin{equation*}
    \chi(z)=
    \begin{cases}
    1 & z\in B^{n}(x,\frac{1}{2})\\
    0 & z\notin B^{n}(x,\frac{3}{4})
    \end{cases}.
\end{equation*}
We assume $y\in B^n(x,\frac{1}{4})$.
We first observe that
\begin{equation}\label{S}
K_k(y,x)=\left(\chi K_{k,x},K^{(N)}_{k,y}\right)_{k\phi}+ \mathcal S_{N+1} (\phi, k) k^{\frac{3n}{2}}e^{k(\frac{\phi(x)}{2}+\frac{\phi(y)}{2})},
\end{equation}
where $|\mathcal S_{N+1} (\phi, k)|\leq \frac{C^{N+1}(N+1)!^{2a+2\eps}}{k^{N+1}}$.
This is because, by Proposition \ref{LocalBergmanKernel}, we have
\begin{align*}
K_{k,x}(y)=\left(\chi K_{k,x}, K^{(N)}_{k,y}\right)_{k\phi}+k^n \mathcal S_{N+1} (\phi, k)   e^{\frac{k\phi(y)}{2}}\|K_{k,x}\|_{k\phi},
\end{align*}
and by the reproducing property of Bergman kernel, we have
\begin{equation*}
\|K_{k,x}\|_{k\phi} e^{-k\frac{\phi(x)}{2}} \leq  \left |  \|K_{k,x}\|_{L^2(M,L^k)} \right |_{h^k}= \sqrt{| K_k(x,x)|_{h^k}}\leq Ck^{\frac{n}{2}}.
\end{equation*}
That why $| K_k(x,x)|_{h^k} \leq Ck^n$ follows from the extreme property of the Bergman function and also the sub-mean value inequality. For a simple proof see for example Lemma 4.1 of \cite{HKSX}.

Next, we define\begin{equation} \label{u}
u_{k,y}(z)= \chi(z) K_{k,y}^{(N)}(z)-\left(\chi K^{(N)}_{k,y}, K_{k,z}\right)_{k\phi}.
\end{equation} Our goal is to estimate $|u_{k, y}(x)|$. 
Since $\left(\chi K^{(N)}_{k,y}, K_{k,x}\right)_{k\phi}$ is the Bergman projection of $\chi K_{k,y}^{(N)}$, $u_{k,y}$ is the minimal $L^2$ solution to the equation
\begin{equation*}
    \dbar u=\dbar(\chi K_{k,y}^{(N)}).
\end{equation*}
So by using H\"ormander's $L^2$ estimates \cite{Ho} (see  \cite{B} for an exposition), we have
\begin{equation*}
\left\|u_{k,y}\right\|_{L^2}^2\leq \frac{C}{k}\left\|\dbar \left(\chi K_{k,y}^{(N)}\right)\right\|_{L^2}^2.
\end{equation*}
We have $\dbar(\chi K_{k,y}^{(N)})(z)=\dbar\chi(z) K_{k,y}^{(N)}(z)+\chi(z) \dbar K_{k,y}^{(N)}(z)$. Recall that by \eqref{KandB}
$${K^{(N)}_{k,y}(z)}=\left(\frac{k}{\pi}\right)^n e^{k\psi(z,\bar{y})}B^{(N)}(z,\bar{y}). $$ 
For the first term, since $\dbar\chi(z)$ is supported in $d(z,x)\geq \frac{1}{2}$, if we choose $d(x,y)\leq \frac{1}{4}$, then using
$$\Real \psi(z,\bar{y})\leq -\delta|z-y|^2+\frac{\phi(x)}{2}+\frac{\phi(y)}{2},$$
we have
$$\left|\dbar\chi(z) K_{k,y}^{(N)}(z)\right|\leq Ck^n e^{-k\delta+k\frac{\phi(y)}{2}+k\frac{\phi(z)}{2}}\|B^{(N)}\|_{L^{\infty}(U\times U)}.$$
We can estimate $\|B^{(N)}\|_{L^\infty(U \times U)}$ using our Theorem \ref{MainLemma}
\begin{align*}
\|B^{(N)}\|_{L^\infty ( U \times U)} 
& \leq 1+ \frac{1}{k} \|b_1\|_{L^\infty ( U\times U)} + \dots \frac{1}{k^N} \|b_N\|_{L^\infty (U\times U)} \\ 
& \leq 1+\frac{C1!^{2a+2\eps}}{k}+\frac{C^2 2! ^{2a+2\eps}}{k^2}+\cdots+\frac{C^N N! ^{2a+2\eps}}{k^N} \\
& \leq C\left(k+\frac{C^NN!^{2a+2\eps}}{k^N}\right) .
\end{align*}
Hence,
\begin{align*}
\left|\dbar\chi(z) K_{k,y}^{(N)}(z)\right|
\leq& k^ne^{-k\delta+k\frac{\phi(y)}{2}+k\frac{\phi(z)}{2}}\left(Ck+\frac{C^NN!^{2a+2\eps}}{k^N}\right)
\\
\leq& k^ne^{k\frac{\phi(y)}{2}+k\frac{\phi(z)}{2}}\frac{C^{N+1}(N+1)!^{2a+2\eps}}{k^{N+1}}.
\end{align*}
For the second term, we have
$$\left|\chi(z) \dbar K_{k,y}^{(N)}(z)\right|\leq Ck^n e^{-k|z-y|^2+k\frac{\phi(y)}{2}+k\frac{\phi(z)}{2}}\left(k\left|\dbar\psi(z,\bar{y})\right|\left|B^{(N)}(z,\bar{y})\right|+\left|\dbar B^{(N)}(z,\bar{y})\right|\right).$$
By using the fact that $\psi(y,z)\in \mathcal{A}^{a,\eps}_{diag}$ and Theorem \ref{MainLemma},
\begin{align*}
k\left|\dbar\psi(z,\bar{y})\right|\left|B^{(N)}(z,\bar{y})\right|+\left|\dbar B^{(N)}(z,\bar{y})\right|
\leq Ck\left(k+\frac{C^NN!^{2a+2\eps}}{k^N}\right)\exp(-b|z-y|^{-\frac{1}{a-1}}) .
\end{align*}
Then by the fact that
\begin{align*}
e^{-k\delta|x-y|^2}\exp\left(-b|x-y|^{-\frac{1}{a-1}}\right)
\leq \frac{C^{N+1}(N+1)!^{2a-1}}{k^{N+3}}, 
\end{align*}
we have
\begin{align*}
\left|\chi(z) \dbar K_{k,y}^{(N)}(z)\right|\leq k^ne^{k\frac{\phi(y)}{2}+k\frac{\phi(z)}{2}}\frac{C^{N+1}(N+1)!^{2a+2\eps}}{k^{N+1}}.
\end{align*}
So 
\begin{align}\label{dbarK}
\left|\dbar(\chi K_{k,y}^{(N)})(z)\right|
\leq k^ne^{k\frac{\phi(y)}{2}+k\frac{\phi(z)}{2}}\frac{C^{N+1}(N+1)!^{2a+2\eps}}{k^{N+1}},
\end{align}
and
\begin{align}\label{uL2norm}
\|u_{k,y}\|^2\leq \frac{C}{k} \|\dbar \left(\chi K_{k,y}^{(N)}\right)\|_{L^2}^2
\leq
C k^{2n}e^{k\phi(y)}\left(\frac{C^{N+1}(N+1)!^{2a+2\eps}}{k^{N+1}}\right)^2.
\end{align}
By using Bochner-Martinelli formula in a small Euclidean ball $B^n(x,r)$, we have
\begin{equation*}
r^{2n-1}|u_{k,y}(x)|\leq C\int_{\partial B^n(x,r)} |u(z)|dS+C\int_{B^n(x,r)}|\dbar u(z)|\frac{r^{2n-1}}{|z-x|^{2n-1}}dV_0,
\end{equation*}
where $dS$ and $dV_0$ are respectively the standard volume forms of $\partial B^n(0,1)$ and $B^n(0,1)$ in Euclidean space. If we use the Bochner coordinates at $x$, then $\phi(z)-\phi(x)=O(|z|^2)$, and thus
$$e^{k\left(\frac{\phi(z)}{2}-\frac{\phi(x)}{2}\right)}\leq C,\quad\mbox{for any } z\in B^n(x,\frac{1}{\sqrt{k}}).$$ 
By integrating the above inequality with respect to $r$ from $0$ to $\frac{1}{\sqrt{k}}$, we obtain
\begin{align*}
|u_{k,y}(x)|
\leq& Ck^n\int_{B^n(x,\frac{1}{\sqrt{k}})} |u(z)|dV_0+Ck^n\int_0^{\frac{1}{\sqrt{k}}}\int_{B^n(x,r)}|\dbar u(z)|\frac{r^{2n-1}}{|z-x|^{2n-1}}dV_0dr
\\
\leq& Ck^n\int_{B^n(x,\frac{1}{\sqrt{k}})} |u(z)|dV_0+C\|e^{-\frac{k\phi(z)}{2}}\dbar u(z)\|_{L^\infty}e^{\frac{k\phi(x)}{2}}\int_{B^n(x,\frac{1}{\sqrt{k}})}\frac{1}{|z-x|^{2n-1}}e^{k\left(\frac{\phi(z)}{2}-\frac{\phi(x)}{2}\right)}dV_0
\\
\leq &Ck^\frac{n}{2}e^{\frac{k\phi(x)}{2}}\left(\|u\|_{L^2}+\|e^{-\frac{k\phi(z)}{2}}\dbar u(z)\|_{L^\infty}\right).
\end{align*}

Therefore, by the estimates \eqref{dbarK} and \eqref{uL2norm}, it follows that
\begin{equation*}
|u_{k,y}(x)|\leq C k^{\frac{3n}{2}}e^{k\frac{\phi(y)}{2}+k\frac{\phi(x)}{2}}\left(\frac{C^{N+1}(N+1)!^{2a+2\eps}}{k^{N+1}}\right).
\end{equation*}
Combining this estimate with \eqref{S} and recalling the definition of $u_{k, y}$ in \eqref{u}, we get the result. 

We point out that we have renewed the constant $C$ at each step, but the final constant is independent of $k$ and $N$. We also note that the constant $C$ may depend on the point $x$, however by a simple compactness argument one can see that each such $C$ can be bounded by a uniform constant independent of $x$.
\end{proof}
Now we are ready to prove Theorem \ref{Main} and its corollaries. 

\subsection{Proof of Theorem \ref{Main}}
By Proposition \ref{LocalGlobal}, we just need to show that with $N=N_0-1= [ (k/C)^{\frac{1}{2a+2\eps}}]-1$, we have\footnote{For convenience, we use $N_0$ for $N_0(k)= {[(k/C)^{\frac{1}{2a+2\eps}}]}$. }
$$  k^{\frac{3n}{2}} \mathcal R_{N_0}(\phi, k) e^{\frac{k\phi(x)}{2}+\frac{k\phi(y)}{2}} = e^{k\left(\frac{\phi(x)}{2}+\frac{\phi(y)}{2}\right)}e^{-\delta k^{\frac{1}{2a+2\eps}}}O(1). $$
However, by the same proposition we know that 
$$ | R_{N_0}(\phi, k)| \leq \frac{C^{N_0}N_0!^{2a+2\eps}}{k^{N_0}}.$$ 
Hence it is enough to show that
$$ \frac{C^{N_0}N_0!^{2a+2\eps}}{k^{N_0}} =  e^{-\delta k^{\frac{1}{2a+2\eps}}} O(1).$$ 
By Stirling's formula,
\begin{equation*}
\frac{C^{N_0}N_0!^{2a+2\eps}}{k^{N_0}}
\leq C' N_0^{a+\eps}\frac{C^{N_0}{N_0}^{(2a+2\eps)N_0}}{e^{(2a+2\eps)N_0}k^{N_0}}
\leq C' N_0^{a+\eps}e^{-(2a+2\eps)N_0}
\leq C''k^{\frac{1}{2}}e^{- (2a+2\eps) (\frac{k}{C})^{\frac{1}{2a+2\eps}}}.
\end{equation*}
Since 
$$k^{\frac{3n}{2}+\frac12} e^{- (2a+2\eps) (\frac{k}{C})^{\frac{1}{2a+2\eps}}} \leq C''' e^{- (a+\eps) (\frac{k}{C})^{\frac{1}{2a+2\eps}}},$$
$\delta =\frac{a+\eps}{C^{\frac{1}{2a+2\eps}}}$ would do the job. 
\subsection{Proof of Corollary \ref{complete asymptotics}}
\begin{proof}
	By Theorem \ref{Main}, uniformly for any $x,y\in U$, we have 
	\begin{equation*}
	K_k(x,y)=e^{k\psi(x,\bar y)}\frac{k^n}{\pi^n}\left (1+\sum _{j=1}^{ N_0(k)-1}\frac{b_j(x, \bar y)}{k^j}\right )+e^{k\left(\frac{\phi(x)}{2}+\frac{\phi(y)}{2}\right)}e^{-\delta k^{\frac{1}{2a+2\eps}}}O(1).
	\end{equation*} 
For any given positive integer $N$, we rewrite the above formula as follows.
	\begin{equation*}
	K_k(x,y)=e^{-k\psi(x,\bar{y})}\frac{k^n}{\pi^n}\left(1+\sum_{j=1}^{N-1}\frac{b_j(x,\bar{y})}{k^j}+\sum_{j=N}^{N_0(k)-1}\frac{b_j(x,\bar{y})}{k^j}+e^{\frac{k}{2}\left(\phi(x)+\phi(y)-2\psi(x,\bar{y})\right)}e^{-\delta k^{\frac{1}{2a+2\eps}}}O(1)\right).
	\end{equation*}
	Our first observation is that, if $d(x,y)\leq \sqrt{\delta}k^{-\frac{1}{2}+\frac{1}{4a+4\eps}}$, then
	\begin{align*}
	\left|e^{\frac{k}{2}\left(\phi(x)+\phi(y)-2\psi(x,\bar{y})\right)}e^{-\delta k^{\frac{1}{2a+2\eps}}}\right|=e^{\frac{k}{2}D(x,y)-\delta k^{\frac{1}{2a+2\eps}}}
	\leq e^{-\frac{1}{4}\delta k^{\frac{1}{2a+2\eps}}}.
	\end{align*}
	Now we estimate the term $\sum_{j=N}^{N_0-1}\frac{b_j(x,\bar{y})}{k^j}$. By Stirling's formula, we have
	\begin{align*}
	\left|\frac{b_j(x,\bar{y})}{k^j}\right|\leq \frac{C^jj!^{2a+2\eps}}{k^j}\leq C'j^{a+\eps}\frac{C^jj^{(2a+2\eps)j}}{e^{(2a+2\eps)j}k^j}.
	\end{align*}
	Since $\frac{C^jj^{(2a+2\eps)j}}{e^{(2a+2\eps)j}k^j}$ is monotonically decreasing for $1\leq j\leq N_0(k)-1$ (with the help of Stirling's formula once more), we get
	\begin{align*}
	\left|\sum_{j=N}^{N_0-1}\frac{b_j(x,\bar{y})}{k^j}\right|
	\leq& \frac{C^NN!^{2a+2\eps}}{k^N}+\sum_{j=N+1}^{N_0-1}C'j^{a+\eps}\frac{C^jj^{(2a+2\eps)j}}{e^{(2a+2\eps)j}k^j}
	\\
	\leq &\frac{C^NN!^{2a+2\eps}}{k^N}+C'N_0^{a+\eps+1}\frac{C^{N+1}(N+1)^{(2a+2\eps)(N+1)}}{e^{(2a+2\eps)(N+1)}k^{N+1}}
	\\
	\leq &\frac{C''^{N}N!^{2a+2\eps}}{k^N}.
	\end{align*}
	Therefore,
	\begin{align*}
	K_k(x,y)
	=&e^{-k\psi(x,\bar{y})}\frac{k^n}{\pi^n}\left(1+\sum_{j=1}^{N-1}\frac{b_j(x,\bar{y})}{k^j}+\frac{C''^NN!^{2a+2\eps}}{k^N}+e^{-\frac{1}{4}\delta k^{\frac{1}{2a+2\eps}}}O(1)\right).
	\end{align*}
	By the fact that
	\begin{align*}
	e^{-\frac{1}{4}\delta k^{\frac{1}{2a+2\eps}}}
	\leq \left(\frac{4}{\delta}\right)^{(2a+2\eps)N}\frac{\left((2a+2\eps)N\right)^{(2a+2\eps)N}}{e^{(2a+2\eps)N}k^N}
	\leq \left(\frac{8a+8\eps}{\delta}\right)^{(2a+2\eps)N}\frac{N!^{2a+2\eps}}{k^N},
	\end{align*}
    the first part of our result follows.
    
    Now we prove the second part. Let $\widetilde{b}_m(x,z)$ be another almost holomorphic extension of $b_m(x,\bar{x})$. By Lemma \ref{AHE Uniqueness}, we have 
    \begin{equation*}
    	|b_m(x,\bar{y})-\widetilde{b}_m(x,\bar{y})|=O\left(|x-y|^\infty\right)=O\left(\frac{1}{k^\infty}\right).
    \end{equation*}
    The second equality follows from our assumption that $d(x,y)\leq \delta k^{-\frac{1}{2}+\frac{1}{4a+4\eps}}$. So the result follows.
\end{proof}

\begin{rmk}
	Let $\widetilde{b}_m$ be the almost holomorphic extension defined in Definition \ref{AHE}. If we take $\tilde{b}_m$ in Corollary \ref{complete asymptotics}, then \eqref{complete formula} and \eqref{complete error} hold. The reason is as follows.
	
	For any $M\in \mathbb{N}$, there exists $C_M$ such that
	\begin{equation*}
		\left|b_m(x,\bar{y})-\widetilde{b}_m(x,\bar{y})\right|\leq C_M|x-y|^{M+1}.
	\end{equation*}
	And $C_M$ depends on the sup norm of the all the $(M+1)$th derivatives of $b_m(x,\bar{y})$ and $\widetilde{b}_m(x,\bar{y})$. By Theorem \ref{MainLemma} and Lemma \ref{AHE2}, we have
	\begin{equation*}
		C_M\leq C^{m+M}m!^{2a+2\eps}M!^{a+\eps-1},
	\end{equation*}
	where $C$ is some positive constant independent of $m$. If we take $M=[(N-m)\frac{4a+4\eps}{2a+2\eps-1}]$, then $|x-y|^{M+1}\leq \frac{1}{k^{N-m}}$, whence by Stirling's formula we obtain
	\begin{align*}
		\frac{\left|b_m(x,\bar{y})-\widetilde{b}_m(x,\bar{y})\right|}{k^m}
		\leq& C^{m+M}m!^{2a+2\eps}M^{M(a+\eps-1)}k^{-N}\\
		\leq& C'^{N}m!^{2a+2\eps} (N-m)^{(N-m)(2a+2\eps)}k^{-N}\\
		\leq& C''^{N}m!^{2a+2\eps} (N-m)!^{2a+2\eps}k^{-N},
	\end{align*}
	where $C'$ and $C''$ are some positive constants independent of $m$. We rename $C''$ by $C$ and \eqref{complete formula}, \eqref{complete error} follows by
	\begin{align*}
		\sum_{j=1}^{N-1}\frac{\left|b_j(x,\bar{y})-\widetilde{b}_j(x,\bar{y})\right|}{k^j}
		\leq \sum_{j=1}^{N-1} \frac{C^NN!^{2a+2\eps}}{k^N}\leq \frac{(2C)^NN!^{2a+2\eps}}{k^N}.
	\end{align*} 
\end{rmk}

\subsection{Proof of Corollary \ref{Log}}
By Theorem \ref{Main}, we have
\begin{equation*}
K_k(x,y)=e^{k\psi(x, \bar y)}\frac{k^n}{\pi^n}\left(1+\sum _{j=1}^{N_0-1}\frac{b_j(x,\bar y)}{k^j}\right)+e^{\frac{k\phi(x)}{2}+\frac{k\phi(y)}{2}}e^{-\delta k^{\frac{1}{2a+2\eps}}} O(1).
\end{equation*}
Recall that $D(x,y)=\phi(x)+\phi(y)-\psi(x,\bar y)-\psi(y, \bar x)$.
Then
\begin{equation*}
\log|K_k(x,y)|_{h^k}=-\frac{kD(x,y)}{2}+n\log k -n\log \pi +\log\left|1+\sum _{j=1}^{N_0-1}\frac{b_j(x, \bar y)}{k^j}+e^{\frac{Q(x,y)}{2}k-\delta k^{\frac{1}{2a+2\eps}}}O(1)\right|,
\end{equation*}
where $Q(x,y)=\phi(x)+\phi(y)-2\psi(x,\bar{y})$.
So it is sufficient to prove
$$
\log\left|1+\sum _{j=1}^{N_0-1}\frac{b_j(x,\bar{y})}{k^j}+e^{\frac{Q(x,y)}{2}k-\delta k^{\frac{1}{2a+2\eps}}}O(1)\right|
=\log\left(1+O \left ( \frac{1}{k} \right )\right). $$
To do this we note that by our assumption $D(x, y) \leq  \frac{1}{2} \delta k^{-1+\frac{1}{2a+2\eps}}$, hence
\begin{equation*}
\left|e^{\frac{Q(x,y)}{2}k-\delta k^{\frac{1}{2a+2\eps}}}\right|=
e^{\frac{D(x,y)}{2}k-\delta k^{\frac{1}{2a+2\eps}}}
\leq e^{-\frac{3\delta}{4}k^{\frac{1}{2a+2\eps}}}.
\end{equation*}
It remains to show that
\begin{equation*}
\left|\sum _{j=1}^{N_0-1}\frac{b_j(x,\bar{y})}{k^j}\right|=O \left (\frac{1}{k} \right ).
\end{equation*}
By the estimates on $b_j(x,\bar{y})$ from Theorem \ref{MainLemma} and Stirling's formula, we have
$$\frac{|b_j(x,\bar{y})|}{k^j}\leq \frac{C^jj!^{2a+2\eps}}{k^j}\leq C'j^{a+\eps}\frac{C^j j^{(2a+2\eps)j}}{e^{2j}k^j}.$$
As shown in Lemma \ref{error term 1 for Gevrey}, the function $f(x)=\log \frac{C^xx^{(2a+2\eps)x}}{e^{2x}k^x}$ is decreasing on the interval $(0,(\frac{k}{C})^{\frac{1}{2a+2\eps}}]$, thus for $j\in [2,N_0-1]$,
$$\frac{|b_j(x,\bar{y})|}{k^j}\leq C'(N_0-1)^{a+\eps}\frac{C^2 2^{4a+4\eps}}{e^{4}k^2}\leq C''C^2\frac{N_0^{a+\eps}}{k^2}.$$
Therefore, 
\begin{equation*}
\left|\sum _{j=1}^{N_0-1}\frac{b_j(x,\bar{y})}{k^j}\right|\leq \frac{C}{k}+C''C^2\frac{N_0^{a+\eps+1}}{k^2}\leq \frac{C}{k}+C''C\frac{1}{k}=O \left (\frac{1}{k} \right ).
\end{equation*}

\section{Estimates on Bergman Kernel Coefficients}\label{Sec ProofofMainLemma}
As before, we assume the \k metric is in Gevrey class $G^a(U)$ for some neighborhood $U$ of $p$. We will estimate the growth rate of the Bergman kernel coefficients $b_m(x,z)$ as $m\rightarrow \infty$ for $x,z$ in $U$. Our goal is to prove Theorem \ref{MainLemma}. 

The key ingredient for the proof is the following recursive formula\footnote{We discussed its proof in \eqref{Recursive1}. } on $b_m(x,z)$ established in  \cite{BBS}.
\begin{equation}\label{Recursive}
	b_m(x,z(x,x,\theta))=-\sum_{l=1}^m\frac{(D_y\cdot D_\theta)^l}{l!}\big(b_{m-l}\left(x,z(x,y,\theta)\right)\Delta_0(x,y,\theta)\big)\Big|_{y=x}.
\end{equation}
We will break the proof of Theorem \ref{MainLemma} into two steps. The first step is to derive from the recursive formula \eqref{Recursive}, a recursive inequality on $\|D^\mu_zD^\nu_{\bar{z}}b_m(x,z)\|_{L^\infty(U\times U)}$ for some neighborhood $U$. The second step is to estimate $\|D^\mu_zD^\nu_{\bar{z}}b_m(x,z)\|_{L^\infty(U\times U)}$ by induction.

In the following we shall use the following standard notations for multi-indicies. 

\begin{itemize}
\item $ \mathbbm 1 = (1, 1, \cdots, 1)$.
\item $\binom{\alpha}{\beta}=\binom{\alpha_1}{\beta_1}\binom{\alpha_2}{\beta_2}\cdots \binom{\alpha_n}{\beta_n}$.
\item $\binom{l}{\delta_1,\delta_2,\cdots,\delta_n}=\frac{l!}{\delta_1!\delta_2!\cdots\delta_n!}$ for any non-negative integer $l$ and multi-index $\delta\geq 0$ such that $|\delta|=l$.
\end{itemize}

\begin{lemma}\label{bmLemmaGevrey}
	Suppose the \k potential $\phi\in G^a(U)$. Let $W=\{(x,z)\in U\times U: x\neq z\}$ and $$b=\min\{b(\Delta_0(x,y,\theta)),b(\psi_{x_i}(x,z)),b(z_i(x,y,\theta)): 1\leq i\leq n\}.$$ 
	If we denote $v=(x,z)$ and 
	\begin{equation*}
	b_{m,\mu\bar{\nu}}=\left\|\frac{1}{\lambda_{b,|\nu|}(x,x,z)}D_v^\mu D_{\oo{v}}^{\nu} b_m(x,z)\right\|_{L^\infty(W)},
	\end{equation*}
	then the exists some positive constant $C$ independent of $m,\mu,\nu$, such that
	\begin{align}\label{bmGevrey}
	\begin{split}
	b_{m,\mu\bar{\nu}}
	& \leq\sum_{l=1}^m
	\sum_{|\delta|=l}\delta!^{2a+2\eps-1}
	\sum_{\alpha,\beta\leq \delta}
	\sum_{|\xi+\eta|\leq |\alpha+\beta|}
	\sum_{\mu_0\leq \mu}\sum_{\nu_0\leq \nu} \frac{ b_{m-l,\widetilde{\xi}+\mu_0\oo{\widetilde{\eta}+\nu_0}}}{(\xi!\eta!)^{a+\eps}}
	C^{|\mu-\mu_0+\nu-\nu_0|+|\delta+\xi+\eta|}
	\\&
	\quad\quad \cdot 
	\binom{\mu}{\mu_0}\binom{\nu}{\nu_0}
	(\mu-\mu_0)!^{a+\eps}(\nu-\nu_0)!^{a+\eps},
	\end{split}
	\end{align}
	where $\xi,\eta\in (\mathbb{Z}^{\geq 0})^n$ and $\widetilde{\xi}=(0,\cdots 0,\xi), \widetilde{\eta}=(0, \cdots, 0,\eta)\in (\mathbb{Z}^{\geq 0})^{2n}$.
\end{lemma}

\begin{proof}
	We first work on $(D_y\cdot D_\theta)^l\Big(b_{m-l}\left(x,z\left(x,y,\theta\right)\right)\Delta_0\left(x,y,\theta\right)\Big)$. We expand $(D_y\cdot D_\theta)^l$ and obtain
	\begin{align}\label{DDl}
	\begin{split}
	(D_y&\cdot D_\theta)^l\big(b_{m-l}\left(x,z\left(x,y,\theta\right)\right)\Delta_0\left(x,y,\theta\right)\big)\\
	=&\sum_{|\delta|=l}\binom{l}{\delta_1,\delta_2,\cdots,\delta_n}
	D_y^\delta D_\theta^\delta
	\big(b_{m-l}(x,z(x,y,\theta))\Delta_0(x,y,\theta)\big)\\
	=&\sum_{|\delta|=l}\binom{l}{\delta_1,\delta_2,\cdots,\delta_n}
	\sum_{\alpha,\beta\leq \delta}\binom{\delta}{\alpha}\binom{\delta}{\beta}D_y^\alpha D_\theta^\beta
	\big(b_{m-l}(x,z(x,y,\theta))\big)D_y^{\delta-\alpha}D_\theta^{\delta-\beta}\Delta_0
	\\
	=& b_{m-l}(x,z(x,y,\theta))(D_y\cdot D_\theta)^l\Delta_0(x,y,\theta)
	\\&+
	\sum_{|\delta|=l}l!\delta!
	\sum_{\substack{\alpha,\beta\leq \delta\\ \alpha+\beta>0}}
	\sum_{1\leq|\xi+\eta|\leq |\alpha+\beta|} \frac{D_z^\xi D_{\bar{z}}^\eta b_{m-l}(x,z)}{\xi!\eta!}
	\frac{D_y^{\delta-\alpha}D_\theta^{\delta-\beta}\Delta_0}{(\delta-\alpha)!(\delta-\beta)!} 
	\sum_{A_{\alpha\beta\xi\eta}}
	\prod_{ij}\frac{D_y^{\alpha_{ij}}D_{\theta}^{\beta_{ij}}z_i}{\alpha_{ij}!\beta_{ij}!}
	\prod_{ik}\frac{D_y^{\alpha'_{ik}}D_{\theta}^{\beta'_{ik}}\oo{z_i}}{\alpha'_{ik}!\beta'_{ik}!},
	\end{split}
	\end{align}
	where the index set $A_{\alpha\beta\xi\eta}$ is defined by
	\begin{align}\label{Index Set}
	\begin{split}
	A_{\alpha \beta\xi\eta}& 
	\\
	=&\left \{ \{ \alpha_{ij},\beta_{ij} \}_{1\leq i \leq n, 1 \leq j \leq\xi_i},  \{ \alpha'_{ik}, \beta'_{ik}\}_{1\leq i \leq n, 1 \leq k \leq\eta_i}: \quad  \begin{array}{ll} \sum_{1\leq i \leq n, 1 \leq j \leq \xi_i} \alpha_{ij}+\sum_{1\leq i \leq n, 1 \leq k \leq \eta_i} \alpha'_{ik}=\alpha, \\ \sum_{1\leq i \leq n, 1 \leq j \leq \xi_i} \beta_{ij}+\sum_{1\leq i \leq n, 1 \leq k \leq \eta_i} \beta'_{ik}=\beta,  \\ \alpha_{ij}+\beta_{ij}>0,\quad \alpha'_{ik}+\beta'_{ik}>0  \end{array}  \right \}.
	\end{split}
	\end{align}
	
	We now substitute \eqref{DDl} into equation \eqref{Recursive} and obtain
	\begin{align*}
	b_m&(x,z(x,x,\theta))\\
	&=-\sum_{l=1}^m
	\Bigg(\frac{1}{l!}b_{m-l}(x,z(x,x,\theta))(D_y\cdot D_\theta)^l\Delta_0(x,x,\theta)+\sum_{|\delta|=l}\delta!
	\sum_{\substack{\alpha,\beta\leq \delta\\ \alpha+\beta>0}}\sum_{1\leq|\xi+\eta|\leq |\alpha+\beta|} \frac{D_z^\xi D_{\bar{z}}^\eta b_{m-l}(x,z(x,x,\theta))}{\xi!\eta!}
	\\&	
	\qquad \cdot\frac{D_y^{\delta-\alpha}D_\theta^{\delta-\beta}\Delta_0}{(\delta-\alpha)!(\delta-\beta)!} (x,x,\theta)
	\sum_{A_{\alpha\beta\xi\eta}}
	\prod_{ij}\frac{D_y^{\alpha_{ij}}D_{\theta}^{\beta_{ij}}z_i}{\alpha_{ij}!\beta_{ij}!}(x,x,\theta)
	\prod_{ik}\frac{D_y^{\alpha'_{ik}}D_{\theta}^{\beta'_{ik}}\oo{z_i}}{\alpha'_{ik}!\beta'_{ik}!}(x,x,\theta)\Bigg).
	\end{align*}
	
	The correspondence $(x,x,z)\leftrightarrow (x,x,\theta=\psi_x(x,z))$, turns this into
	\begin{align*}
	b_m&(x,z)\\
	&=-\sum_{l=1}^m
	\Bigg(\frac{1}{l!}b_{m-l}(x,z)(D_y\cdot D_\theta)^l\Delta_0(x,x,\psi_x(x,z))+\sum_{|\delta|=l}\delta!
	\sum_{\substack{\alpha,\beta\leq \delta\\ \alpha+\beta>0}}\sum_{1\leq|\xi+\eta|\leq |\alpha+\beta|} \frac{D_z^\xi D_{\bar{z}}^\eta b_{m-l}(x,z)}{\xi!\eta!}
	\\&	
	\qquad \cdot\frac{D_y^{\delta-\alpha}D_\theta^{\delta-\beta}\Delta_0}{(\delta-\alpha)!(\delta-\beta)!} (x,x,\psi_x(x,z))
	\sum_{A_{\alpha\beta\xi\eta}}
	\prod_{ij}\frac{D_y^{\alpha_{ij}}D_{\theta}^{\beta_{ij}}z_i}{\alpha_{ij}!\beta_{ij}!}(x,x,\psi_x(x,z))
	\prod_{ik}\frac{D_y^{\alpha'_{ik}}D_{\theta}^{\beta'_{ik}}\oo{z_i}}{\alpha'_{ik}!\beta'_{ik}!}(x,x,\psi_x(x,z))\Bigg).
	\end{align*}
	
	Denote $v=(x,z)$. Note that in this recursive formula, the coefficients $b_m$ depend on not only the previous coefficients $b_{m-l}$ but also derivatives of $b_{m-l}$. 
	Hence, we need to include derivatives of $b_m$ in our inductive argument. To do this we apply $D_v^{\mu}D_{\oo{v}}^{\nu}$ on both sides and obtain a recursive formula for the derivatives of $b_m$.
	\begin{align}\label{bmrecursive}
	\begin{split}
	D_v^{\mu}&D_{\oo{v}}^{\nu} b_m(x,z)
	\\
	=&-\sum_{l=1}^m\sum_{\mu_0\leq \mu}\sum_{\nu_0\leq \nu}\binom{\mu}{\mu_0}\binom{\nu}{\nu_0}
	\Bigg\{\frac{1}{l!}D_v^{\mu_0}
	D_{\oo{v}}^{\nu_0}b_{m-l}(x,z)D_v^{\mu-\mu_0} D_{\oo{v}}^{\nu-\nu_0}\left((D_y\cdot D_\theta)^l\Delta_0(x,x,\psi_x(x,z))\right)
	\\
	&+\sum_{|\delta|=l}\delta!
	\sum_{\substack{\alpha,\beta\leq \delta\\ \alpha+\beta>0}}\sum_{1\leq|\xi+\eta|\leq |\alpha+\beta|} \frac{D_v^{\mu_0+\widetilde{\xi}} D_{\oo{v}}^{\nu_0+\widetilde{\eta}} b_{m-l}(x,z)}{\xi!\eta!} 
	\\&	 \quad \cdot D_v^{\mu-\mu_0}D_{\oo{v}}^{\nu-\nu_0} \left(\frac{D_y^{\delta-\alpha}D_\theta^{\delta-\beta}\Delta_0}{(\delta-\alpha)!(\delta-\beta)!} (x,x,\psi_x(x,z))
	\sum_{A_{\alpha\beta\xi\eta}}
	\prod_{ij}\frac{D_y^{\alpha_{ij}}D_{\theta}^{\beta_{ij}}z_i}{\alpha_{ij}!\beta_{ij}!}
	\prod_{ik}\frac{D_y^{\alpha'_{ik}}D_{\theta}^{\beta'_{ik}}\oo{z_i}}{\alpha'_{ik}!\beta'_{ik}!}(x,x,\psi_x(x,z))\right)\Bigg\}.
	\end{split}
	\end{align}
	
	Now we will estimate the factor $$D_v^{\mu-\mu_0} D_{\oo{v}}^{\nu-\nu_0}\left(\frac{D_y^{\delta-\alpha}D_\theta^{\delta-\beta}\Delta_0}{(\delta-\alpha)!(\delta-\beta)!} (x,x,\psi_x(x,z))
	\sum_{A_{\alpha\beta\xi\eta}}
	\prod_{ij}\frac{D_y^{\alpha_{ij}}D_{\theta}^{\beta_{ij}}z_i}{\alpha_{ij}!\beta_{ij}!}
	\prod_{ik}\frac{D_y^{\alpha'_{ik}}D_{\theta}^{\beta'_{ik}}\oo{z_i}}{\alpha'_{ik}!\beta'_{ik}!}(x,x,\psi_x(x,z))\right).$$
	
	Denote $\Phi(x,z)=(\varphi_1,\varphi_2,\cdots,\varphi_{3n})=(x,x,\psi_x(x,z))$, and $w=(x,y,\theta)$. In general, for any smooth function $f(x,y,\theta)$ and any multi-indices $\mu,\nu\in (\mathbb{Z}^{\geq 0})^{2n}$,  we have
	\begin{align*}
	\frac{D_v^{\mu}D_{\oo{v}}^{\nu}}{\mu!\nu!}\left(f\left(\Phi(x,z)\right)\right)
	=\sum_{0\leq |\rho+\tau|\leq |\mu+\nu|}\frac{D_w^\rho D_{\oo{w}}^\tau f}{\rho!\tau!}\sum_{A_{\mu\nu\rho\tau}}
	\prod_{ij}\frac{D_v^{\mu_{ij}}D_{\oo{v}}^{\nu_{ij}}\varphi_i}{\mu_{ij}!\nu_{ij}!}
	\prod_{ik}\frac{D_v^{\mu'_{ik}}D_{\oo{v}}^{\nu'_{ik}}\oo{\varphi_i}}{\mu'_{ik}!\nu'_{ik}!},
	\end{align*}
	where the index set $A_{\mu\nu\rho\tau}$ is defined similar as in \eqref{Index Set} with a minor change that $1\leq i\leq 3n$.
	Applying this to our case,  we obtain
	\begin{align*}
	\frac{D_v^{\mu}D_{\oo{v}}^{\nu}}{\mu!\nu!}&\left(\frac{D_y^{\delta-\alpha}D_\theta^{\delta-\beta}\Delta_0}{(\delta-\alpha)!(\delta-\beta)!} (x,x,\psi_x(x,z))
	\sum_{A_{\alpha\beta\xi\eta}}
	\prod_{ij}\frac{D_y^{\alpha_{ij}}D_{\theta}^{\beta_{ij}}z_i}{\alpha_{ij}!\beta_{ij}!}
	\prod_{ik}\frac{D_y^{\alpha'_{ik}}D_{\theta}^{\beta'_{ik}}\oo{z_i}}{\alpha'_{ik}!\beta'_{ik}!}(x,x,\psi_x(x,z))\right)
	\\
	&
	=\sum_{0\leq |\rho+\tau|\leq |\mu+\nu|}\frac{D_w^\rho D_{\oo{w}}^\tau}{\rho!\tau!}
	\left(\frac{D_y^{\delta-\alpha}D_\theta^{\delta-\beta}\Delta_0}{(\delta-\alpha)!(\delta-\beta)!} 
	\sum_{A_{\alpha\beta\xi\eta}}
	\prod_{ij}\frac{D_y^{\alpha_{ij}}D_{\theta}^{\beta_{ij}}z_i}{\alpha_{ij}!\beta_{ij}!}
	\prod_{ik}\frac{D_y^{\alpha'_{ik}}D_{\theta}^{\beta'_{ik}}\oo{z_i}}{\alpha'_{ik}!\beta'_{ik}!}\right)
	\\
	&\qquad\cdot \sum_{A_{\mu\nu\rho\tau}}
	\prod_{ij}\frac{D_v^{\mu_{ij}}D_{\oo{v}}^{\nu_{ij}}\varphi_i}{\mu_{ij}!\nu_{ij}!}
	\prod_{ik}\frac{D_v^{\mu'_{ik}}D_{\oo{v}}^{\nu'_{ik}}\oo{\varphi_i}}{\mu'_{ik}!\nu'_{ik}!}
	\end{align*}
	
	We will use $C$ to denote a constant depending on constants $\eps,a,n$ and \k potential $\phi$ but independent with all the indices $m,\mu,\nu$, which may vary from line to line.
	Since $\Phi(x,z)\in \mathcal{A}^{a,\eps}_{diag}$ for each $1\leq i\leq n$, we have
	\begin{align}\label{diagonal 1.1}
	\left|\sum_{A_{\mu\nu\rho\tau}}
	\prod_{ij}\frac{D_v^{\mu_{ij}}D_{\oo{v}}^{\nu_{ij}}\varphi_i}{\mu_{ij}!\nu_{ij}!}
	\prod_{ik}\frac{D_v^{\mu'_{ik}}D_{\bar{v}}^{\nu'_{ik}}\oo{\varphi_i}}{\mu'_{ik}!\nu'_{ik}!}\right|
	\leq \sum_{A_{\mu\nu\rho\tau}} C^{|\mu+\nu|+|\rho+\tau|}\prod_{ij}\left(\mu_{ij}+\nu_{ij}\right)!^{a+\eps-1}
	\prod_{ik}\left(\mu'_{ik}+\nu'_{ik}\right)!^{a+\eps-1}.
	\end{align}
	Apply the combinatorial lemma \ref{Combinatoric 2} that we will prove later to the two products appearing above, we have
	\begin{align*}
	\prod_{ij}\left(\mu_{ij}+\nu_{ij}\right)!\leq \frac{(2n)^{|\rho|}}{\rho!}\left(\sum_{ij}(\mu_{ij}+\nu_{ij})\right)!,
	\quad
	\prod_{ik}\left(\mu'_{ik}+\nu'_{ik}\right)!\leq \frac{(2n)^{|\tau|}}{\tau!}\left(\sum_{ik}(\mu'_{ik}+\nu'_{ik})\right)!.
	\end{align*}
	Therefore,
	\begin{align*}
	\left|\sum_{A_{\mu\nu\rho\tau}}
	\prod_{ij}\frac{D_v^{\mu_{ij}}D_{\oo{v}}^{\nu_{ij}}\varphi_i}{\mu_{ij}!\nu_{ij}!}
	\prod_{ik}\frac{D_v^{\mu'_{ik}}D_{\oo{v}}^{\nu'_{ik}}\oo{\varphi_i}}{\mu'_{ik}!\nu'_{ik}!}\right|
	&\leq \sum_{A_{\mu\nu\rho\tau}} C^{|\mu+\nu|+|\rho+\tau|}\left(\frac{(\mu+\nu)!}{\rho!\tau!}\right)^{a+\eps-1}
	\\
	&\leq C^{|\mu+\nu|+|\rho+\tau|}\left(\frac{\mu!\nu!}{\rho!\tau!}\right)^{a+\eps-1}.
	\end{align*}
	The last inequality follows from
	\begin{align*}
	\#A_{\mu\nu\rho\tau}\leq \binom{\mu+|\rho+\tau|\mathbbm{1}}{|\rho+\tau|\mathbbm{1}} \binom{\nu+|\rho+\tau|\mathbbm{1}}{|\rho+\tau|\mathbbm{1}}
	\leq 2^{|\mu+\nu|+4n|\rho+\tau|}.
	\end{align*}

	As $z(x,y,\theta), \Delta_0(x,y,\theta)\in \mathcal{A}^{a,\eps}_{\theta}$ by Lemma \ref{AHE Composition} and Lemma \ref{AHE Inverse}, after a straightforward calculation, we have
	\begin{align}\label{diagonal 1.2}
	\begin{split}
	&\left|\frac{D_w^\rho D_{\oo{w}}^\tau}{\rho!\tau!}
	\left(\frac{D_y^{\delta-\alpha}D_\theta^{\delta-\beta}\Delta_0}{(\delta-\alpha)!(\delta-\beta)!} 
	\sum_{A_{\alpha\beta\xi\eta}}
	\prod_{ij}\frac{D_y^{\alpha_{ij}}D_{\theta}^{\beta_{ij}}z_i}{\alpha_{ij}!\beta_{ij}!}
	\prod_{ik}\frac{D_y^{\alpha'_{ik}}D_{\theta}^{\beta'_{ik}}\oo{z_i}}{\alpha'_{ik}!\beta'_{ik}!}\right)\right|
	\\
	&\qquad\leq C^{|\rho+\tau|+|\delta+\xi+\eta|+1}\left(\frac{\delta!^2\rho!\tau!}{\xi!\eta!}\right)^{a+\eps-1}
	\binom{\alpha+|\xi+\eta|\mathbbm{1}}{|\xi+\eta|\mathbbm{1}}
	\binom{\beta+|\xi+\eta|\mathbbm{1}}{|\xi+\eta|\mathbbm{1}}.
	\end{split}
	\end{align}
	Therefore, for any $\mu,\nu\geq 0$, we obtain
	\begin{align*}
	\begin{split}
	&\left|\frac{D_v^{\mu}D_{\oo{v}}^{\nu}}{\mu!\nu!}\left(\frac{D_y^{\delta-\alpha}D_\theta^{\delta-\beta}\Delta_0}{(\delta-\alpha)!(\delta-\beta)!} (x,x,\psi_x(x,z))
	\sum_{A_{\alpha\beta\xi\eta}}
	\prod_{ij}\frac{D_y^{\alpha_{ij}}D_{\theta}^{\beta_{ij}}z_i}{\alpha_{ij}!\beta_{ij}!}
	\prod_{ik}\frac{D_y^{\alpha'_{ik}}D_{\theta}^{\beta'_{ik}}\oo{z_i}}{\alpha'_{ik}!\beta'_{ik}!}(x,x,\psi_x(x,z))\right)\right|
	\\
	&
	\qquad\leq \sum_{0\leq |\rho+\tau|\leq |\mu+\nu|}C^{|\mu+\nu|+|\delta+\xi+\eta|+1}\left(\frac{\delta!^2\mu!\nu!}{\xi!\eta!}\right)^{a+\eps-1}
	\binom{\alpha+|\xi+\eta|\mathbbm{1}}{|\xi+\eta|\mathbbm{1}}
	\binom{\beta+|\xi+\eta|\mathbbm{1}}{|\xi+\eta|\mathbbm{1}}
	\\
	&
	\qquad
	\leq C^{|\mu+\nu|+|\delta+\xi+\eta|+1}\left(\frac{\delta!^2\mu!\nu!}{\xi!\eta!}\right)^{a+\eps-1}.
	\end{split}
	\end{align*}
	The last inequality follows from the fact that $\alpha,\beta\leq \delta$, 
	\begin{align*}
	\binom{\alpha+|\xi+\eta|\mathbbm{1}}{|\xi+\eta|\mathbbm{1}},\quad
	\binom{\beta+|\xi+\eta|\mathbbm{1}}{|\xi+\eta|\mathbbm{1}}
	\leq 2^{|\delta|+|\xi+\eta|n}.
	\end{align*} 
	Similarly, for any $\mu,\nu\geq 0$, we also have
	\begin{align*}
	\left| \frac{1}{l!}\frac{D_v^{\mu} D_{\oo{v}}^{\nu}}{\mu!\nu!}\left((D_y\cdot D_\theta)^l\Delta_0(x,x,\psi_x(x,z))\right)\right|
	\leq \sum_{|\delta|=l}\delta! C^{|\mu+\nu|+|\delta|+1}\left(\delta!^2\mu!\nu!\right)^{a+\eps-1}
	\end{align*}
	Then \eqref{bmrecursive} implies the following inequality
	\begin{align}\label{diagonal 2}
	\begin{split}
	\left|D_v^\mu D_{\oo{v}}^\nu b_m(x,z)\right|
	&\leq\sum_{l=1}^m
	\sum_{|\delta|=l}\delta!^{2a+2\eps-1}
	\sum_{\alpha,\beta\leq \delta}
	\sum_{|\xi+\eta|\leq |\alpha+\beta|}
	\sum_{\mu_0\leq \mu}\sum_{\nu_0\leq \nu} \frac{\left|D_v^{\widetilde{\xi}+\mu_0} D_{\oo{v}}^{\widetilde{\eta}+\nu_0} b_{m-l}(x,z)\right|}{(\xi!\eta!)^{a+\eps}}
	\\&
	\quad\quad \cdot
	C^{|\mu-\mu_0+\nu-\nu_0|+|\delta+\xi+\eta|}\binom{\mu}{\mu_0}\binom{\nu}{\nu_0}
	(\mu-\mu_0)!^{a+\eps}(\nu-\nu_0)!^{a+\eps}.
	\end{split}
	\end{align}
	
	Now we will change all the derivatives $\left|D_v^\mu D_{\oo{v}}^\nu b_m(x,z)\right|$ to the notation $b_{m,\mu\bar{\nu}}$ in the above inequality.
	Note that on the right hand side of \eqref{bmrecursive}, when $\nu\neq 0$, the anti-holomorphic derivative will hit on either $b_{m-l}(x,z)$ or at least one of the these functions $z_i(x,y,\theta), \Delta_0(x,y,\theta) \in \mathcal{A}^{a,\eps}_{\theta}$ and $\varphi_i(x,z)\in \mathcal{A}^{a,\eps}_{diag}$ for $1\leq i\leq n$. We will consider each case in the following. If the anti-holomorphic derivative hits on $z_i(x,y,\theta)$ or $\Delta_0(x,y,\theta)$, since the derivatives are evaluated at $(x,x,\psi_x(x,z))$, we will have the extra factor $\exp(-b|x-\bar{z}|^{-\frac{1}{a-1}})$ on the right hand side. If the anti-holomorphic derivative hits on $\varphi_i(x,z)$, we also have the extra factor $\exp(-b|x-\bar{z}|^{-\frac{1}{a-1}})$ since $\varphi_i(x,z)\in \mathcal{A}^{a,\eps}_{diag}$. The last case is that the anti-holomorphic derivative hits on $b_{m-l}(x,z)$, which means $|\eta+\nu_0|\neq 0$. We again have the extra factor $\exp(-b|x-\bar{z}|^{-\frac{1}{a-1}})$ when we change $D_z^{\xi+\mu_0}D_{\bar{z}}^{\eta+\nu_0}b_{m-l}(x,z)$ into $b_{m-l,\widetilde{\xi}+\mu_0\oo{\widetilde{\eta}+\nu_0}}$. So no matter in which case, at least one $\exp(-b|x-\bar{z}|^{-\frac{1}{a-1}})$ will appear on the right side when $\nu\neq 0$. And thus the desired result follows. 
\end{proof}

Next we use this lemma to prove Theorem \ref{MainLemma}.

\subsection{Proof of Theorem \ref{MainLemma}}

	For convenience, we define
	\begin{equation}
	a_{m,\mu\bar{\nu}}=\frac{b_{m,\mu\bar{\nu}}}{(2m+1)!^{a+\eps}\mu!^{a+\eps}\nu!^{a+\eps}}.
	\end{equation}
	Then by Lemma \ref{bmLemmaGevrey}
	\begin{align}\label{amrecursiveGevrey}
	\begin{split}
	&a_{m,\mu\bar{\nu}}
	\\
	&\leq\sum_{l=1}^m
	\sum_{|\delta|=l}
	\sum_{\alpha,\beta\leq \delta}
	\sum_{|\xi+\eta|\leq |\alpha+\beta|}
	\sum_{\mu_0\leq \mu}\sum_{\nu_0\leq \nu} \frac{ a_{m-l,\widetilde{\xi}+\mu_0\oo{\widetilde{\eta}+\nu_0}}}{\binom{2m+1}{2l}^{a+\eps}}
	\binom{\widetilde{\xi}+\mu_0}{\widetilde{\xi}}^{a+\eps}
	\binom{\widetilde{\eta}+\nu_0}{\widetilde{\eta}}^{a+\eps}
	C^{|\mu-\mu_0+\nu-\nu_0+|\delta+\xi+\eta|},
	\end{split}
	\end{align}
	where $\xi,\eta\in (\mathbb{Z}^{\geq 0})^n$ and $\widetilde{\xi}=(0,\cdots 0,\xi), \widetilde{\eta}=(0, \cdots, 0,\eta)\in (\mathbb{Z}^{\geq 0})^{2n}$.
	Since $b_0(x,z)=1$, we have
	\begin{equation}\label{a0Gevrey}
	a_{0,\mu\bar{\nu}}=
	\begin{cases}
	1 & \mu=\nu=(0,0,\cdots,0),\\
	0 & \mbox{otherwise}.
	\end{cases}
	\end{equation}
	We will argue by induction on $m$ to prove that for any integer $m\geq 0$ and multi-index $\mu,\nu\geq 0$,
	\begin{equation}\label{amupperboundGevrey}
	a_{m,\mu\bar{\nu}}\leq\binom{2m+|\mu+\nu|}{|\mu+\nu|}^{a+\eps} A^{m}(2C)^{|\mu+\nu|},
	\end{equation}
	where $C$ is the same constant which appears on the right hand side of \eqref{amrecursiveGevrey} and $A$ is a bigger constant to be selected later. Without losing of generality, we assume $C\geq 1$. Obviously \eqref{a0Gevrey} implies that \eqref{amupperboundGevrey} holds for $m=0$ and any $\mu,\nu\geq 0$. Assume that \eqref{amupperboundGevrey} holds up to $m-1$ and we proceed to $m$. By \eqref{amrecursiveGevrey}, we have
	\begin{align*}
	a_{m,\mu\bar{\nu}}
	&\leq\sum_{l=1}^m
	\sum_{|\delta|=l}
	\sum_{\alpha,\beta\leq \delta}
	\sum_{|\xi+\eta|\leq |\alpha+\beta|}
	\sum_{\mu_0\leq \mu}\sum_{\nu_0\leq \nu} 
	\frac{A^{m-l}}{\binom{2m+1}{2l}^{a+\eps}}
	\binom{\widetilde{\xi}+\mu_0}{\widetilde{\xi}}^{a+\eps}
	\binom{\widetilde{\eta}+\nu_0}{\widetilde{\eta}}^{a+\eps}
	\\&
	\qquad \cdot
	2^{|\xi+\eta|+|\mu_0+\nu_0|}C^{|\mu+\nu|+|\delta+2\xi+2\eta|}	\binom{2(m-l)+|\widetilde{\xi}+\mu_0+\widetilde{\eta}+\nu_0|}{|\widetilde{\xi}+\mu_0+\widetilde{\eta}+\nu_0|}^{a+\eps} 
	\\
	& \leq\sum_{l=1}^m
	\sum_{|\delta|=l}
	\sum_{|\xi+\eta|\leq 2l}
	\sum_{\mu_0\leq \mu}\sum_{\nu_0\leq \nu} 
	\frac{A^{m-l}}{\binom{2m+1}{2l}^{a+\eps}}
	\binom{\widetilde{\xi}+\mu_0}{\widetilde{\xi}}^{a+\eps}
	\binom{\widetilde{\eta}+\nu_0}{\widetilde{\eta}}^{a+\eps}
	\\&
	\qquad\cdot
	2^{|\xi+\eta|+|\mu_0+\nu_0|}C^{|\mu+\nu|+|\delta+2\xi+2\eta|}
	\binom{2(m-l)+|\widetilde{\xi}+\mu_0+\widetilde{\eta}+\nu_0|}{|\widetilde{\xi}+\mu_0+\widetilde{\eta}+\nu_0|}^{a+\eps} 
	\cdot\#\{\alpha\leq \delta\} \cdot	\#\{\beta\leq \delta\}.
	\end{align*}
	
	Due to the fact
	\begin{equation*}
	\#\{\alpha\leq \delta\}
	=\#\{\beta\leq \delta\}
	\leq 2^{|\delta|},
	\end{equation*}
	it follows that
	\begin{align*}
	a_{m,\mu\bar{\nu}}
	& \leq\sum_{l=1}^m
	\sum_{|\delta|=l}
	\sum_{|\xi+\eta|\leq 2l}
	\sum_{\mu_0\leq \mu}\sum_{\nu_0\leq \nu} 
	\frac{A^{m-l}}{\binom{2m+1}{2l}^{a+\eps}}
	\binom{\widetilde{\xi}+\mu_0}{\widetilde{\xi}}^{a+\eps}
	\binom{\widetilde{\eta}+\nu_0}{\widetilde{\eta}}^{a+\eps}
	\\&
	\qquad\cdot
	2^{|\xi+\eta|+|\mu_0+\nu_0|+2|\delta|}C^{|\mu+\nu|+|\delta+2\xi+2\eta|}
	\binom{2(m-l)+|\widetilde{\xi}+\mu_0+\widetilde{\eta}+\nu_0|}{|\widetilde{\xi}+\mu_0+\widetilde{\eta}+\nu_0|}^{a+\eps} 
	\\
	&\leq 
	A^m(2C)^{|\mu+\nu|}\sum_{l=1}^m
	\sum_{|\delta|=l}
	\sum_{|\xi+\eta|\leq 2l}
	\sum_{\mu_0\leq \mu}\sum_{\nu_0\leq \nu} 
	\frac{1}{\binom{2m+1}{2l}^{a+\eps}}
	\binom{\widetilde{\xi}+\mu_0}{\widetilde{\xi}}^{a+\eps}
	\binom{\widetilde{\eta}+\nu_0}{\widetilde{\eta}}^{a+\eps}A^{-l}
	\\&
	\qquad\cdot
	2^{4l+|\mu_0+\nu_0|-|\mu+\nu|}C^{5l}
	\binom{2(m-l)+|\widetilde{\xi}+\mu_0+\widetilde{\eta}+\nu_0|}{|\widetilde{\xi}+\mu_0+\widetilde{\eta}+\nu_0|}^{a+\eps} 
	\end{align*}
	
	Moreover, since
	\begin{align*}
	\#\{|\delta|=l\}=\binom{l+n-1}{n-1}\leq 2^{l+n-1}\leq 2^{nl},
	\end{align*}
	we have
	\begin{align}\label{am2Gevrey}
	\begin{split}
	a_{m,\mu\bar{\nu}}
	&\leq 
	A^m(2C)^{|\mu+\nu|}\sum_{l=1}^m
	\sum_{|\xi+\eta|\leq 2l}
	\sum_{\mu_0\leq \mu}\sum_{\nu_0\leq \nu} 
	\frac{1}{\binom{2m+1}{2l}^{a+\eps}}
	\binom{\widetilde{\xi}+\mu_0}{\widetilde{\xi}}^{a+\eps}
	\binom{\widetilde{\eta}+\nu_0}{\widetilde{\eta}}^{a+\eps}
	\\&
	\qquad\cdot
	\binom{2(m-l)+|\widetilde{\xi}+\mu_0+\widetilde{\eta}+\nu_0|}{|\widetilde{\xi}+\mu_0+\widetilde{\eta}+\nu_0|}^{a+\eps} 		2^{|\mu_0+\nu_0|-|\mu+\nu|}\left(\frac{2^{n+4}C^5}{A}\right)^l
	\end{split}
	\end{align}
	
	In the next step we apply the combinatorial inequality
	\begin{equation*}
	\binom{\widetilde{\xi}+\mu_0}{\mu_0}\binom{\widetilde{\eta}+\nu_0}{\nu_0}
	\leq \binom{\widetilde{\xi}+\widetilde{\eta}+\mu_0+\nu_0}{\mu_0+\nu_0}
	\leq \binom{|\xi+\eta|+|\mu_0+\nu_0|}{|\mu_0+\nu_0|}
	\end{equation*}
	and
	\begin{align*}
	\binom{2(m-l)+|\xi+\eta|+|\mu_0+\nu_0|}{|\xi+\eta|+|\mu_0+\nu_0|}\binom{|\xi+\eta|+|\mu_0+\nu_0|}{|\mu_0+\nu_0|}
	=\binom{2(m-l)+|\xi+\eta|+|\mu_0+\nu_0|}{|\mu_0+\nu_0|}\binom{2m-2l+|\xi+\eta|}{2m-2l}.
	\end{align*}
	Observe that, since $|\xi+\eta|\leq 2l, \mu_0\leq \mu,\nu_0\leq\nu$, we have
	\begin{align*}
	\binom{2(m-l)+|\xi+\eta|+|\mu_0+\nu_0|}{|\mu_0+\nu_0|}
	\leq \binom{2m+|\mu+\nu|}{|\mu+\nu|}.
	\end{align*}
	Plugging these into \eqref{am2Gevrey}, we obtain
	\begin{align*}
	\begin{split}
	a_{m,\mu\bar{\nu}}
	&\leq 
	A^m(2C)^{|\mu+\nu|}\sum_{l=1}^m
	\frac{1}{\binom{2m+1}{2l}^{a+\eps}}
	\binom{2m+|\mu+\nu|}{|\mu+\nu|}^{a+\eps}
	\sum_{|\xi+\eta|\leq 2l}
	\binom{2m-2l+|\xi+\eta|}{2m-2l}^{a+\eps}
	\\&
	\qquad\qquad\qquad\quad\cdot 		
	\sum_{\mu_0\leq \mu}\sum_{\nu_0\leq \nu} 
	2^{|\mu_0+\nu_0|-|\mu+\nu|}\left(\frac{2^{n+4}C^5}{A}\right)^l
	\end{split}
	\end{align*}
	
	Again since
	\begin{equation}
	\#\{|\xi+\eta|=k\}=\binom{k+2n-1}{2n-1}\leq 2^{k+2n-1},
	\end{equation}
	the sum over index $\xi,\eta$ on the right hand side can be estimated as
	\begin{align*}
	\sum_{|\xi+\eta|\leq 2l}     \binom{2m-2l+|\xi+\eta|}{2m-2l}^{a+\eps}
	=\sum_{k=0}^{2l}\sum_{|\xi+\eta|=k}     \binom{2m-2l+k}{2m-2l}^{a+\eps}
	\leq 2^{2l+2n-1}\binom{2m+1}{2m-2l+1}^{a+\eps}.
	\end{align*}
	Therefore,
	\begin{align*}
	a_{m,\mu\bar{\nu}}
	\leq& 
	A^m(2C)^{|\mu+\nu|} \binom{2m+|\mu+\nu|}{|\mu+\nu|}^{a+\eps}
	\sum_{l=1}^m
	\left(\frac{2^{3n+6}C^5}{A}\right)^l	 		
	\sum_{\mu_0\leq \mu}2^{|\mu_0|-|\mu|}\sum_{\nu_0\leq \nu} 
	2^{|\nu_0|-|\nu|}
	\\
	\leq& 
	A^m(2C)^{|\mu+\nu|} \binom{2m+|\mu+\nu|}{|\mu+\nu|}^{a+\eps}
	\sum_{l=1}^m
	\left(\frac{2^{3n+6}C^5}{A}\right)^l2^{4n}	 		
	\end{align*}	
	By taking $A= 2^{7n+7}C^5$, we surely have $\sum_{l=1}^m\left(\frac{2^{3n+6}C^5}{A}\right)^l2^{4n}<1$, which implies that $$a_{m,\mu\bar{\nu}}\leq A^m(2C)^{|\mu+\nu|}\binom{2m+|\mu+\nu|}{|\mu+\nu|}^{a+\eps}.$$ 
	So if we write $a_{m,\mu\bar{\nu}}$ in terms of $D_v^\mu D_{\oo{v}}^{\nu}b_m(x,z)$, then by the continuity of each $D_v^\mu D_{\oo{v}}^{\nu}b_m(x,z)$ for all $x,z\in U$, we have
	\begin{align*}
	\begin{split}
	|D_v^\mu D_{\oo{v}}^{\nu}b_m(x,z)|&\leq  (2m+1)!^{a+\eps}\mu!^{a+\eps}\nu!^{a+\eps} a_{m,\mu\bar{\nu}}
	\lambda_{b,|\nu|}(x,x,z)
	\\
	&\leq (64^{a+\eps}A)^{m+|\mu+\nu|}m!^{2a+2\eps}\mu!^{a+\eps}\nu!^{a+\eps}
	\lambda_{b,|\nu|}(x,x,z).
	\end{split}
	\end{align*}
	Thus \eqref{bmupperbound} follows by renaming $64^{a+\eps}A$ to $C$.
	
	In particular, when we are restricted to diagonal $z=\bar{x}$, 
	\begin{equation*}
	D_v^\mu D_{\oo{v}}^{\nu}b_m(x,\bar{x})=0, \qquad \mbox{ for any multi-indices $\mu\geq 0$ and $|\nu|\neq 0$}.
	\end{equation*}
	And thus when $z=\bar{x}$, the recursive inequality \eqref{diagonal 2} reduces to
	\begin{align*}
		\left|D_v^\mu b_m(x,\bar{x})\right|
		\leq\sum_{l=1}^m
		\sum_{|\delta|=l}\delta!^{2a+2\eps-1}
		\sum_{\alpha,\beta\leq \delta}
		\sum_{|\xi|\leq |\alpha+\beta|}
		\frac{\left|D_v^{\widetilde{\xi}+\mu_0} b_{m-l}(x,\bar{x})\right|}{(\xi!\eta!)^{a+\eps}}
		C^{|\mu-\mu_0|+|\delta+\xi|}\binom{\mu}{\mu_0}
		(\mu-\mu_0)!^{a+\eps}.
	\end{align*}
	Note that the constant $\eps$ only comes from \eqref{diagonal 1.1} and \eqref{diagonal 1.2} because of the derivatives of $\varphi_i(x,z)\in \mathcal{A}^{a,\eps}_{diag}$ for $1\leq i\leq n$ and $z(x,y,\theta), \Delta_0(x,y,\theta)\in \mathcal{A}^{a,\eps}_\theta$. By the definitions of $\mathcal{A}^{a,\eps}_{diag}$ and $\mathcal{A}^{a,\eps}_\theta$, $\eps$ can be replaced by $0$ when we are restricted to $x=y=\bar{z}$. Therefore, \eqref{diagonal 2} further reduces to
	\begin{align*}
	\left|D_v^\mu b_m(x,\bar{x})\right|
	\leq\sum_{l=1}^m
	\sum_{|\delta|=l}\delta!^{2a-1}
	\sum_{\alpha,\beta\leq \delta}
	\sum_{|\xi|\leq |\alpha+\beta|}
	\frac{\left|D_v^{\widetilde{\xi}+\mu_0} b_{m-l}(x,\bar{x})\right|}{(\xi!\eta!)^{a+\eps}}
	C^{|\mu-\mu_0|+|\delta+\xi|}\binom{\mu}{\mu_0}
	(\mu-\mu_0)!^{a+\eps}.
	\end{align*}
	By using a similar inductive argument as that of estimating $\left|D_v^\mu D_{\oo{v}}^\nu b_m(x,z)\right|$, we obtain for any $x\in U$,
	\begin{align*}
	\left|D_v^\mu b_m(x,\bar{x})\right|\leq C^{m+|\mu|}m!^{2a}\mu!^{a}.
	\end{align*}

\section{Optimality of the upper bounds on Bergman coefficients $b_m$} \label{optimal}
In this section, we will show that although it would be desirable to improve the estimate \eqref{bmupperbound} to
\begin{equation}\label{bmconjecture}
\left | D_v^\alpha D_{\oo{v}} ^\beta b_m(x,z) \right |  \leq C^{m+|\alpha|+|\beta|}m!^{2a+2\eps-1}\alpha!^{a+\eps} \beta!^{a+\eps} \exp{\left(-b (1- \delta_0(|\beta|)){|x-\bar{z}|^{-\frac{1}{a-1}}}\right)}\lambda_{b,|\eta|}(x,x,z),
\end{equation}
it is not possible to prove it by the recursive inequality \eqref{bmGevrey}. Here we provide an example which satisfies \eqref{bmGevrey} while fails \eqref{bmconjecture}. For simplicity, we assume $C=1$ in \eqref{bmGevrey}. Let's consider the worse case when the equality holds, i.e.
\begin{align}\label{bmoptimalGevrey}
\begin{split}
&b_{m,\mu\bar{\nu}}
\\
&=\sum_{l=1}^m
\sum_{|\delta|=l}\delta!^{2a+2\eps-1}
\sum_{\alpha,\beta\leq \delta}
\sum_{|\xi+\eta|\leq |\alpha+\beta|}
\sum_{\mu_0\leq \mu}\sum_{\nu_0\leq \nu} \frac{ b_{m-l,\widetilde{\xi}+\mu_0\oo{\widetilde{\eta}+\nu_0}}}{(\xi!\eta!)^{a+\eps}}
\binom{\mu}{\mu_0}\binom{\nu}{\nu_0}
(\mu-\mu_0)!^{a+\eps}(\nu-\nu_0)!^{a+\eps},
\end{split}
\end{align}
One can easily check that this recursive equation uniquely defines $\{b_{m,\mu\bar{\nu}} \}$ given an initial data $\{ b_{0, \mu\bar{\nu}} \}$.
We shall only focus on the terms $b_{m,k\widetilde{e_1}}$ where $e_1=(1,0,\cdots,0)\in \mathbb{R}^n$, $\widetilde{e}_1=(0,\cdots, 0, e_1)\in \mathbb{R}^{2n}$ and show by induction that
\begin{equation}\label{bmLB}
b_{m,k\widetilde{e_1}}\geq 2^{-(a+\eps)m}(2m-2+k)!^{a+\eps} \hspace{12pt} \mbox{ for any }m\geq 1,k\geq 0.
\end{equation}
First let's check $b_{1,k\widetilde{e_1}}$. Since we know
\begin{equation*}
b_{0,\mu\bar{\nu}}=
\begin{cases}
1 & \mu=\nu=(0,0,\cdots,0),\\
0 & \mbox{otherwise},
\end{cases}
\end{equation*}
by \eqref{bmoptimalGevrey} we have
\begin{align}\label{eq 3.15}
b_{1,\mu\bar{\nu}}
=
\sum_{|\delta|=1}
\sum_{\alpha,\beta\leq \delta}
\mu!^{a+\eps}\nu!^{a+\eps}
\geq \mu!^{a+\eps}\nu!^{a+\eps}.
\end{align}
Therefore \eqref{bmLB} holds for $b_{1,k\widetilde{e_1}}$.
Assume that \eqref{bmLB} holds for $b_{1,k\widetilde{e_1}},b_{2,k\widetilde{e_1}},\cdots b_{m-1,k\widetilde{e_1}}$. Then by only considering the terms with index $l=|\alpha|=|\beta|=1$, $\mu_0=\mu$ and $\xi=2e_1$, $\eta=0$ in \eqref{eq 3.15}, we obtain for $m\geq 2$
\begin{align*}
b_{m,k\widetilde{e}_1}
\geq 
\sum_{|\delta|=1}
\sum_{|\alpha|=|\beta|=1}
\frac{ b_{m-1,(k+2)\widetilde{e_1}}}{2^{a+\eps}}
\geq 
\frac{ b_{m-1,(k+2)\widetilde{e_1}}}{2^{a+\eps}}
\geq 
2^{-(a+\eps)m}(2m-2+k)!^{a+\eps}.	
\end{align*}
Note that if in particular we put $k=0$ into \eqref{bmLB}, then we get
\begin{align*}
b_{m,0}\geq \left(\frac{1}{8}\right)^{(a+\eps)m}m!^{2a+2\eps},
\end{align*} 
which show that up to an exponential factor $C^m$, $m!^{2a+2\eps}$ is the best upper bound one can hope from the recursive inequality \eqref{bmGevrey}.

\section{Proofs of main lemmas on almost holomorphic extensions of Gevrey functions}\label{ProofOfGevreyStuff}
In this section, we will complete all the proofs skipped in Section \ref{AHEoGF}. 

\begin{proof}[Proof of Lemma \ref{lem almostholomorphic}]
	We will prove the estimate on $ D_{\bar{z}}F(f)$. The other one follows in the same way. For simplicity, we denote 
	$$\widetilde{\chi}(|\alpha+\beta|)=\chi\big(|\alpha+\beta|^{2(a-1)}4^{a-1}C_1^2\left|y-\bar{z}\right|^2\big).$$
	For any $1\leq i\leq n$, we have
	\begin{align*}
	D_{\bar{z}_i}&F(f)\\
	&=\frac{1}{2}\sum_{\alpha,\beta\geq 0} \frac{ D_x^{\alpha+e_i} D_{\bar{x}}^{\beta} f}{\alpha!\beta!}\left(\frac{y-\bar{z}}{2}\right)^\alpha
	\left(\frac{z-\bar{y}}{2}\right)^\beta \Big(\widetilde{\chi}(|\alpha+\beta|)-\widetilde{\chi}\left(|\alpha+\beta+e_i|\right)\Big)
	\\&
	+2\sum_{\alpha,\beta\geq 0} \frac{ D_x^\alpha D_{\bar{x}}^\beta f}{\alpha!\beta!}\left(\frac{y-\bar{z}}{2}\right)^\alpha \left(\frac{z-\bar{y}}{2}\right)^{\beta+e_i}|\alpha+\beta|^{2(a-1)}4^{a-1}C_1^2\chi'\left(|\alpha+\beta|^{2(a-1)}4^{a-1}C_1^2\left|y-\bar{z}\right|^2\right).
	\\
	&:=I+II.
	\end{align*}
	We will use $C$ to denote a constant depend on $a, \eps, C_1(f)$ and the cut-off function $\chi$, which may change from line to line. Since $f\in G^a(U)$ and by Stirling's formula, each term in $I$ is bounded by
	\begin{align*}
	C_0C_1^{|\alpha+\beta|+1}(\alpha_i+1)^a&(\alpha!\beta!)^{a-1}\left|\frac{y-\bar{z}}{2}\right|^{|\alpha+\beta|}
	\\
	&\leq
	CC_0C_1^{|\alpha+\beta|}\left(\left|\alpha+\beta\right|+1\right)^{a+\frac{a-1}{2}}\left(\frac{|\alpha+\beta|}{e}\right)^{(a-1)|\alpha+\beta|}\left|\frac{y-\bar{z}}{2}\right|^{|\alpha+\beta|}.
	\end{align*}
	
	Note that the difference of cut-off functions in $I$ is zero unless
	\begin{equation}\label{IndexRange}
	\left|\alpha+\beta\right|\in \left[\frac{1}{2}\left(\sqrt{2}C_1|y-\bar{z}|\right)^{-\frac{1}{a-1}}-1,\frac{1}{2}\left(C_1|y-\bar{z}|\right)^{-\frac{1}{a-1}}\right].
	\end{equation}
	It implies that each term in $I$ is bounded by
	\begin{equation*}
	CC_0(|\alpha+\beta|+1)^{a+\frac{a-1}{2}}2^{-|\alpha+\beta|}e^{-(a-1)|\alpha+\beta|}
	\leq
	CC_0e^{-\frac{a-1}{2}(\sqrt{2}C_1|y-\bar{z}|)^{-\frac{1}{a-1}}}.
	\end{equation*}
	Since there are less than $\left(\frac{1}{2}(C_1|y-\bar{z}|)^{-\frac{1}{a-1}}+1\right)^{2n}$ many terms in $I$, we have
	\begin{equation}
	|I|\leq CC_0 e^{-b|y-\bar{z}|^{-\frac{1}{a-1}}},
	\end{equation}
	where $b$ is a positive constant depending on $a,C_1=C_1(f)$.
	
	For the second term $II$, similarly we have \eqref{IndexRange} or $\chi'$ vanishes otherwise. And thus each term is bounded by
	\begin{align*}
	CC_0 C_1&^{|\alpha+\beta|+2}\left(\frac{|\alpha+\beta|}{e}\right)^{(a-1)|\alpha+\beta|}\left|\frac{y-\bar{z}}{2}\right|^{|\alpha+\beta|+1}|\alpha+\beta|^{\frac{5}{2}(a-1)}
	\\
	\leq &
	CC_0 e^{-(a-1)|\alpha+\beta|}|\alpha+\beta|^{\frac{3}{2}(a-1)}
	\\
	\leq &
	CC_0 e^{-\frac{a-1}{4}(\sqrt{2}C_1|y-\bar{z}|)^{-\frac{1}{a-1}}}
	\\
	\leq &
	CC_0 e^{-b|y-\bar{z}|^{-\frac{1}{a-1}}}.
	\end{align*}
	So we have
	\begin{equation*}
	| D_{\bar{z}_i}F(f)(y,z)|\leq CC_0 \exp\left({-b|y-\bar{z}|^{-\frac{1}{a-1}}}\right) \qquad \mbox{for } 1\leq i\leq n. 
	\end{equation*}
	Thus the result follows.
\end{proof}

\begin{proof}[Proof of Lemma \ref{AHE Uniqueness}]
	Since $F$ is an almost holomorphic extension of $f$, by taking the Taylor expansion, for any $N\in \mathbb{N}$ we have
	\begin{align*}
	F(x+y,\bar{x}+z)&=\sum_{|\alpha+\beta|\leq N-1}\frac{D_y^{\alpha} D_z^{\beta}F}{\alpha!\beta!}\left(x,\bar{x}\right)y^\alpha z^\beta+O\left(|(y,z)|^{N}\right)\\
	&=\sum_{|\alpha+\beta|\leq N-1}\frac{D_x^\alpha D_{\bar{x}}^{\beta}f}{\alpha!\beta!}\left(x\right)y^\alpha z^\beta+O\left(|(y,z)|^{N}\right).
	\end{align*}
	If we take $x=\frac{y+\bar{z}}{2}$ and replace $(y,z)$ by $(\frac{y-\bar{z}}{2},\frac{z-\bar{y}}{2})$, then
	\begin{equation*}
	F(y,z)=\sum_{|\alpha+\beta|\leq N-1}\frac{D_x^\alpha D_{\bar{x}}^{\beta}f}{\alpha!\beta!}\left(\frac{y+\bar{z}}{2}\right)\left(\frac{y-\bar{z}}{2}\right)^\alpha \left(\frac{z-\bar{y}}{2}\right)^\beta+O\left(\left|y-\bar{z}\right|^{N}\right).
	\end{equation*} 
	Similarly, the same identity holds for $\widetilde{F}(y,z)$. Therefore, for any $N\in \mathbb{N}$, we have
	\begin{equation*}
	F(y,z)-\widetilde{F}(y,z)=O\left(|y-\bar{z}|^N\right).
	\end{equation*}
\end{proof}

\begin{proof}[Proof of Lemma \ref{AHE2}]
	To prove this lemma we first need to obtain some estimates on the derivatives of our cut-off function $\chi$.
	\begin{lemma}
		Let $\eps>0$ be a constant and $\chi\in G^{1+\eps}(\mathbb{R})$ be the cut-off function constructed in \eqref{Gevreycutoff}. Then there exists some positive constant $C=C(\chi)$ such that for any multi-indices $\gamma,\delta,\xi,\eta\geq 0$, we have
		\begin{align*}
		\| D_y^\gamma D_z^\delta D_{\bar{y}}^\xi D_{\bar{z}}^\eta &\left(\chi\left(|\alpha+\beta|^{2(a-1)}4^{a-1}C_1^2\left|y-\bar{z}\right|^2\right)\right)\|_{L^\infty(\mathbb{C}^{2n})}
		\\   
		& \leq \left(2^{a-1}CC_1|\alpha+\beta|^{a-1}\right)^{|\gamma+\eta+\xi+\delta|}
		(\gamma+\eta+\xi+\delta)!^{1+\eps}.
		\end{align*}
	\end{lemma}
	
	\begin{proof}
		By a straightforward calculation, we have
		\begin{align*}
		D_y^\gamma \left(\chi\left(|\alpha+\beta|^{2(a-1)}4^{a-1}C_1^2\left|y-\bar{z}\right|^2\right)\right)
		=\left(|\alpha+\beta|^{2(a-1)}4^{a-1}C_1^2(\bar{y}-z)\right)^\gamma\chi^{(|\gamma|)},
		\end{align*}
		\begin{align*}
		D_y^\gamma  D_{\bar{z}}^\eta \left(\chi\left(|\alpha+\beta|^{2(a-1)}4^{a-1}C_1^2\left|y-\bar{z}\right|^2\right)\right)
		=(-1)^{|\eta|}\left(|\alpha+\beta|^{2(a-1)}4^{a-1}C_1^2(\bar{y}-z)\right)^{\gamma+\eta}
		\chi^{(|\gamma+\eta|)}.
		\end{align*}
		Therefore,
		\begin{align*}
		& D_y^\gamma D_z^\delta D_{\bar{y}}^\xi D_{\bar{z}}^\eta \left(\chi((\alpha+\beta)^{2(a-1)}4^{a-1}C_1^2\left|y-\bar{z}\right|^2)\right)
		\\
		=&\sum_{\xi_0\leq \xi,\delta_0\leq \delta}\binom{\xi}{\xi_0}\binom{\delta}{\delta_0}(-1)^{|\eta+\xi_0+\delta-\delta_0|}\left(|\alpha+\beta|^{2(a-1)}4^{a-1}C_1^2\right)^{|\gamma+\eta+\xi_0+\delta_0|}
		\\&
		\frac{(\gamma+\eta)!}{(\gamma+\eta+\xi_0+\delta_0-\xi-\delta)!}
		(\bar{y}-z)^{\gamma+\eta-(\xi-\xi_0)-(\delta-\delta_0)}
		(\bar{z}-y)^{\xi_0+\delta_0}
		\chi^{(|\gamma+\eta+\xi_0+\delta_0|)}.
		\end{align*}
		Since the cut-off function $\chi$ is in $G^{1+\eps}(R)$, there exists a positive constant $C=C(\chi)$, which may vary from line to line, such that
		\begin{align*}
		&\left| D_y^\gamma D_z^\delta D_{\bar{y}}^\xi D_{\bar{z}}^\eta \left(\chi(|\alpha+\beta|^{2(a-1)}4^{a-1}C_1^2\left|y-\bar{z}\right|^2)\right)\right|
		\\
		\leq&\sum_{\xi_0\leq \xi,\delta_0\leq \delta}\binom{\xi}{\xi_0}\binom{\delta}{\delta_0}\left(|\alpha+\beta|^{2(a-1)}4^{a-1}C_1^2\right)^{|\gamma+\eta+\xi_0+\delta_0|}C^{|\gamma+\eta+\xi_0+\delta_0|}
		\\&
		\frac{(\gamma+\eta)!}{(\gamma+\eta+\xi_0+\delta_0-\xi-\delta)!}
		\left|\bar{y}-z\right|^{|\gamma+\eta-\xi+2\xi_0-\delta+2\delta_0|}
		|\gamma+\eta+\xi_0+\delta_0|!^{1+\eps}
		\\
		\leq&C^{|\gamma+\delta+\xi+\eta|}(\gamma+\eta+\xi+\delta)!^{1+\eps}\\
		&\sum_{\xi_0\leq \xi,\delta_0\leq \delta}\binom{\xi}{\xi_0}\binom{\delta}{\delta_0}\left(|\alpha+\beta|^{2(a-1)}4^{a-1}C_1^2\right)^{|\gamma+\eta+\xi_0+\delta_0|}
		\left|\bar{y}-z\right|^{|\gamma+\eta-\xi+2\xi_0-\delta+2\delta_0|}.
		\end{align*}
		Our result follows by using that for any $y,z\in\mathbb{C}^n$, $$|\alpha+\beta|^{2(a-1)}4^{a-1}C_1^2\left|y-\bar{z}\right|^2\leq 1.$$
	\end{proof}
	
	We now estimate the derivatives on $F(f)$. By a straightforward calculation, we have
	\begin{align*}
	& D_y^\gamma D_z^\delta  D_{\bar{y}}^{\xi}  D_{\bar{z}}^{\eta} F(f)(y,z)
	\\
	=&\sum_{\substack{\gamma_0+\gamma_1+\gamma_2=\gamma\\\delta_0+\delta_1+\delta_2=\delta}}
	\sum_{\substack{\xi_0+\xi_1+\xi_2=\xi\\\eta_0+\eta_1+\eta_2=\eta}}
	\sum_{\substack{\alpha\geq \gamma_1+\eta_1,\\ \beta\geq \xi_1+\delta_1}}
	\binom{\gamma}{\gamma_0,\gamma_1,\gamma_2}
	\binom{\delta}{\delta_0,\delta_1,\delta_2}
	\binom{\xi}{\xi_0,\xi_1,\xi_2}
	\binom{\eta}{\eta_0,\eta_1,\eta_2}
	\\&
	\cdot \left(\frac{1}{2}\right)^{|\gamma_0+\gamma_1+\delta_0+\delta_1+\xi_0+\xi_1+\eta_0+\eta_1|}
	(-1)^{|\xi_1+\eta_1|}
	\frac{ D_x^{\alpha+\gamma_0+\eta_0} D_{\bar{x}}^{\beta+\delta_0+\xi_0} f}{(\alpha-\gamma_1-\eta_1)!(\beta-\delta_1-\xi_1)!}  
	\\
	&
	\cdot
	\left(\frac{y-\bar{z}}{2}\right)^{\alpha-\gamma_1-\eta_1} \left(\frac{z-\bar{y}}{2}\right)^{\beta-\delta_1-\xi_1}
	D_y^{\gamma_2} D_z^{\delta_2} D_{\bar{y}}^{\xi_2}  D_{\bar{z}}^{\eta_2} \left(\chi\left(|\alpha+\beta|^{2(a-1)}4^{a-1}C_1^2\left|y-\bar{z}\right|^2\right)\right).
	\end{align*}
	Let $C_0=C_0(f)$ introduced in Definition \ref{Gevrey}. We use $C$ to denote a constant which depends on $a$, $C_1(f)$ introduced in Definition \ref{Gevrey} and the cut-off function $\chi$, which may vary from line to line. By the fact that $f\in G^a(U)$ and the previous lemma, it follows that
	\begin{align*}
	&\left| D_y^\gamma D_z^\delta  D_{\bar{y}}^{\xi}  D_{\bar{z}}^{\eta} F(f)(y,z)\right|
	\\
	&\leq 
	C_0C^{|\gamma+\delta+\xi+\eta|}
	\sum_{\substack{\gamma_0+\gamma_1+\gamma_2=\gamma\\\delta_0+\delta_1+\delta_2=\delta}}
	\sum_{\substack{\xi_0+\xi_1+\xi_2=\xi\\\eta_0+\eta_1+\eta_2=\eta}}
	\sum_{\substack{\alpha\geq \gamma_1+\eta_1\\\beta\geq \xi_1+\delta_1}}
	\binom{\gamma}{\gamma_0,\gamma_1,\gamma_2}
	\binom{\delta}{\delta_0,\delta_1,\delta_2}
	\binom{\xi}{\xi_0,\xi_1,\xi_2}
	\binom{\eta}{\eta_0,\eta_1,\eta_2}
	\\
	& \qquad\cdot  C_1^{|\alpha+\beta|}\frac{(\alpha+\gamma_0+\eta_0)!^a(\beta+\delta_0+\xi_0)!^a}{(\alpha-\gamma_1-\eta_1)!(\beta-\delta_1-\xi_1)!}\left|\frac{y-\bar{z}}{2}\right|^{|\alpha-\gamma_1-\eta_1+\beta-\delta_1-\xi_1|} 
	\\
	&\qquad \cdot
	(\gamma_2+\eta_2+\xi_2+\delta_2)!^{1+\eps}
	\left(2^{a-1}|\alpha+\beta|^{a-1}C_1\right)^{|\gamma_2+\eta_2+\xi_2+\delta_2|}.
	\end{align*}
	Using the fact $|\alpha+\beta|^{2(a-1)}4^{a-1}C_1^2\left|y-\bar{z}\right|^2\leq 1$ and Stirling's formula, it is bounded by
	\begin{align*}
	&
	C_0C^{|\gamma+\delta+\xi+\eta|}
	\sum_{\substack{\gamma_0+\gamma_1+\gamma_2=\gamma\\\delta_0+\delta_1+\delta_2=\delta}}
	\sum_{\substack{\xi_0+\xi_1+\xi_2=\xi\\\eta_0+\eta_1+\eta_2=\eta}}
	\sum_{\substack{\alpha\geq \gamma_1+\eta_1\\\beta\geq \xi_1+\delta_1}}
	\binom{\gamma}{\gamma_0,\gamma_1,\gamma_2}
	\binom{\delta}{\delta_0,\delta_1,\delta_2}
	\binom{\xi}{\xi_0,\xi_1,\xi_2}
	\binom{\eta}{\eta_0,\eta_1,\eta_2}
	\\
	&\cdot \frac{(\alpha+\gamma_0+\eta_0)!^a(\beta+\delta_0+\xi_0)!^a}{(\alpha-\gamma_1-\eta_1)!(\beta-\delta_1-\xi_1)!}\left(\frac{C_1}{2}\right)^{|\alpha+\beta|} (\gamma_2+\eta_2+\xi_2+\delta_2)!^{1+\eps}
	\\
	&\cdot
	\left(2^{a-1}|\alpha+\beta|^{a-1}C_1\right)^{|\gamma_2+\eta_2+\xi_2+\delta_2+\gamma_1+\delta_1+\xi_1+\eta_1|-|\alpha+\beta|}
	\\
	\leq&
	C_0C^{|\gamma+\delta+\xi+\eta|}
	\sum_{\substack{\gamma_0+\gamma_1+\gamma_2=\gamma\\\delta_0+\delta_1+\delta_2=\delta}}
	\sum_{\substack{\xi_0+\xi_1+\xi_2=\xi\\\eta_0+\eta_1+\eta_2=\eta}}
	\sum_{\substack{\alpha\geq \gamma_1+\eta_1\\\beta\geq \xi_1+\delta_1}}
	\binom{\gamma}{\gamma_0,\gamma_1,\gamma_2}
	\binom{\delta}{\delta_0,\delta_1,\delta_2}
	\binom{\xi}{\xi_0,\xi_1,\xi_2}
	\binom{\eta}{\eta_0,\eta_1,\eta_2}
	\\
	&\cdot
	\frac{(\alpha+\gamma_0+\eta_0)!^a(\beta+\delta_0+\xi_0)!^a}{(\alpha-\gamma_1-\eta_1)!(\beta-\delta_1-\xi_1)!}2^{-a(|\alpha+\beta|)} (\gamma_2+\eta_2+\xi_2+\delta_2)!^{1+\eps}
	\\
	&
	\cdot
	|\alpha+\beta|^{(a-1)|\gamma_2+\eta_2+\xi_2+\delta_2+\gamma_1+\delta_1+\xi_1+\eta_1|-(a-1)|\alpha+\beta|}
	\\
	\leq&
	C_0C^{|\gamma+\delta+\xi+\eta|}
	\sum_{\alpha,\beta\geq 0}\sum_{\substack{\gamma_0+\gamma_1+\gamma_2=\gamma\\\delta_0+\delta_1+\delta_2=\delta}}
	\sum_{\substack{\xi_0+\xi_1+\xi_2=\xi\\\eta_0+\eta_1+\eta_2=\eta}}
	\binom{\gamma}{\gamma_0,\gamma_1,\gamma_2}
	\binom{\delta}{\delta_0,\delta_1,\delta_2}
	\binom{\xi}{\xi_0,\xi_1,\xi_2}
	\binom{\eta}{\eta_0,\eta_1,\eta_2}
	\\&\cdot
	(\gamma_0+\eta_0)!^a(\delta_0+\xi_0)!^a(\gamma_1+\eta_1)!(\delta_1+\xi_1)! (\gamma_2+\eta_2+\xi_2+\delta_2)!^{1+\eps}
	\\&
	\cdot
	e^{-(a-1)|\alpha+\beta|}
	|\alpha+\beta|^{(a-1)(|\gamma_2+\eta_2+\xi_2+\delta_2+\gamma_1+\delta_1+\xi_1+\eta_1|+1)}
	\end{align*}
	
	For any $\alpha,\beta\geq 0$, we have
	\begin{align*}
	|\alpha+\beta|&^{(a-1)(|\gamma_1+\delta_1+\xi_1+\eta_1+\gamma_2+\delta_2+\xi_2+\eta_2|+1)}e^{-\frac{1}{2}(a-1)|\alpha+\beta|}
	\\
	&\leq 2^{(a-1)(|\gamma_1+\delta_1+\xi_1+\eta_1+\gamma_2+\delta_2+\xi_2+\eta_2|+1)}(|\gamma_1+\delta_1+\xi_1+\eta_1+\gamma_2+\delta_2+\xi_2+\eta_2|+1)!^{a-1}.
	\end{align*}
	Therefore,
	\begin{align*}
	| D_y^\gamma D_z^\delta D_{\bar{y}}^\xi D_{\bar{z}}^{\eta} F(f)(y,z)|
	\leq 2^{a-1}C_0C^{|\gamma+\delta+\xi+\eta|}(\gamma!\delta!\xi!\eta!)^{a+\eps}
	\sum_{\alpha,\beta\geq 0}
	e^{-\frac{1}{2}(a-1)|\alpha+\beta|}.
	\end{align*}
	Note that when $\xi+\eta>0$, we have \eqref{IndexRange}. So
	\begin{align*}
	| D_y^\gamma D_z^\delta D_{\bar{y}}^\xi D_{\bar{z}}^{\eta} F(f)(y,z)|
	\leq 2^{a-1}C_0C^{|\gamma+\delta+\xi+\eta|}(\gamma!\delta!\xi!\eta!)^{a+\eps}
	e^{-\frac{a-1}{8}(\sqrt{2}C_1|y-\bar{z}|)^{-\frac{1}{a-1}}}
	\sum_{\alpha,\beta\geq 0}
	e^{-\frac{1}{4}(a-1)|\alpha+\beta|}.
	\end{align*}
	The result follows as $\sum_{\alpha,\beta\geq 0}
	e^{-\frac{1}{4}(a-1)|\alpha+\beta|}\leq \left(\frac{4}{a-1}\right)^{2n}e^{\frac{(a-1)n}{2}}$.
	
	In addition, when $z=\bar{y}$, note all the derivatives of $\chi$ vanish and $|\chi|\leq 1$, whence we can replace $\eps$ by zero.
\end{proof}

\begin{proof}[Proof of Lemma \ref{AHE CLOSED}]
	
	It is easy to see that $\mathcal{A}^{a,\eps}_\theta$ is closed under summation, subtraction and differentiation. Now we consider multiplication. Take $f,g \in \mathcal{A}^{a,\eps}_{\theta}$. We will use $C_0(f),C_1(f),b(f)$ and $C_0(g),C_1(g),b(g)$ to denote the constants in \eqref{eq 4.13} corresponding to $f,g$ respectively. Take $C_1=\max\{C_1(f), C_1(g)\}$ and $b=\min\{b(f),b(g)\}$.
	Let $v'=(x,y,\theta)$. Then
	\begin{align*}
	&\left|D^{\alpha}_{v'}D^{\beta}_{\bar{v}'}(fg)(x,y,\theta(x,y,z))\right|
	\\
	\leq& \sum_{\alpha_0\leq \alpha,\beta_0\leq \beta} \binom{\alpha}{\alpha_0}\binom{\beta}{\beta_0} \left|D^{\alpha_0}_{v'}D^{\beta_0}_{\bar{v}'}f(x,y,\theta(x,y,z))\right| \left|D^{\alpha-\alpha_0}_{v}D^{\beta-\beta_0}_{\bar{v}'}g(x,y,\theta(x,y,z))\right|
	\\
	\leq &
	\sum_{\alpha_0\leq \alpha,\beta_0\leq \beta}C_0(f)C_0(g) \binom{\alpha}{\alpha_0}\binom{\beta}{\beta_0} C_1^{|\alpha+\beta|}
	(\alpha!\beta!)^{a+\eps}
	\lambda_{b,|\beta|}(x,y,z)
	\\
	= &
	C_0(f)C_0(g)(2C_1)^{|\alpha+\beta|}
	(\alpha!\beta!)^{a+\eps}
	\lambda_{b,|\beta|}(x,y,z).
	\end{align*}
	In addition, when we are restricted to $x=y=\bar{z}$, it is easy to see that we can replace $\eps$ by $0$.
	Therefore, $fg\in \mathcal{A}^{a,\eps}_{\theta}$. And we can choose $C_0(fg)=C_0(f)C_0(g)$, $C_1(fg)=2\max\{C_1(f),C_1(g)\}$ and $b(fg)=\min\{b(f),b(g)\}$.
	
	To prove $\mathcal{A}^{a,\eps}_{\theta}$ is closed under division, we will verify that if $f\in \mathcal{A}^{a,\eps}_{\theta}(U')$ and $\inf_{U'}|f|\geq C_2>0$, then the reciprocal $\frac{1}{f}\in \mathcal{A}^{a,\eps}_{\theta}(U')$. Define $h(w)=\frac{1}{w}$ for
	$w\in \mathbb{C}\setminus\{0\}$.  Then for any $v_0\in U'$, we have
	\begin{align*}
	D_{v'}^\alpha D_{\bar{v}'}^{\beta}(h\circ f)(v_0)=\alpha!\beta!\sum_{k=0}^{|\alpha+\beta|} \frac{(-1)^k}{f(v_0)^{k+1}}\sum_{\substack{\alpha_1+\alpha_2+\cdots+\alpha_k=\alpha\\\beta_1+\beta_2+\cdots+\beta_k=\beta\\\alpha_1+\beta_1>0,\cdots,\alpha_k+\beta_k>0}}\frac{D_{v'}^{\alpha_1}D_{\bar{v}'}^{\beta_1}f(v_0)}{\alpha_1!\beta_1!}\cdots \frac{D_{v'}^{\alpha_k}D_{\bar{v}'}^{\beta_k}f(v_0)}{\alpha_k!\beta_k!}.
	\end{align*}      
	We use $C_0=C_0(f), C_1=C_1(f)$ and $b=b(f)$ to denote the constants in \eqref{eq 4.13} for $f$. Without losing of generality, we can assume $C_0>1$ and $C_2<1$. Then
	\begin{align*}
	\Big|D_{v'}^\alpha D_{\bar{v}'}^{\beta}(h\circ f)(x,y,&\theta(x,y,z))\Big|
	\\
	\leq &\sum_{k=0}^{|\alpha+\beta|} \frac{1}{C_2^{k+1}}\sum_{\substack{\alpha_1+\alpha_2+\cdots+\alpha_k=\alpha\\\beta_1+\beta_2+\cdots+\beta_k=\beta\\\alpha_1+\beta_1>0,\cdots,\alpha_k+\beta_k>0}}
	(\alpha!\beta!)^{a+\eps}C_0^{k}C_1^{|\alpha+\beta|}\lambda_{b,|\beta|}(x,y,z)
	\\
	\leq &\sum_{k=0}^{|\alpha+\beta|}\binom{\alpha+k\mathbbm{1}}{k\mathbbm{1}} \binom{\beta+k\mathbbm{1}}{k\mathbbm{1}}
	\frac{1}{C_2}\left(\frac{C_0C_1}{C_2}\right)^{|\alpha+\beta|}(\alpha!\beta!)^{a+\eps}\lambda_{b,|\beta|}(x,y,z)
	\\
	\leq &\frac{1}{C_2} \left(\frac{2^{6n+2}C_0C_1}{C_2}\right)^{|\alpha+\beta|}(\alpha!\beta!)^{a+\eps}\lambda_{b,|\beta|}(x,y,z).
	\end{align*}
	In addition, when we are restricted to $x=y=\bar{z}$, it is easy to see that we can replace $\eps$ by $0$.
	Therefore, $\frac{1}{f}=h\circ f\in \mathcal{A}^{a,\eps}_{\theta}(U')$.
\end{proof}

\begin{proof}[Proof of Lemma \ref{AHE Integral}]
	Denote $v=(x,y,z)$ and $C_0=C_0(f), C_1=C_1(f), b=b(f)$. As $t\in [0,1]$, we have
	\begin{align*}
	\left|D_v^\alpha D_{\bar{v}}^{\beta} g(x,y,z)\right|
	\leq& \max_{t\in [0,1]} \left|D_v^\alpha D_{\bar{v}}^{\beta} \big(f(x,tx+(1-t)y,z)\big)\right|
	\end{align*}
	We write $D_v^\alpha=D_x^{\alpha_1}D_y^{\alpha_2}D_z^{\alpha_3}$ and $D_{\bar{v}}^\beta=D_{\bar{x}}^{\beta_1}D_{\bar{y}}^{\beta_2}D_{\bar{z}}^{\beta_3}$. Then
	\begin{align*}
	\Big|D_v^\alpha D_{\bar{v}}^{\beta} &g(x,y,z)\Big|
	\\
	\leq& \max_{t\in [0,1]} \left|\sum_{\alpha'_1\leq\alpha_1}\sum_{\alpha'_1\leq\alpha_1} \binom{\alpha_1}{\alpha_1'}\binom{\beta_1}{\beta_1'}
	D_x^{\alpha_1'}D_y^{\alpha_2+\alpha_1-\alpha_1'}D_z^{\alpha_3}
	D_{\bar{x}}^{\beta_1'}D_{\bar{y}}^{\beta_2+\beta_1-\beta_1'}D_{\bar{z}}^{\beta_3}
	f(x,tx+(1-t)y,z)\right|
	\\
	\leq&  
	\sum_{\alpha'_1\leq\alpha_1}\sum_{\alpha'_1\leq\alpha_1} \binom{\alpha_1}{\alpha_1'}\binom{\beta_1}{\beta_1'}
	C_0\left(2^{a+\eps}C_1\right)^{|\alpha+\beta|}\alpha!^{a+\eps}\beta!^{a+\eps}
	\max_{t\in [0,1]}\lambda_{b,|\beta|}(x,tx+(1-t)y,z) 
	\\
	\leq& C_0\left(2^{a+\eps+1}C_1\right)^{|\alpha+\beta|}\alpha!^{a+\eps}\beta!^{a+\eps}\lambda_{b,|\beta|}(x,y,z). 
	\end{align*}
	The last inequality follows from $|tx+(1-t)y-\bar{z}|\leq \max\{|x-\bar{z}|,|y-\bar{z}|\}$ for any $t\in [0,1]$. In addition, as $f\in \mathcal{A}^{a,\eps}_\theta$, when restricted to $x=y=\bar{z}$, we can replace $\eps$ by $0$. So we get the first part of the lemma. The second part follows by the same argument.
\end{proof}

\begin{proof}[Proof of Lemma \ref{AHE Composition}]
	Let $m=3n$. We denote $v=(x,y,z)$, $v'=(x,y,\theta)$.
	By a straightforward calculation, we have
	\begin{align*}
	\frac{D_{v'}^\alpha D_{\bar{v}'}^\beta}{\alpha!\beta!}&\widetilde{f}
	\\
	=&
	\sum_{0\leq |\xi+\eta|\leq |\alpha+\beta|}
	\frac{D_{v}^\xi D_{\bar{v}}^\eta f}{\xi!\eta!}
	\sum_{A_{\alpha\beta\xi\eta}} \frac{D_{v'}^{\alpha_{11}}D_{\bar{v}'}^{\beta_{11}}v_1}{\alpha_{11}!\beta_{11}!}
	\frac{D_{v'}^{\alpha'_{11}}D_{\bar{v}'}^{\beta'_{11}}\oo{v_1}}{\alpha'_{11}!\beta'_{11}!}
	\cdots
	\frac{D_{v'}^{\alpha_{m\xi_m}}D_{\bar{v}'}^{\beta_{m\xi_m}}v_m}{\alpha_{m\xi_m}!\beta_{m\xi_m}!}
	\frac{D_{v'}^{\alpha'_{m\eta_m}}D_{\bar{v}'}^{\beta'_{m\eta_m}}\oo{v_m}}{\alpha'_{m\eta_m}!\beta'_{m\eta_m}!},
	\end{align*}
	where $A_{\alpha\beta\xi\eta}$ is defined in \eqref{Index Set}.
	Then since $f\in \mathcal{A}^{a,\eps}_z$ and $v=(x,y,z(x,y,\theta))\in \mathcal{A}^{a,\eps}_\theta$, by taking $b=\min\{b(f),b(v)\}$, $C_0(v)=\max_{i}C_0(v_i)$ and $C_1(v)=\max_{i}C_1(v_i)$, we have
	\begin{align*}
	&\left|\frac{D_{v'}^\alpha D_{\bar{v}'}^\beta}{\alpha!\beta!}\widetilde{f}(x,y,\theta(x,y,z))\right|
	\\
	\leq& 
	\sum_{0\leq |\xi+\eta|\leq |\alpha+\beta|}
	C_0(f)C_1(f)^{|\xi+\eta|}
	\sum_{A_{\alpha\beta\xi\eta}}C_0(v)^{|\xi+\eta|}C_1(v)^{|\alpha+\beta|}\left(\xi!\eta!\prod_{i,j}\alpha_{ij}!\beta_{ij}!\prod_{ik}\alpha'_{ik}!\beta'_{ik}!\right)^{a-1+\eps}\lambda_{b,|\beta|}(x,y,z).
	\end{align*}
	
	Now we prove two combinatorial lemmas to estimate of $\xi!\eta!\prod_{i,j}\alpha_{ij}!\beta_{ij}!\prod_{i,k}\alpha'_{ik}!\beta'_{ik}!$ appearing in the above inequality.
	\begin{lemma}\label{lem 4.6}
		For any integers $k,i_1,i_2\cdots i_k\in \mathbb{Z}^+$, we have
		\begin{equation}
		k!i_1!i_2!\cdots i_k!\leq (i_1+i_2+\cdots+i_k)!.
		\end{equation}
	\end{lemma}
	\begin{proof}
		We will do induction on $k$. When $k=1$, the result follows trivially. Assume it is true for $k-1$ and we proceed to the case $k$. For simplicity, we denote $i=i_1+i_2+\cdots+i_k$. Then by using the result by induction, we have
		\begin{equation}
		k!i_1!i_2!\cdots i_k!\leq k(i_1+i_2+\cdots+i_{k-1})!i_k!=\frac{k}{\binom{i}{i_k}}i!.
		\end{equation}
		Note that $\binom{i}{i_k}\geq i\geq k$ and thus the result follows.
	\end{proof}

	\begin{lemma}\label{Combinatoric 2}
		For any multi-indices, $\alpha_1,\alpha_2,\cdots, \alpha_k\in (\mathbb{Z}^{\geq 0})^n$, if $|\alpha_i|>0$ for each $1\leq i\leq k$, then we have
		\begin{equation}
		\alpha_1!\alpha_2!\cdots\alpha_k!\leq \frac{(\alpha_1+\alpha_2+\cdots+\alpha_k)!}{k!}n^k.
		\end{equation}
	\end{lemma}
	\begin{proof}
		We denote $\alpha_i=(\alpha_{i1},\alpha_{i2},\cdots,\alpha_{in})$ and define $l_j=\#\{1\leq i\leq k, \alpha_{ij}\neq 0\}$ for $1\leq j\leq n$. Then by applying Lemma \ref{lem 4.6} to the $j$-th component of each $\alpha_i$ for $1\leq i\leq k$, we have
		\begin{align*}
		\alpha_{1j}!\alpha_{2j}!\cdots\alpha_{kj}!\leq \frac{(\alpha_{1j}+\alpha_{2j}+\cdots+\alpha_{kj})!}{l_j!}.
		\end{align*}
		Therefore,
		\begin{align}\label{eq 4.16}
		\alpha_1!\alpha_2!\cdots\alpha_k!\leq \frac{(\alpha_1+\alpha_2+\cdots+\alpha_k)!}{l_1!l_2!\cdots l_n!}.
		\end{align}
		Since $|\alpha_i|\neq 0$ for each $i$, we have $l:=l_1+l_2+\cdots l_n\geq k$. Then we can find nonnegative integers $k_1,k_2,\cdots,k_n$ such that $k_1+k_2+\cdots+k_n=k$ and $k_j\leq l_j$ for $1\leq j\leq n$. Therefore
		\begin{equation}
		l_1!l_2!\cdots l_n!\geq k_1!k_2!\cdots k_n!=\frac{k!}{\binom{k}{k_1,k_2,\cdots,k_n}}\geq \frac{k!}{n^k}.
		\end{equation}
		Plug this back into \eqref{eq 4.16}, we have the result.
	\end{proof}

	By using Lemma \ref{Combinatoric 2}, we have the upper bound for the factorials on the right hand side as
	\begin{align*}
	\xi!\eta!\prod_{i,j}\alpha_{ij}!\beta_{ij}!\prod_{ik}\alpha'_{ik}!\beta'_{ik}!
	\leq m^{|\xi+\eta|}\left(\sum_{ij}\alpha_{ij}+\beta_{ij}\right)! \left(\sum_{ik}\alpha'_{ik}+\beta'_{ik}\right)!
	\leq m^{|\xi+\eta|}(\alpha+\beta)!.
	\end{align*}
	
	Therefore, 
	\begin{align*}
	&\left|\frac{D_{v'}^\alpha D_{\bar{v}'}^\beta}{\alpha!\beta!}\widetilde{f}(x,y,\theta(x,y,z))\right|
	\\
	&\quad\leq 
	\sum_{0\leq |\xi+\eta|\leq |\alpha+\beta|}
	C_0(f)C_1(f)^{|\xi+\eta|}
	\sum_{A_{\alpha\beta\xi\eta}}C_0(v)^{|\xi+\eta|}C_1(v)^{|\alpha+\beta|}\left(m^{|\xi+\eta|}(\alpha+\beta)!\right)^{a-1+\eps}\lambda_{b,|\beta|}(x,y,z)
	\\
	&\quad\leq 
	\sum_{0\leq |\xi+\eta|\leq |\alpha+\beta|}
	C_0(f)\left(m^{a-1+\eps}C_1(f)C_0(v)C_1(v)\right)^{|\alpha+\beta|}
	(\alpha+\beta)!^{a-1+\eps}\lambda_{b,|\beta|}(x,y,z)\cdot\#A_{\alpha\beta\xi\eta}.
	\end{align*}
	Note the cardinality of $A_{\alpha\beta\xi\eta}$ has the following upper bounded $$\#A_{\alpha\beta\xi\eta}\leq \binom{\alpha+(|\xi+\eta|)\mathbbm{1}}{|\xi+\eta|\mathbbm{1}}\binom{\beta+(|\xi+\eta|)\mathbbm{1}}{|\xi+\eta|\mathbbm{1}}.$$
	So we get 
	\begin{align*}
	&\left|\frac{D_{v'}^\alpha D_{\bar{v}'}^\beta}{\alpha!\beta!}\widetilde{f}(x,y,\theta(x,y,z))\right|
	\\
	&\quad\leq 
	\sum_{0\leq |\xi+\eta|\leq |\alpha+\beta|}
	C_0(f)\Big(2^{a+\eps+2m}m^{a-1+\eps}C_1(f)C_0(v)C_1(v)\Big)^{|\alpha+\beta|}
	(\alpha!\beta!)^{a-1+\eps}\lambda_{b,|\beta|}(x,y,z)
	\\
	&\quad\leq 
	C_0(f)\Big(2^{a+3m+\eps}m^{a-1+\eps}C_1(f)C_0(v)C_1(v)\Big)^{|\alpha+\beta|}
	(\alpha!\beta!)^{a-1+\eps}\lambda_{b,|\beta|}(x,y,z).
	\end{align*}
	And if we keep track of the constant $\eps$, it is easy to see that $\eps$ comes from derivatives of since $f$ and $v=(x,y,z(x,y,\theta))$. Since $f\in \mathcal{A}^{a,\eps}_z$ and $v=(x,y,z(x,y,\theta))\in \mathcal{A}^{a,\eps}_\theta$, we can replace $\eps$ by $0$ when restricted to $x=y=\bar{z}$. 
	Therefore, we can take $C_0(\widetilde{f})=C_0(f)$, $C_1(\widetilde{f})=2^{a+3m+\eps} m^{a-1+\eps}C_0(v)C_1(f)C_1(v)$.
\end{proof}

\begin{proof}[Proof of Lemma \ref{AHE Inverse}]
	We are going to prove the following more general lemma. Note that we can assume $\psi(y,z)=yz+O(|(y,z)|^4)$ by using the Bochner coordinates at $0$. Then the Lemma \ref{AHE Inverse} follows directly by taking $F(x,y,z,\theta)=\int_0^1(D_y\psi)(tx+(1-t)y,z)dt-\theta$ and $\theta(x,y,z)=\int_0^1(D_y\psi)(tx+(1-t)y,z)dt$. 
	\begin{lemma}
		Consider smooth maps $\theta(x,y,z)=(\theta_1(x,y,z),\theta_2(x,y,z),\cdots,\theta_n(x,y,z))$  and $F(x,y,z,\theta)=(F_1(x,y,z,\theta),F_2(x,y,z,\theta),\cdots,F_n(x,y,z,\theta))$ satisfying the system of equations $F(x,y,z,\theta(x,y,z))=0$. Assume that for any $x,y,z\in U=B^n(0,1)$ and multi-indices $\alpha,\beta\geq 0$, we have
		\begin{align}\label{F}
		\left|\left(D_{(x,y,z,\theta)}^\alpha D_{(\bar{x},\bar{y},\bar{z},\bar{\theta})}^{\beta}F\right)(x,y,z,\theta(x,y,z))\right|\leq C_0C_1^{|\alpha+\beta|}\alpha!^{a+\eps}\beta!^{a+\eps}\lambda_{b,|\beta|}(x,y,z),
		\end{align}
		where $C_0=C_0(F), C_1=C_1(F)$ and $b=b(F)$ are some positive constants. And $\eps$ can be replaced by $0$, when we are restricted to $x=y=\bar{z}$. 
		
		If the $2n\times2n$ matrix 
		$\begin{pmatrix}
		\frac{\partial F}{\partial z} & \frac{\partial F}{\partial \oo{z}}\\
		\frac{\partial \oo{F}}{\partial z} & \frac{\partial \oo{F}}{\partial \oo{z}}
		\end{pmatrix}$ is the identity matrix at $(x_0,y_0,z_0,\theta_0)$, then the implicit functions $z=z(x,y,\theta)$ near $(x_0,y_0,\theta_0)$ determined by the equation $F(x,y,z,\theta)=0$ belong to $\mathcal{A}^{a,\eps}_\theta$.
	\end{lemma}	
	
	We first consider a special case when $F$ is a function in the following lemma.
	\begin{lemma}
		Consider smooth maps $\theta(x,y,z)=(\theta_1(x,y,z),\theta_2(x,y,z),\cdots,\theta_n(x,y,z))$ and function $f(x,y,z,\theta)$ such that $f(x,y,z,\theta(x,y,z))=0$. And for any $x,y,z\in U$ and multi-indices $\alpha,\beta\geq 0$, we have
		\begin{align}\label{f}
		\left|\left(D_{(x,y,z,\theta)}^\alpha D_{(\bar{x},\bar{y},\bar{z},\bar{\theta})}^{\beta}f\right)(x,y,z,\theta(x,y,z))\right|\leq C_0C_1^{|\alpha+\beta|}\alpha!^{a+\eps}\beta!^{a+\eps}\lambda_{b,|\beta|}(x,y,z),
		\end{align}
		where $C_0, C_1$ and $b$ are some positive constants. And $\eps$ can be replace by $0$, when we are restricted to $x=y=\bar z$.
		
		Assume at $(x_0,y_0,z_0, \theta_0=\theta(x_0,y_0,z_0))$, the matrix 
		$\begin{pmatrix}
		\frac{\partial f}{\partial z_n} & \frac{\partial f}{\partial \oo{z_n}}\\
		\frac{\partial \oo{f}}{\partial z_n} & \frac{\partial \oo{f}}{\partial \oo{z_n}}
		\end{pmatrix}$ is non-singular. Then the implicit function  $z_n=z_n(x,y,z_1,z_2\cdots,z_{n-1},\theta)$ determined by the equation $f(x,y,z,\theta)=0$ satisfies that for any multi-indices $\alpha,\beta\geq 0$, 
		\begin{equation}\label{zn}
		\left|\left(D_{(x,y,z',\theta)}^\alpha D_{(\bar{x},\bar{y},\bar{z}',\bar{\theta})}^{\beta}z_n\right)(x,y,z',\theta(x,y,z))\right|\leq C'_0{C'_1}^{|\alpha+\beta|}\alpha!^{a+\eps}\beta!^{a+\eps}\lambda_{b',\beta}(x,y,z),
		\end{equation}
		where $C'_0,C'_1$ and $b'$ are some positive constants and $z'=(z_1,z_2,\cdots,z_{n-1})$. In addition, when we are restricted to $x=y=\bar z$, $\eps$ can be replace by $0$.
	\end{lemma}	
	
	\begin{proof}
		For simplicity, we denote $v=(x,y,z_1,z_2\cdots,z_{n-1},\theta)$. Near some point $(v, z_n)$, we have the Taylor series of $f$ as
		\begin{equation*}
		f(v',z'_n)=\sum_{\alpha,\beta\geq 0,i,j\geq 0} a_{\alpha\bar{\beta}i\bar{j}}(v'-v)^\alpha(\oo{v'-v})^{\beta}(z'_n-z_{n})^i(\oo{z'_n-z_{n}})^j,
		\end{equation*}
		where $a_{\alpha\bar{\beta}i\bar{j}}=\frac{ D_v^{\alpha} D_{\bar{v}}^\beta D_{z_n}^i D_{{\bar{z}_n}}^jf}{\alpha!\beta!i!j!}(v,z_{n})$.
		The equation $f=0$ implies
		\begin{align}\label{eq 4.15}
		\begin{split}
		&a_{0\bar{0}0\bar{0}}+a_{0\bar{0}1\bar{0}}(z'_n-z_{n})+a_{0\bar{0}0\bar{1}}(\oo{z'_n-z_{n}})
		\\=&-\sum_{|\alpha+\beta|>0}\left(a_{\alpha\bar{\beta}0\bar{0}}+a_{\alpha\bar{\beta}1\bar{0}}(z'_n-z_{n})+a_{\alpha\bar{\beta}0\bar{1}}(\oo{z'_n-z_{n}})\right)(v'-v)^\alpha(\oo{v'-v})^{\beta}
		\\&
		-\sum_{\alpha,\beta,i+j\geq 2}a_{\alpha\bar{\beta}i\bar{j}}(v'-v)^\alpha(\oo{v'-v})^{\beta}(z'_n-z_{n})^i(\oo{z'_n-z_{n}})^j.
		\end{split}
		\end{align}
		Assume near $v$, the Taylor series of $z_n=z_n(v)$ is as follows.
		\begin{equation*}
		z'_n-z_{n}=\sum_{|\gamma+\delta|>0}b_{\gamma\bar{\delta}}(v'-v)^\gamma(\oo{v'-v})^\delta,
		\end{equation*}
		where $b_{\gamma\bar{\delta}}=\frac{ D_v^\gamma D_{\bar{v}}^\delta z_n}{\gamma!\delta!}(v)$.
		
		We define the following index sets for simplicity. 
		\begin{align*}
		A_{\alpha\beta\gamma\delta}
		=\left \{ i,j, \{ \xi_k,\eta_k \}_{1\leq k \leq i},  \{ \xi'_{l}, \eta'_{l}\}_{1\leq l \leq j}: \quad  \begin{array}{ll} \alpha+\sum_{1\leq k \leq i} \xi_{k}+\sum_{1\leq l \leq j} \xi'_{l}=\gamma, \\ \beta+\sum_{1\leq k \leq i} \eta_{k}+\sum_{1\leq l \leq j} \eta'_{l}=\delta,  \\ \xi_{k}+\xi'_{k}>0,\eta_{l}+\eta'_{l}>0,  i+j\geq 2 \end{array}  \right \}.
		\end{align*}
		\begin{align*}
		B_{\gamma\delta}=\left\{\alpha,\beta,\xi,\eta: 
		\alpha+\xi=\gamma, \beta+\eta=\delta, |\alpha+\beta|>0\right\}.
		\end{align*}
		When restrict \eqref{eq 4.15} to points $(x,y,z,\theta(x,y,z))$, $a_{0\bar{0}0\bar{0}}=0$. By comparing the coefficients of  \eqref{eq 4.15}, for any multi-indices $|\gamma+\delta|>0$, we have
		\begin{align*}
		\begin{split}
		&a_{0\bar{0}1\bar{0}}b_{\gamma\bar{\delta}}+a_{0\bar{0}0\bar{1}}\oo{b_{\delta\bar{\gamma}}}
		\\&\quad=-a_{\gamma\bar{\delta}0\bar{0}}-\sum_{B_{\gamma\delta}}a_{\alpha\bar{\beta}1\bar{0}}b_{\xi\bar{\eta}}
		-\sum_{B_{\gamma\delta}}a_{\alpha\bar{\beta}0\bar{1}}\oo{b_{\eta\bar{\xi}}}
		-\sum_{A_{\alpha\beta\gamma\delta}}a_{\alpha\bar{\beta}i\bar{j}}b_{\xi_1\bar{\eta}_1}\cdots b_{\xi_i\bar{\eta}_i}\oo{b_{\eta_1'\bar{\xi}_1'}}\cdots\oo{b_{\eta_j'\bar{\xi}_j'}}.
		\end{split}
		\end{align*}
		Taking the conjugate and switch the multi-indices $\gamma$ and $\delta$, we have
		\begin{align*}
		\begin{split}
		&\oo{a_{0\bar{0}0\bar{1}}}b_{\gamma\bar{\delta}}+\oo{a_{0\bar{0}1\bar{0}}}\oo{b_{\delta\bar{\gamma}}}
		\\&\quad =-\oo{a_{\delta\bar{\gamma}0\bar{0}}}-\sum_{B_{\delta\gamma}}\oo{a_{\alpha\bar{\beta}1\bar{0}}}\oo{b_{\xi\bar{\eta}}}
		-\sum_{B_{\delta\gamma}}\oo{a_{\alpha\bar{\beta}0\bar{1}}}b_{\eta\bar{\xi}}
		-\sum_{A_{\alpha\beta\delta\gamma}}\oo{a_{\alpha\bar{\beta}i\bar{j}}}\oo{b_{\xi_1\bar{\eta}_1}}\cdots \oo{b_{\xi_i\bar{\eta}_i}}b_{\eta_1'\bar{\xi}_1'}\cdots b_{\eta_j'\bar{\xi}_j'}.
		\end{split}
		\end{align*}
		Then for any $|\gamma+\delta|>0$, by solving $b_{\gamma\bar{\delta}}$, we obtain the following recursive formula on the coefficients  $b_{\gamma\bar{\delta}}$.
		\begin{align}\label{eq 4.19}
		\begin{split}
		&b_{\gamma\bar{\delta}}
		\\=&-\frac{\oo{a_{0\bar{0}1\bar{0}}}}{|a_{0\bar{0}1\bar{0}}|^2-|a_{0\bar{0}0\bar{1}}|^2}
		\left(a_{\gamma\bar{\delta}0\bar{0}}+\sum_{B_{\gamma\delta}}a_{\alpha\bar{\beta}1\bar{0}}b_{\xi\bar{\eta}}
		+\sum_{B_{\gamma\delta}}a_{\alpha\bar{\beta}0\bar{1}}\oo{b_{\eta\bar{\xi}}}
		+\sum_{A_{\alpha\beta\gamma\delta}}a_{\alpha\bar{\beta}i\bar{j}}b_{\xi_1\bar{\eta}_1}\cdots b_{\xi_i\bar{\eta}_i}\oo{b_{\eta_1'\bar{\xi}_1'}}\cdots \oo{b_{\eta_j'\bar{\xi}_j'}}\right)
		\\
		&+\frac{a_{0\bar{0}0\bar{1}}}{|a_{0\bar{0}1\bar{0}}|^2-|a_{0\bar{0}0\bar{1}}|^2}
		\left(\oo{a_{\delta\bar{\gamma}0\bar{0}}}+\sum_{B_{\delta\gamma}}\oo{a_{\alpha\bar{\beta}1\bar{0}}}\oo{b_{\xi\bar{\eta}}}
		+\sum_{B_{\delta\gamma}}\oo{a_{\alpha\bar{\beta}0\bar{1}}}b_{\eta\bar{\xi}}
		+\sum_{A_{\alpha\beta\delta\gamma}}\oo{a_{\alpha\bar{\beta}i\bar{j}}}\oo{b_{\xi_1\bar{\eta}_1}}\cdots \oo{b_{\xi_i\bar{\eta}_i}}b_{\eta_1'\bar{\xi}_1'}\cdots b_{\eta_j'\bar{\xi}_j'}\right).
		\end{split}
		\end{align}
		By \eqref{f}, when $\theta=\theta(x,y,z)$, the Taylor coefficients $a_{\alpha\bar{\beta}k\bar{l}}$ satisfies that
		\begin{equation*}
		|a_{\alpha\bar{\beta}k\bar{l}}|\leq C_0C_1^{|\alpha+\beta|+k+l}(\alpha!\beta!k!l!)^{a-1+\eps}\lambda_{b,|\beta|+l}(x,y,z),
		\end{equation*}
		where $\lambda$ is as defined in \eqref{Lambda}.
		We normalized $a_{\alpha\bar{\beta}k\bar{l}}$ to $\widetilde{a}_{\alpha\beta kl}$ as 
		\begin{equation}\label{atilde}
		\widetilde{a}_{\alpha\beta kl}
		=\frac{|a_{\alpha\bar{\beta}k\bar{l}}|}{(\alpha!\beta!k!l!)^{a-1+\eps}\lambda_{b,|\beta|+l}(x,y,z)},
		\end{equation}
		which is dominated by $C_0C_1^{|\alpha+\beta|+k+l}$. Similarly, we define
		\begin{equation}\label{btilde}
		\widetilde{b}_{\gamma\delta}=
		\frac{|b_{\gamma\bar{\delta}}|}{(\gamma!\delta!)^{a-1+\eps}\lambda_{b,|\beta|+l}(x,y,z)}.
		\end{equation}
		Since the matrix
		$\begin{pmatrix}
		\frac{\partial f}{\partial z_n} & \frac{\partial f}{\partial \oo{z_n}}\\
		\frac{\partial \oo{f}}{\partial z_n} & \frac{\partial \oo{f}}{\partial \oo{z_n}}
		\end{pmatrix}$ is non-singular at $(x_0,y_0,z_0,\theta_0)$, by choosing a sufficiently small neighborhood $U$ of $(x_0,y_0,z_0,\theta_{n0})$,  we have $\inf_{U}\left||a_{0\bar{0}1\bar{0}}|^2-|a_{0\bar{0}0\bar{1}}|^2\right|\geq A>0$.
		By \eqref{atilde}, \eqref{btilde} and using triangle inequalities and Lemma \ref{Combinatoric 2}, we write \eqref{eq 4.19} the following recursive inequality on $b_{\gamma\delta}$.
		\begin{align}\label{eq 4.28}
		\begin{split}
		\widetilde{b}_{\gamma\delta}
		&\leq \frac{C_0C}{A}
		\left(\widetilde{a}_{\gamma\delta00}+\sum_{B_{\gamma\delta}}\widetilde{a}_{\alpha\beta10}\widetilde{b}_{\xi\eta}
		+\sum_{B_{\gamma\delta}}\widetilde{a}_{\alpha\beta01}\widetilde{b}_{\eta\xi}
		+\sum_{A_{\alpha\beta\gamma\delta}}\widetilde{a}_{\alpha\beta ij}\widetilde{b}_{\xi_1\eta_1}\cdots \widetilde{b}_{\eta_j'\xi_j'}(6n)^{i+j}\right)
		\\
		&\quad +\frac{C_0C}{A}
		\left(\widetilde{a}_{\delta\gamma00}+\sum_{B_{\delta\gamma}}\widetilde{a}_{\alpha\beta10}\widetilde{b}_{\xi\eta}
		+\sum_{B_{\delta\gamma}}\widetilde{a}_{\alpha\beta01}\widetilde{b}_{\eta\xi}
		+\sum_{A_{\alpha\beta\delta\gamma}}\widetilde{a}_{\alpha\beta ij}\widetilde{b}_{\xi_1\eta_1}\cdots \widetilde{b}_{\eta_j'\xi_j'}(6n)^{i+j}\right).
		\end{split}
		\end{align}
		
		Recall the definition of \textit{majorant} for power series as follows.
		\begin{defn}[Majorant]
			Consider two power series in variables $x\in \mathbb{R}^n$. 
			\begin{align*}
			f(x)\sim \sum_{\alpha\geq 0} a_{\alpha}x^{\alpha}, \quad g(x)\sim \sum_{\alpha\geq 0} b_{\alpha}x^{\alpha}.
			\end{align*}
			We say that $g$ is a majorant of $f$, or $b_\alpha$ is a majorant of $a_\alpha$, if
			$|a_\alpha|\leq b_\alpha$ for any $\alpha\geq 0$. And  we denote this by $f<<g$
		\end{defn}
		For multi-indices $|\ga+\delta|>0$, we define $d_{\ga\delta}$ recursively as
		\begin{align}\label{eq 4.29}
		\begin{split}
		d_{\gamma\delta}
		&=\frac{C_0^2C_1}{A}
		\left(C_2^{|\gamma+\delta|}+\sum_{B_{\gamma\delta}}C_2^{|\alpha+\beta|+1}d_{\xi\eta}
		+\sum_{B_{\gamma\delta}}C_2^{|\alpha+\beta|+1}d_{\eta\xi}
		+\sum_{A_{\alpha\beta\gamma\delta}}C_2^{|\alpha+\beta|+i+j}d_{\xi_1\eta_1}\cdots d_{\eta_j'\xi_j'}\right)
		\\
		&\quad+\frac{C_0^2C_1}{A}
		\left(C_2^{|\delta+\gamma|}+\sum_{B_{\delta\gamma}}C_2^{|\alpha+\beta|+1}d_{\xi\eta}
		+\sum_{B_{\delta\gamma}}C_2^{|\alpha+\beta|+1}d_{\eta\xi}
		+\sum_{A_{\alpha\beta\delta\gamma}}C_2^{|\alpha+\beta|+i+j}d_{\xi_1\eta_1}\cdots d_{\eta_j'\xi_j'}\right),
		\end{split}
		\end{align}
		where $C_2=6nC_1$. Since $C_0C_2^{|\alpha+\beta|+i+j}$ is a majorant of $\widetilde{a}_{\alpha\beta ij}$,   $d_{\gamma\delta}$ defined as above is a majorant of $\widetilde{b}_{\gamma\delta}$ in \eqref{eq 4.28} for any $|\gamma+\delta|>0$. Now we will solve $d_{\gamma\delta}$ by the recursive equation \eqref{eq 4.29}.
		Formally, we define $d(u,v)=\sum_{|\gamma+\delta|>0}d_{\gamma\delta} u^{\gamma}v^{\delta}$. Then \eqref{eq 4.29} is equivalent to
		\begin{align*}
		d&(u,v)\\
		&=\frac{2C_0^2C_1}{A}\left(\frac{1}{1-C_2d(u,v)}\frac{1}{1-C_2d(v,u)}\prod_{i=1}^m\frac{1}{(1-C_2u_i)(1-C_2v_i)}-1-C_2d(u,v)-C_2d(v,u)\right),
		\end{align*}
		where $m=3n$.
		It is easy to see that $d(u,v)=d(v,u)$ and thus
		\begin{equation}
		d(u,v)=\frac{2C_0^2C_1}{A}\left(\frac{1}{(1-C_2d(u,v))^2}\prod_{i=1}^m\frac{1}{(1-C_2u_i)(1-C_2v_i)}-1-2C_2d(u,v)\right).
		\end{equation}
		Observe that $(u,v,d)=0$ satisfies the equation and there is no linear term of $d$ on the right hand side. By the Implicit Theorem for real analytic functions (See \cite{KrPa} for more details), it follows that $d(u,v)$ is real analytic near the origin. Therefore, there exists some constant $C_3$ such that $\widetilde{b}_{\gamma\delta}\leq d_{\gamma\delta}\leq C_3^{\gamma+\delta}$ for any $|\gamma+\delta|>0$. By using \eqref{btilde}, we obtain the desired bounds for $b_{\gamma\bar{\delta}}$. 
		In addition, note that the constant $\eps$ only comes from the estimate of $f$ in \eqref{f}. Therefore, when we are restricted to $x=y=\bar{z}$, the constant $\eps$ can be replaced by $0$. 
	\end{proof}

	Now we will do induction on the dimension $n$. When $n=1$, the result directly follows from the previous lemma. We assume the result holds for $n-1$ and proceed to $n$. First, we consider the equation $F_n(x,y,z,\theta)=0$. Since the matrix	
	$\begin{pmatrix}
	\frac{\partial F_n}{\partial z_n} & \frac{\partial F_n}{\partial \oo{z_n}}\\
	\frac{\partial \oo{F}_n}{\partial z_n} & \frac{\partial \oo{F}_n}{\partial \oo{z_n}}
	\end{pmatrix}$ is identity at $(x_0,y_0,z_0,\theta_0)$, by using the previous lemma again, we have the implicit function $z_n=h_n(x,y,z',\theta)$, which satisfies
	\begin{equation*}
	F_n(x,y,z',h_n,\theta)=0.
	\end{equation*}
	Take the derivative with respect to $z_j$ and $\bar{z}_j$ for $1\leq j\leq n-1$. Then at $(x_0,y_0,z_0,\theta_0)$,
	\begin{align*}
	\frac{\partial h_n}{\partial z_j}=
	\frac{\partial h_n}{\partial \bar{z}_j}=0.
	\end{align*}
	Define $G_{i}(x,y,z',\theta)=F_{i}(x,y,z',h_n(x,y,z',\theta))$ for $1\leq i\leq n-1$. Since functions $F$ and $h_n$ satisfy \eqref{F} and \eqref{zn} respectively, the composition function $G_i(x,y,z',\theta)$ for $1\leq i\leq n-1$ also satisfy the estimates on the derivatives \eqref{F} by a similar argument as in the proof of Lemma \ref{AHE Composition}. On the other hand, we have for any $1\leq i,j\leq n-1$, at $(x_0,y_0,z_0,\theta_0)$
	\begin{align*}
	\frac{\partial G_i}{\partial z_j}=\frac{\partial F_i}{\partial z_j}=\delta_{ij}, \quad \frac{\partial G_i}{\partial \bar{z}_j}=\frac{\partial F_i}{\partial \bar{z}_j}=0.
	\end{align*}
	Therefore, the matrix$\begin{pmatrix}
	\frac{\partial G}{\partial z'} & \frac{\partial G}{\partial \oo{z}'}\\
	\frac{\partial \oo{G}}{\partial z'} & \frac{\partial \oo{G}}{\partial \oo{z}'}
	\end{pmatrix}$ is identity at $(x_0,y_0,z_0,\theta_0)$.
	Using the conclusion from the induction, we have the implicit functions $z_i=h_i(x,y,\theta)$ of the equations $G_i(x,y,z',\theta)=0$ for $1\leq i\leq n-1$. It is easy to verify that $z_i=h_i(x,y,\theta)$ for $1\leq i\leq n-1$ and $z_n=h_n(x,y,h_1,h_2,\cdots, h_{n-1},\theta_n)$ satisfy all the requirements and our result follows.
\end{proof}

\section*{Acknowledgements}

The author is very grateful to Prof. Hamid Hezari for many stimulating conversations and valuable suggestions.  The author also thanks  Prof. Zhiqin Lu and Prof. Bernard Shiffman for their constant support and mentoring. 

%\bibliographystyle{plain}
%\bibliography{references}

\begin{thebibliography}{HHHH}

\bibitem[BeBeSj08]{BBS}
	Berman, R., Berndtsson, B., Sj\"ostrand, J.,
	\emph{A direct approach to Bergman kernel asymptotics for positive line bundles}.
	Ark. Mat. \textbf{46(2)}, 197--217 (2008). 

\bibitem[Be03]{Bern} Berndtsson, B. \emph{Bergman kernels related to Hermitian line bundles over compact complex manifolds}, Explorations in complex and Riemannian geometry,
Contemp. Math., \textbf{332},  1--17, Amer. Math. Soc., Providence, RI, 2003. 

\bibitem[Ber08]{Ber}
Berman, R.,
\emph{Sharp asymptotics for {T}oeplitz determinants and convergence
	towards the {G}aussian free field on {R}iemann surfaces}.
Int. Math. Res. Not. \textbf{22}, 5031--5062 (2012).

\bibitem[Be10]{B}
	Berndtsson, B.
	\emph{An introduction to things $\dbar$. Analytic and Algebraic Geometry}, McNeal, 7--76 (2010).

\bibitem[BlShZe00]{BSZ} Bleher, P., Shiffman, B., Zelditch, S., \emph{Universality and scaling of correlations between zeros on complex manifolds}.
Invent. Math. \textbf{142} (2), 351--395 (2000).

\bibitem[Bo47]{Boc} Bochner, S.
\emph{Curvature in {H}ermitian metric}. Bull. Amer. Math. Soc. 179--195 (1947). 

\bibitem[BoSj75]{BoSj} Boutet de Monvel, L.,  Sj\"ostrand, J.
\emph{Sur la singularit\'e des noyaux de Bergman et de Szeg\"o}. \'Equations aux D\'eriv\'ees Partielles de Rennes, Asterisque \textbf{34-35}, 123--164 (1976), Soc. Math. France, Paris.

%\bibitem[Bo86]{Bo} Boutet de Monvel, L.,
%\emph{Compl\'ement sur le noyau de Bergman}. S\'eminaire sur les \'equations aux d\'eriv\'ees partielles, Exp. No. \textbf{XX}, 13 pp., \'Ecole Polytech., Palaiseau (1986).

\bibitem[Cal53]{Cal} Calabi, E.,
\emph{Isometric imbedding of complex manifolds.}
Ann. of Math. \textbf{58} (2), 1--23 (1953).

\bibitem[Ca99]{Ca}
	Catlin, D.,
	\emph{The Bergman kernel and a theorem of Tian. Analysis and Geometry in Several Complex Variables}, Katata,
	Trends Math., 1--23. Birkh\"auser, Boston (1999).

\bibitem[Ch91]{Ch1} Christ, M. \emph{On the $\bar \partial$ equation in weighted $L^2$ norms in $\mathbb C^1$.} J. Geom. Anal., \textbf{1}(3), 193--230 (1991).

\bibitem[Ch03]{Ch} Christ, M.,
\emph{Slow off-diagonal decay for Szeg\"o kernels associated to smooth Hermitian line bundles.} Harmonic analysis at Mount Holyoke (South Hadley, MA, 2001), 77--89,
Contemp. Math., \textbf{320}, Amer. Math. Soc., Providence, RI, 2003.

\bibitem[Ch13a]{Ch2} Christ, M., \emph{Upper bounds for Bergman kernels associated to positive line bundles with smooth Hermitian metrics}, unpublished (2013), arXiv:1308.0062. 

\bibitem[Ch13b]{Ch3} Christ, M., \emph{Off-diagonal decay of Bergman kernels: On a conjecture of Zelditch}, unpublished (2013), arXiv:1308.5644. 

\bibitem[CC05]{CC}
Chung, S.-Y. and Chung, J.,
\emph{There exist no gaps between {G}evrey differentiable and
	nowhere {G}evrey differentiable}.
Proc. Amer. Math. Soc. \textbf{133}, 859--863 (2005). 

\bibitem[DaLiMa06]{DLM}  Dai, X., Liu, K., Ma, X., \emph{On the asymptotic expansion of {B}ergman kernel.} J. Differential Geom. \textbf{72}, 1--41(2006).

\bibitem[Do01]{Do} Donaldson S. K. \emph{Scalar curvature and projective embeddings. I}, J. Differential Geom. \textbf{59}(3), 479--522 (2001).

\bibitem[De98]{Del} Delin, H. \emph{Pointwise estimates for the weighted Bergman projection kernel in $\mathbb C^n$, using a weighted $L^2$ estimate for the $\bar \partial$ equation.} Ann. Inst. Fourier (Grenoble), \textbf{48}(4), 967--997 (1998).

\bibitem[En00]{En}  Engli\v s, M., \emph{The asymptotics of a Laplace integral on a K\"ahler manifold.} Trans. Amer. Math. Soc. \textbf{528}, 1--39 (2000).

\bibitem[Fe74]{Fe}
	Fefferman, C.,
	\emph{The Bergman kernel and biholomorphic mappings of psuedoconvex domains.}
	Invent. Math. \textbf{26}, 1--66 (1974).

\bibitem[Ge18]{Ge}
Gevrey, M.,
\emph{Sur la nature analytique des solutions des \'equations aux
	d\'eriv\'ees partielles.}
{P}remier m\'emoire. (French) Ann. Sci. \'Ecole Norm. Sup. \textbf{35}(3), 129--190 (1918). 

\bibitem[HeKeSeXu16]{HKSX} Hezari, H., Kelleher, C., Seto, S., Xu, H., \emph{Asymptotic expansion of the Bergman kernel via perturbation of the Bargmann-Fock model,} Journal of Geometric Analysis, \textbf{26}(4), 2602--2638 (2016). 

\bibitem[HeLuXu18]{HLXanalytic}
Hezari, H., Lu, Z., and Xu, H.,
\emph{Off-diagonal asymptotic properties of bergman kernels associated to
	analytic K\"ahler potentials},
\newblock {\em International Mathematics Research Notices} (2018).

\bibitem[HeXu18]{HX} Hezari, H., Xu, H., \emph{Quantitative upper bounds for Bergman kernels associated to smooth \k potentials,} arXiv:1807.00204 (2018). 

\bibitem[Ho66]{Ho} H\"ormander, L., \emph{An introduction to complex analysis in several variables.} D. Van Nostrand Co., Inc., Princeton, N.J.-Toronto, Ont.-London 1966.

\bibitem[Ho68]{HoAH} H\"ormander, L., \emph{Lecture notes at the Nordic Sumer School of Mathematics}, 1968.

\bibitem[Ju97]{Ju1} Jung, K., \emph{The Adiabatic Theorem for Switching Processes with Gevrey Class Regularity.} A note in English from author's thesis titled: Adiabatik und Semiklassik bei Regularit\"at vom Gevrey-Typ, Berlin, Techn. Universi\"tat, Diss., 1997.

\bibitem[Ju00]{Ju2} Jung, K.,
\emph{Phase space tunneling for operators with symbols in a Gevrey class. }
J. Math. Phys. \textbf{41}(7), 4478--4496 (2000).

\bibitem[KaSc01]{KS} Karabegov, A., Schlichenmaier, M. \emph{Identification of Berezin-Toeplitz quantization}, J. Reine Angew. Math. \textbf{540}, 49--76 (2001).

\bibitem[KP02]{KrPa} Krantz, Steven G.; Parks, Harold R.,  \emph{A primer of real analytic functions, Second edition.} Birkh\"auser Advanced Texts: Basler Lehrb\"ucher. [Birkh\"auser Advanced Texts: Basel Textbooks] Birkh\"auser Boston, Inc., Boston, MA, 2002. 

%\bibitem[KoNo96]{KoNo}
	%Kobayashi, S., Nomizu, K.
 	%\emph{Foundations of differential geometry. {V}ol.{II}},
	%Wiley Classics Library. John Wiley \& Sons, Inc., New York, 1996.

\bibitem[Lin01]{Lindholm} Lindholm, N. \emph{Sampling in weighted $L^p$ spaces of entire functions in $\mathbb C^n$ and estimates of the
Bergman kernel.} J. Funct. Anal. \textbf{182}(2), 390--426 (2001).

\bibitem[Liu10]{Liu}
	Liu, C.-J.,
	\emph{The asymptotic {T}ian-{Y}au-{Z}elditch expansion on {R}iemann surfaces with constant curvature},
	Taiwanese J. Math., 1665--1675 (2010). 

\bibitem[LiuLu15]{LiuLu1}
	Liu, C.-J., Lu, Z.,
	\emph{Uniform asymptotic expansion on {R}iemann surfaces}. Analysis, complex geometry, and mathematical physics: in honor of {D}uong {H}. {P}hong,
	Contemp. Math. \textbf{644}, 159--173 (2015). 
	
\bibitem[LiuLu16]{LiuLu2}
	Liu, C.-J.,  Lu, Z.,
	\emph{Abstract {B}ergman kernel expansion and its applications},
	Trans. Amer. Math. Soc. \textbf{368}, 1467--1495 (2016). 
 
\bibitem[Lo04]{Loi} Loi, A.,  \emph{The Tian-Yau-Zelditch asymptotic expansion for real analytic K\"ahler metrics.} Int. J. Geom. Methods Mod. Phys. \textbf{1}(3), 253--263 (2004).

\bibitem[Lu00]{Lu}
	Lu, Z.,
 	\emph{On the Lower Order Terms of the Asymptotic Expansion of Tian-Yau-Zelditch},
	American Journal of Mathematics, \textbf{122}(2), 235--273 (2000). 

\bibitem[LuSe17]{LS} Lu, Z.,  Seto S. \emph{Agmon type estimates of the Bergman Kernel for non-compact manifolds}, preprint, 2017.

\bibitem[LuSh15]{LuSh}
	Lu, Z., Shiffman, B.,
	\emph{Asymptotic Expansion of the Off-Diagonal Bergman Kernel on Compact \k Manifolds},
	Journal of Geometric Analysis, \textbf{25}(2), 761--782 (2015). 

\bibitem[LuTi04]{LuTian}
	Lu, Z., Tian, G.,
 	\emph{The log term of the {S}zeg\"o kernel},
	Duke Math. J., \textbf{125}(2), 351--387 (2004).
	
\bibitem[LuZe16]{LuZe}
	Lu, Z., Zelditch, S.,
 	\emph{Szeg\"o kernels and {P}oincar\'e series},
	J. Anal. Math., \textbf{130}, 167--184 (2016).	

\bibitem[MaMa07]{MaMaBook} Ma, X. and  Marinescu, G., \emph{ Holomorphic Morse inequalities and Bergman kernels}, Progress in Math., \textbf{254}, Birkh\"auser, Basel, 2007.

\bibitem[MaMa08]{MaMa}
	Ma, X., Marinescu, G.,
 	\emph{Generalized {B}ergman kernels on symplectic manifolds},
	Adv. Math., \textbf{217}(4), 1756-1815 (2008).


\bibitem [MaMa13]{MaMaOff} Ma, X., Marinescu, G., \emph{Remark on the Off-Diagonal Expansion of the Bergman Kernel on Compact K\"ahler Manifolds}, Communications in Mathematics and Statistics, \textbf{1}(1), 37--41 (2013).


\bibitem[MaMa15]{MaMaAgmon} Ma, X., Marinescu, G., \emph{Exponential estimate for the asymptotics of Bergman kernels.}
Math. Ann. \textbf{362}(3-4), 1327--1347 (2015). 

\bibitem[RoSjNg18]{RSN} Rouby, O., Sj\"ostrand, J., Ngoc, S.,\emph{Analytic Bergman operators in the semiclassical limit,} arXiv:1808.00199v1 (2018). 

\bibitem[Se15]{Seto} Seto, S. \emph{On the asymptotic expansion of the Bergman kernel}, Thesis (Ph.D.)-University of California, Irvine. (2015).


\bibitem[ShZe02]{ShZe}
 Shiffman, B., Zelditch, S.,
 \emph{Asymptotics of almost holomorphic sections of ample line bundles on symplectic manifolds}, J. Reine Angew. Math. \textbf{544}, 181-222 (2002). 

\bibitem[Sj82]{Sj} Sj\"ostrand, Singularit\'es analytiques microlocales, Ast\'erisque,  \textbf{95}(1982), 1--166, Soc. Math. France, Paris.

\bibitem[Ti90]{Ti}
	Tian, G.,
	\emph{On a set of polarized \k metrics on algebraic manifolds}, J. Differ. Geom. \textbf{32}(1), 99--130 (1990).

\bibitem[Xu12]{Xu} Xu, H.,
\emph{A closed formula for the asymptotic expansion of the Bergman kernel.}
Comm. Math. Phys. \textbf{314}(3), 555--585 (2012).

\bibitem[YuZh16]{YZ} Yuan, Y., Zhu, J.
\emph{Holomorphic line bundles over a tower of coverings.}
J. Geom. Anal. \textbf{26}(3), 2013--2039 (2016).

\bibitem[Ze98]{Ze1}
	Zelditch, S., \emph{Szeg\"{o} kernels and a theorem of Tian}, Internat. Math. Res. Notices \textbf{6}, 317--331 (1998). 

\bibitem[Ze09]{ZeBookReview} Zelditch, S. \emph{Book review of "Holomorphic Morse inequalities and Bergman kernels" (by Xiaonan Ma and George Marinescu)}. Bulletin of the American Mathematical Society \textbf{46}, 349--361 (2009).

\bibitem[Ze12]{Ze2}
	Zelditch, S.,
 	\emph{Pluri-potential theory on {G}rauert tubes of real analytic {R}iemannian manifolds, {I}}. Spectral geometry, Proc. Sympos. Pure Math. \textbf{84}, 299--339 (2012). 
	
\bibitem[Ze16]{Ze3}  Zelditch, S.,  \emph{Off-diagonal decay of toric Bergman kernels}, Lett. Math. Phys. \textbf{106}(12), 1849--1864 (2016). Volume in Memory of Louis Boutet de Monvel. 



\end{thebibliography}

\end{document}